\documentclass[12pt]{amsart}
\usepackage{times,amsfonts,amsmath,amstext,amsbsy,amssymb,amsopn,amsthm,upref,eucal,amscd}
\usepackage[T1]{fontenc}
\usepackage{color}
\numberwithin{equation}{section}

\newtheorem{theorem}{Theorem}
\newtheorem{Theorem}[theorem]{Theorem}

\newtheorem{lemma}{Lemma}[section]
\newtheorem{Lemma}[lemma]{Lemma}
\newtheorem*{NNLemma}{Lemma}
\newtheorem{corollary}{Corollary}[section]
\newtheorem{Corollary}[corollary]{Corollary}
\newtheorem{proposition}{Proposition}
\newtheorem{Proposition}[proposition]{Proposition}
\newtheorem{problem}{Problem}
\newtheorem{Problem}[problem]{Problem}

\theoremstyle{definition}

\theoremstyle{remark}

\newtheorem{remark}{Remark}[section]
\newtheorem{Remark}[remark]{Remark}
\newtheorem*{NNRemark}{Remark}

\renewcommand{\Im}{\operatorname{Im}}
\renewcommand{\Re}{\operatorname{Re}}
\renewcommand{\epsilon}{\varepsilon}
\renewcommand{\phi}{\varphi}

\newcommand{\field}[1]{\mathbb{#1}}
\newcommand{\Proj}{\field{P}}
\newcommand{\R}{\field{R}}
\newcommand{\N}{\field{N}}
\newcommand{\C}{\field{C}}
\newcommand{\Z}{\field{Z}}

\newcommand{\Cal}{\mathcal}

\newcommand{\cI}{\mathcal I}
\newcommand{\M}{\mathcal M}

\renewcommand{\>}{{\rangle}}
\newcommand{\pref}[1]{(\ref{#1})}

\newcommand{\cB}{{\mathcal B}}

\newcommand{\cH}{{\mathcal H}}
\newcommand{\cK}{{\mathcal K}}
\newcommand{\cL}{{\mathcal L}}
\newcommand{\cM}{{\mathcal M}}

\newcommand{\cQ}{{\mathcal Q}}
\newcommand{\cR}{{\mathcal R}}

\newcommand{\modN}{\ (\textrm{ mod } N)}
\newcommand{\MNa}{M_N(a_1,a_2,a_3,a_4)}

\newcommand{\SL}{\operatorname{SL}(2,{\mathbb R})}
\newcommand{\SO}{\operatorname{SO}(2,{\mathbb R})}

\newcommand{\GL}{\operatorname{GL}(2,{\mathbb R})}

\newcommand{\CP}{{\mathbb C}\!\operatorname{P}^1}

\newcommand{\rk}{\operatorname{rank}}
\newcommand{\Ann}{\operatorname{Ann}}
\newcommand{\Tr}{\operatorname{Tr}}

\newcommand{\const}{\mathit{const}}

\newcommand{\PcH}{\operatorname{P}\!{\mathcal H}}

\newlength{\halfbls}

\begin{document}
\title[Equivariant subbundles of Hodge bundle]{Lyapunov spectrum of invariant subbundles of the Hodge bundle}
\author{Giovanni Forni}
\address{Giovanni Forni: Department of Mathematics, University of Maryland, College Park, MD 20742-4015, USA}
\email{gforni@math.umd.edu.}
\author{Carlos Matheus}
\address{Carlos Matheus: CNRS, LAGA, Institut Galil\'ee, Universit\'e Paris 13, 99, avenue Jean-Baptiste Cl\'ement, 93430, Villetaneuse, France}
\email{matheus@impa.br, matheus@math.univ-paris13.fr.}
\author{Anton Zorich}
\address{Anton Zorich: IRMAR, Universit\'e de Rennes 1, Campus de Beaulieu, 35042, Rennes, France}
\email{anton.zorich@univ-rennes1.fr.}

\date{\today}

\begin{abstract}
%A cyclic cover of the projective plane branched at four points has a
%natural structure of a square-tiled surface. We describe
%the combinatorics of such a square-tiled surface, the geometry of the
%corresponding Teichm\"uller curve, and compute the Lyapunov exponents
%of the determinant bundle over the Teichm\"uller curve with respect
%to the geodesic flow.

%We find a new example of a Teichm\"uller curve with a maximally
%degenerate Lyapunov spectrum (the only known example
%found previously by G.~Forni also corresponds to a cyclic cover).
%Presumably, these two examples cover all possible Teichm\"uller
%curves with maximally degenerate Lyapunov spectrum.

We  study  the Lyapunov spectrum of the Kontsevich--Zorich cocycle on
$\SL$-invariant  subbundles of the Hodge bundle over the support of
$\SL$-invariant  probability  measures  on the moduli space of Abelian
differentials.

In  particular,  we  prove  formulas  for  partial  sums  of Lyapunov
exponents   in   terms   of   the   second   fundamental   form   (the
Kodaira--Spencer   map)   of   the   Hodge  bundle  with  respect  to
Gauss--Manin  connection  and  investigate  the relations between the
central  \mbox{Oseledets}  subbundle  and  the  kernel  of  the second
fundamental form. We illustrate our conclusions in two special cases.
\end{abstract}

\maketitle

\setcounter{tocdepth}{2}
\tableofcontents

%\newpage

\section{Introduction}
Consider  a  billiard  on the plane with $\mathbb{Z}^{2}$-periodic rectangular
obstacles as in Figure~\ref{fig:windtree}.

\begin{figure}[htb]
\includegraphics{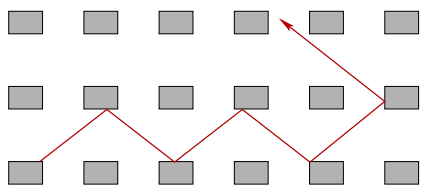}
\vspace{70bp}
\caption{
\label{fig:windtree}
Billiard in the plane with periodic rectangular obstacles.
}
\end{figure}

In~\cite{DHL}, it is shown  that  for  almost  all  parameters  $(a,b)$ of  the
obstacle (i.e., lenghts $0<a,b<1$ of the sides of the rectangular obstacles),  for  almost  all  initial directions $\theta$, and for any starting
point $x$ the billiard trajectory $\{\phi^{\theta}_t(x)\}_{t\in\mathbb{R}}$ escapes
to infinity with a rate $t^{2/3}$
(unless it hits the corner):
$$
\limsup_{t\to+\infty}
\frac{\log(\text{distance between }x \textrm{ and } \phi^{\theta}_t(x))}
{\log t} = \frac{2}{3}\,.
$$

Note  that  changing  the  height  and the width of the obstacle (see
Figure~\ref{fig:different:billiards})    we   get   quite   different
billiards, but this does not change the diffusion rate.

\begin{figure}[htb]
\includegraphics{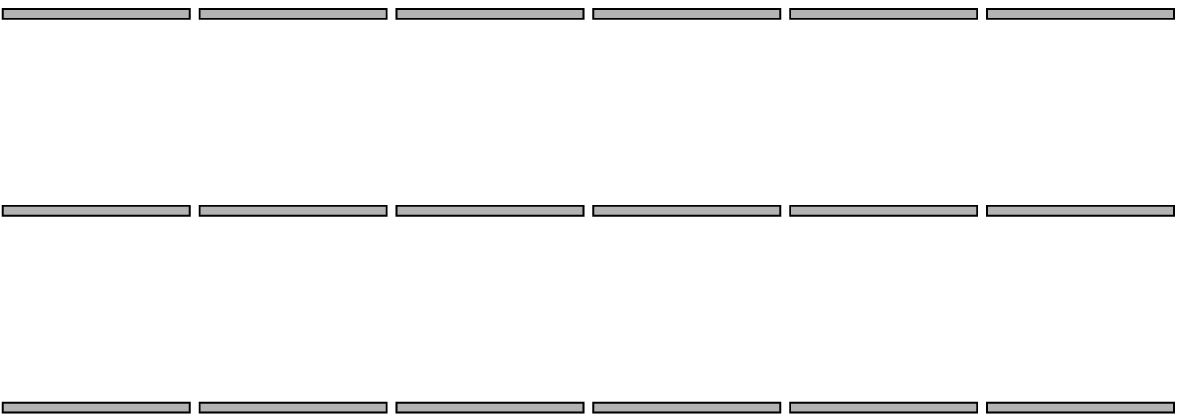}
\includegraphics{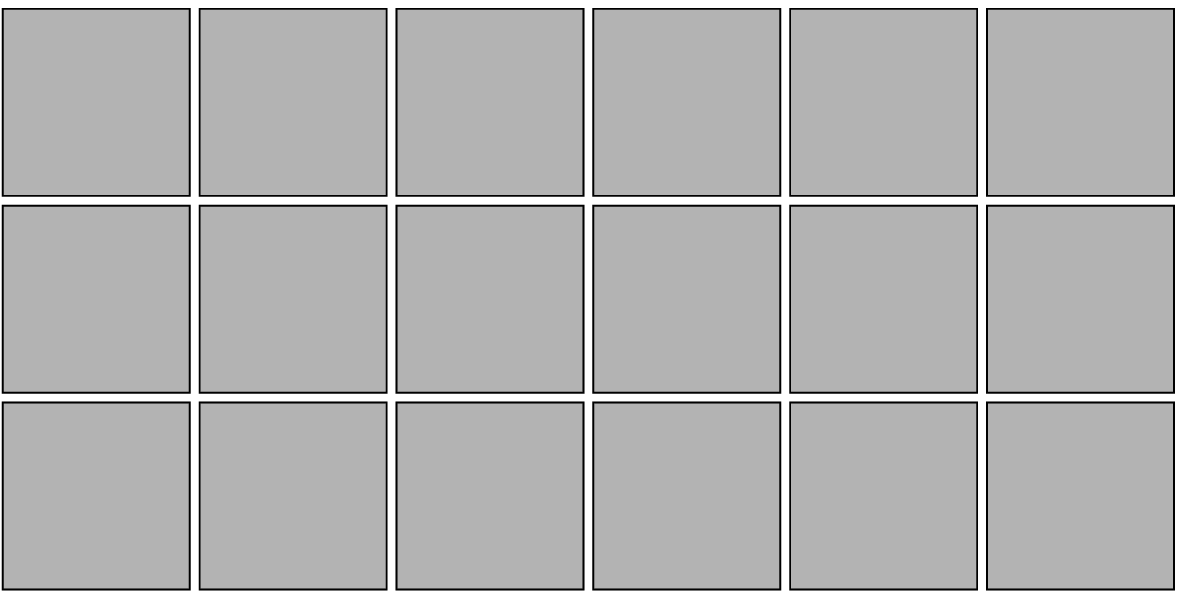}
\vspace{75bp}
\caption{
\label{fig:different:billiards}
The escape rate does not depend on the size of the obstacles!}
\end{figure}

The number ``$\frac{2}{3}$'' here is the Lyapunov exponent of a certain
renormalizing  dynamical  system  associated to the initial one. More
precisely,  it  is  the Lyapunov exponent of certain subbundle of the
Hodge bundle under the Kontsevich--Zorich cocycle.

The  Lyapunov exponents of the Hodge bundle also govern the deviation
spectrum   for   the  asymptotic  cycle  of  an  orientable  measured
foliation, as well as the rate of convergence of ergodic averages for
interval  exchange  transformations  and  for certain area preserving
flows  on  surfaces see~\cite{Forni2}, \cite{Zorich:asymptotic:flag},
\cite{Zorich:how:do}.  The  range  of  phenomena  where the Lyapunov
exponents of the Hodge bundle are extremely helpful keeps growing: nowadays, it
includes, in  particular, the evaluation of volumes of the moduli spaces
of quadratic differentials on $\CP$, see~\cite{Athreya:Eskin:Zorich},
and  the  classification of commensurability classes of all presently
known non-arithmetic ball quotients~\cite{Kappes:Moeller}.

In  this  paper  we  develop  the  study  of the Lyapunov spectrum of
invariant subbundles of the Hodge bundle under the Kontsevich--Zorich
cocycle  (with  respect  to  general  $\SL$-invariant  measures).  We
revisit  variational  formulas  of G.~Forni from~\cite{Forni2} for the
Hodge   norm  interpreting  them  in  more  geometric  terms  and  we
generalize  them  to  invariant  subbundles.  We  generalize  Forni's
formulas   for   partial  sums  of  the  Lyapunov  exponents  of  the
Kontsevich--Zorich  cocycle  in terms of geometric characteristics of
the      Hodge      bundle      (Theorem~\ref{thm:partialsum}      in
\S~\ref{ssec:formulas}).  We  establish the reduciblity of the second
fundamental  form  with  respect  to  any  decomposition  into  Hodge
star-invariant,   Hodge   orthogonal  subbundles  and  generalize  the
Kontsevich   formula  for  the  sum  of  all  non-negative  exponents
(Corollary~\ref{cor:KZformgen}     in     \S~\ref{ssec:reduc}).    We
investigate  the  occurrence  of  zero  exponents  with  a particular
emphasis  on the relation between the central Oseledets subbundle and
the  kernel  of the second fundamental form. Our main theorem in this
direction (Theorem~\ref{th:Ann} in \S~\ref{ssec:central}) establishes
sufficient  conditions for the inclusion of one into the other, hence
for their equality.

We illustrate our conclusions with two examples. The first model case,  inspired by recent work of  A.~Eskin, M.~Kontsevich and A.~Zorich \cite{Eskin:Kontsevich:Zorich:cyclic}, is given by  arithmetic
Teichm\"uller curves of \textit{square-tiled cyclic covers} (introduced in \cite{Forni:Matheus:Zorich1});
the second model case is a certain $\SL$-invariant locus $\mathcal{Z}$ (inspired by a paper of C.~McMullen~\cite{McMullen}) supporting an $\SL$-invariant ergodic probability measure with some
zero exponents.

The  study  of  square-tiled  cyclic  covers  was  motivated  by  the
discovery  of  two  arithmetic  Teichm\"uller  curves of square-tiled
cyclic  covers  with maximally degenerate Kontsevich--Zorich spectrum
(see        \cite{ForniSurvey},        \cite{Forni:Matheus}       and
\cite{Forni:Matheus:Zorich1}).   Conjecturally  there  are  no  other
$\SL$-invariant   probability   measures  with  maximally  degenerate
spectrum. Progress on this conjecture has been made by M.~M\"oller~\cite{Moeller} in
the    case    of    Teichm\"uller    curves    and,   recently,   by
D.~Aulicino~\cite{Aulicino} in  the general case. A conditional proof
of the conjecture in sufficiently high genera can also be derived from
a quite explicit   formula  for   the  sum  of  all  non-negative  exponents
(see~\cite{Eskin:Kontsevich:Zorich}).

The  case  of square-tiled cyclic covers is especially rigid. In this
case,  the  central  Oseledets  subbundle  is always $\SL$-invariant,
smooth   and  in  fact  coincides  with  the  kernel  of  the  second
fundamental form. This pictures does not hold in general. In fact, in
our  second  model  case  the  central  Oseledets  subbundle does not
coincide  with  the kernel of the second fundamental form even though
it has the same rank. We emphasize that the rank of the kernel of the
second  fundamental  form  at  generic points can be explained in all
examples  by  symmetries  (automorphisms) of the underlying surfaces.
Indeed, the discovery of the two above-mentioned maximally degenerate
examples  was  based on a symmetry criterion for the vanishing of the
second  fundamental  form  on  the  complement  of  the  tautological
subbundle    (see    \cite{ForniSurvey},   \cite{Forni:Matheus}   and
\S~\ref{symmetry} of this paper).

This  article  is  organized  as  follows.  In \S 2, we introduce the
Kontsevich--Zorich  cocycle on the Hodge bundle over the moduli space
of  Abelian  differentials  and  we  compute  the  relevant geometric
tensors  of  the  bundle,  endowed  with the Hodge Hermitian product,
namely,  the  second  fundamental  form  and  the  curvature  of  the
Hermitian  connection  with respect to the Gauss-Manin connection. We
then  prove  first and second variational formulas for the Hodge norm
in  terms  of the second fundamental form and of the curvature. In \S
3,  we  derive  formulas  for  partial  sums  of  the \mbox{Lyapunov}
exponents of the restriction of the Kontsevich--Zorich cocycle to any
$\SL$-invariant subbundle of the Hodge bundle. In \S 4 we investigate
the  presence  of zero exponents and we prove results on the relation
between the central Oseledets subbundle of the cocycle and the kernel
of        the        second        fundamental        form.        In
Appendix~\ref{s:rk:B:cyclic:covers}   we   describe   the   case   of
(arithmetic)   \mbox{Teichm\"uller}  curves  of  square-tiled  cyclic
covers.  Finally,  in  Appendix \ref{s:rk:B:Z}, we present our second
model  case.  Conjecturally, this second example is representative of
the  general  features  related  to the presence of zero exponents on
invariant subbundles of the Hodge bundle.

\section{The Hodge bundle}

%-----------------------------------------------------------------------

\subsection{The Kontsevich--Zorich cocycle}

The  moduli  space of Abelian differentials $\cH_g$ has the structure
of a complex vector bundle over the moduli space $\cM_g$ of Riemann
surfaces  of genus $g$. The fiber over a point of $\cM_g$ represented
by  a  Riemann surface $S$ corresponds to the complex $g$-dimensional
vector space of all holomorphic 1-forms $\omega$ on $S$.

The  space  $\cH_g$  admits a natural action of the group $\GL$ (see,
for instance,~\cite{Masur:Tabachnikov} or~\cite{Zorich:Houches}    for   an
elementary  description  of  this  action). It is well-known that the
orbits of the diagonal subgroup
$
\left(
\begin{array}{cc}
e^t & 0 \\
0 &  e^{-t}
\end{array}\right)
$
of  $\GL$  project on $\cM_g$  to  the  geodesics  with respect to the Teichm\"uller
metric. For this reason the flow on $\cH_g$ defined by the action of  the diagonal
subgroup is called the \textit{Teichm\"uller geodesic flow}.

The    real   (respectively  complex)   Hodge   bundle   $H^1_\R$
(respectively $H^1_\C$)  over  the  moduli  space $\cM_g$ is the
vector    bundle    having    the    first   cohomology   $H^1(S,\R)$
(respectively  $H^1(S,\C)$) as its fiber over a point represented by
the   Riemann   surface  $S$.  By identifying the  lattices  $H^1(S,\Z)$
(respectively $H^1(S,\Z\oplus  i\Z)$) in  the fibers of these vector bundles
it is possible to canonically  identify  fibers  over  nearby  Riemann  surfaces.
This identification is called the \textit{Gauss--Manin} connection.

Let us consider now the pullback of the Hodge bundle to $\cH_g$ with respect
to    the    natural    projection    $\cH_g\to\cM_g$.  We can lift the Teichm\"uller
geodesic flow to the Hodge bundle by parallel transport
of cohomology classes with respect to the \textit{Gauss--Manin}  connection, thus   getting   a   cocycle   $G_t^{KZ}$
called the \textit{Kontsevich-Zorich cocycle}.

 The  Lyapunov exponents of this cocycle were proved
to  be  responsible  for  some  fine dynamical properties of flows on
individual  Riemann surfaces,  see~\cite{Forni2}, \cite{Zorich:how:do},
which  motivated  the  study  of  these  exponents for all known
Teichm\"uller flow-invariant (and especially  $\SL$-invariant) ergodic
measures on $\cH_g$.

In  order  to  analyze  the  Lyapunov  spectrum  (i.e., collection of
Lyapunov   exponents)  of  $G_t^{KZ}$,  we  need  to  understand  the
evolution  of  cohomology  classes  $[c]\in  H^1(S,\mathbb{R})$ under
$G_t^{KZ}$.   A  particularly  useful  tool  for  this  task  is  the
\textit{Hodge norm}, the main object of the next subsection.
%-----------------------------------------------------------------------

\subsection{The Hodge product}
%\subsection{Variational formulas and Lyapunov exponents of the Hodge bundle}
%\label{sec:varfor}
   %
 The  natural  pseudo-Hermitian intersection  form  on  the  complex  cohomology
 $H^1(S,\C{})$ of a Riemann  surface $S$ can  be  defined  on  any pair
 $(\omega_1, \omega_2)$ of complex-valued closed
 $1$-forms on $S$ representing cohomology classes in  $H^1(S,\C{})$ as
\begin{equation}
\label{eq:Intform}
(\omega_1,\omega_2):=
\frac{i}{2}\int_S\omega_1\wedge\bar\omega_2\ .
\end{equation}
Restricted to the subspace $H^{1,0}(S)$ of holomorphic $1$-forms, the
intersection   form   induces  a  positive-definite  Hermitian  form;
restricted   to   the   subspace   $H^{0,1}(S)$  of  anti-holomorphic
$1$-forms,  it  induces a negative-definite Hermitian form, so on the
entire      complex     cohomology     space  the   pseudo-Hermitian
form~\eqref{eq:Intform} has signature $(g,g)$.

For later use, we define (with the aid of the Hodge norm)
\begin{equation}
\label{eq:Ab1}
\cH_g^{(1)}:=\{\omega\in\cH_g: \|\omega\|^2=(\omega,\omega)=1\},
\end{equation}
that is, $\cH_g^{(1)}$ is the moduli space of \emph{unit area} Abelian differentials on Riemann surfaces of genus $g$.

The \textit{Hodge  representation  theorem}  affirms  that  for  any
cohomology   class   $c$   in   the  \textit{real  }cohomology  space
$H^1(S,\R)$   of   a  Riemann  surface  $S$  there  exists  a  unique
holomorphic  form $h(c)$ such that $c$ is the cohomology class of the
closed  1-form  $\Re  h(c)$  in  $H^1_{\mathit{de  Rham}}(S,\R)\simeq
H^1(S,\R)$.

By  the Hodge representation theorem, the positive-definite Hermitian
form~\eqref{eq:Intform}  on  $H^{1,0}$  induces  a  positive-definite
bilinear  form  on  the  cohomology  $H^1(S,\R)$: for any $c_1$, $c_2\in H^1(S,\R)$,
$$
( c_1, c_2) := \Re\left( h(c_1), h (c_2) \right)\,.
$$
The  Hodge  bundle
$H^1_{\R{}}$  is  thus endowed with an inner product, called the {\it
Hodge inner product} and a norm, called the {\it Hodge norm}.

Given a cohomology class $c\in H^1(S,\R)$, let $h(c)$ be the unique holomorphic $1$-form such that $c=[\Re h(c)]$. Define $\ast c$ to be the real cohomology class $[\Im h(c)]$. The Hodge norm $\Vert c\Vert$ satisfies
$$
\Vert c\Vert^2=\frac{i}{2}\int_S h(c)\wedge\overline{h(c)}\ =
\int_S \Re h(c) \wedge\Im h(c)\ ,
$$
or, in other words, $\Vert c\Vert^2$ is the value of $c\cdot\ast c$ on the fundamental cycle. The
operator $c\mapsto\ast c$ on the real cohomology $H^1(S,\R)$ of a Riemann surface $S$ is called
the \textit{Hodge star operator}.

We will denote the Hodge inner product of cycles $c_1, c_2\in H^1(S,\R)$ by
round brackets: $(c_1,c_2)=(c_2,c_1)$, and their symplectic
product by angular brackets: $\langle c_1, c_2\rangle =
-\langle c_2, c_1 \rangle$. By definition~\eqref{eq:Intform}, the spaces
$H^{1,0}(S)$ and $H^{0,1}(S)$ are Hodge-orthogonal, hence the following relations hold:

\begin{align}
\label{eq:ast:ast:equals:minus:id}
\ast(\ast c) &= -c\\
\langle c_1, \ast c_2\rangle &= -\langle \ast c_1, c_2 \rangle\\
\label{eq:c1:scalar:c2}
( c_1, c_2) &= \langle c_1, \ast c_2\rangle\\
\label{eq:h:c1:h:c2}
\big(h(c_1),h(c_2)\big) &= (c_1, c_2) +
i \langle c_1,  c_2\rangle
\end{align}

%--------------------------------------------------------------------
%
\subsection{Second fundamental form}
%\label{ss:sec:fund:form}
%

Consider the complex Hodge bundle $H^1_\C{}$ over the moduli
space $\cM_g$ of complex structures having the complex
cohomology space $H^1(S,\C)$ as a fiber over the point of
$\cM_g$ represented by a Riemann surface $S$. This complex
$2g$-dimensional vector bundle is endowed with the flat
Gauss--Manin connection $D_{H_{\C}^1}$ which preserves the Hermitian
form~\eqref{eq:Intform} of signature $(g,g)$.

The bundle $H^1_{\C}$ admits a decomposition into a direct sum
of two orthogonal subbundles $H^1_\C{}=H^{1,0}\oplus H^{0,1}$
with respect to the Hermitian form~\eqref{eq:Intform}. This
decomposition is \textit{not} invariant with respect
to either the flat connection on $H^1_{\C}$ or with respect to the
Teichm\"uller flow. The decomposition defines an orthogonal
projection map $\pi_1$ of the vector bundles $\pi_1:
H^1_{\C{}}\to H^{1,0}$.

The subbundle $H^{1,0}$ is a Hermitian vector bundle with
respect to the Hermitian form~\eqref{eq:Intform} restricted to
$H^{1,0}$. Consider the unique connection $D_{H^{1,0}}$ on
$H^{1,0}$ compatible with the Hermitian metric in the fiber and
with the complex structure on the base of the bundle. This
(nonflat) connection coincides with the connection defined as a
composition of the restriction of $D_{H_{\C}^1}$ to the
subbundle $H^{1,0}$ composed with the projection $\pi_1$:
$$
D_{H^{1,0}}=\pi_1\circ D_{H_{\C}^1}\big\vert_{H^{1,0}}\,,
$$
(see, for example, \cite{Griffiths:Harris}, page 73).

The \textit{second fundamental form} $A_{H^{1,0}}$ defined as
\begin{equation}
\label{eq:sec:fund:form}
A_{H^{1,0}}:=D_{H_{\C}^1}\big\vert_{H^{1,0}} - D_{H^{1,0}} = (I-\pi_1)\circ D_{H_{\C}^1}
\big\vert_{H^{1,0}}
\end{equation}
is a differential form  of type $(1,0)$ with values in the bundle of complex-linear maps from $H^{1,0}$
 to $H^{0,1}$, hence  $A_{H^{1,0}}$ can be written as a matrix-valued differential form of type
 $(1,0)$ (see, for example, \cite{Griffiths:Harris}, page 78). In the literature, $A_{H^{1,0}}$ is also known as the \textit{Kodaira-Spencer map}.

Note that we work with the pullbacks of the vector bundles $H^1_{\C}$, $H^{1,0}$ and $H^{0,1}$ to the moduli spaces $\cH_g$  or $\cQ_g$ of Abelian (correspondingly quadratic) differentials with respect to the natural projections $\cH_g\to\cM_g$ (correspondingly $\cQ_g\to\cM_g$).

We recall that there is a canonical identification between the tangent bundle of the moduli space of Riemann surfaces, which can be naturally described as the bundle $\cB_g$ of equivalence classes of Beltrami differentials, and its cotangent bundle, which is naturally identified to the bundle $\cQ_g$ of holomorphic quadratic differentials. This identification follows from the existence of a canonical pairing between the bundle $\cQ_g$ of quadratic differentials and  the bundle of Beltrami differentials given by integration. For any quadratic differential $q$ and any Beltrami  differential $\mu$ on a Riemann surface $S$, the pairing is given by the formula:
$$
<q,\mu> := \int_S  q\cdot \mu  \,.
$$
In fact, quadratic differentials are tensors of type $(2,0)$ while Beltrami differentials are tensors
of type $(-1,1)$, hence the product of a Beltrami and a quadratic differential is a tensor of type
$(1,1)$ which can be integrated.
Beltrami differentials corresponding to trivial deformations of the complex structure are exactly those which are orthogonal to all quadratic differentials \cite{Nag}, hence the pairing between Beltrami and quadratic differentials induces a non-degenerate pairing between the tangent bundle to the moduli space $\cM_g$ of Riemann surfaces and the bundle of holomorphic quadratic differentials. The bundle of quadratic differentials is thus identified to the cotangent bundle of the moduli space of Riemann surfaces. There exists a natural map $\cI: \cQ_g \to \cB_g$ defined as follows,
$$
 \cI (q) :=   [\frac{ \vert q \vert }{ q}] \in \cB_g  \,, \quad \text{ for all } \, q\in \cQ_g\,,
$$
which yields a canonical identification between the bundles of quadratic and Beltrami differentials,
that is, between the cotangent and the tangent bundles to the moduli space $\cM_g$ of Riemann surfaces.
Taking into account this canonical identification,  the differential form $A_{H^{1,0}}$ with values in the bundle of complex-linear maps from $H^{1,0}$ to $H^{0,1}$ defines a vector bundle map $H^{1,0} \to
H^{0,1}$ over the moduli space of quadratic differentials. In other terms, for any $(S, q)\in \cQ_g$
by evaluating the form $A_{H^{1,0}}$ at the tangent vector $v=q$ under the identification between the tangent bundle and the bundle of quadratic differentials, we get a complex-linear map
\begin{equation}
\label{eq:def:second:fundamental:form}
A_q:H^{1,0}(S) \to H^{0,1}(S)\,.
\end{equation}
For any Abelian differential $\omega\in \cH_g$, let  $A_\omega:= A_q$ be
the complex-linear map corresponding to the quadratic differential $q=\omega^2\in \cQ_g$.

%--------------------------------------------------------------------
%
\subsection{Curvature}
The curvature tensor of the metric connections of the holomorphic  Hermitian bundles $H^{1,0}$,
$H^{1,0}$ are differential forms $\Theta_{H^{1,0}}$, $\Theta_{H^{0,1}}$ of type $(1,1)$ with values in
the bundle of complex-linear endomorphisms of $H^{1,0}$, $H^{0,1}$ respectively, hence they can be written, with respect to pseudo-unitary frames, as skew-Hermitian matrices of differentials forms of type $(1,1)$ on the moduli space $\M_g$ (see~\cite{Griffiths:Harris}, page 73).

Let $\{e_1,\dots,e_{2g}\}\subset H^1_{\C}$ be a pseudo-unitary frame with respect to the pseudo-Hermitian intersection form~\eqref{eq:Intform}, that is, a frame which verifies the pseudo-orthonormality conditions
$$
\begin{aligned}
(e_i,e_j)&=0 \,, \qquad  \text{ \rm for }\,i\neq j\,; \\
 (e_i,e_i)&=1\,, \qquad  \text{ \rm for }\,1\leq i\leq g\,; \\
(e_j,e_j)&=-1\,, \quad  \text{\rm for } g+1\leq j\leq 2g\,.
\end{aligned}
$$
Let us consider the connection matrix $\theta:=\theta_{H^1_{\C}}$ associated to the Gauss-Manin connection $D_{H^1_{\C}}$. By the identities
\begin{equation*}
\begin{aligned}
0&= d(e_i,e_j) = (De_i,e_j)+ (e_i,De_j) \\ &= (\sum\limits_{k}\theta_{ik}e_k,e_j)
+(e_i,\sum\limits_{k}\theta_{jk}e_k)\\ &= \theta_{ij}(e_j,e_j)+\bar{\theta}_{ji} (e_i,e_i)
\end{aligned}
\end{equation*}
it follows that $\theta$ has a block structure
$$\theta=\left(\begin{array}{cc}\theta_1 & B \\ A & \theta_2\end{array}\right)$$
with $g\times g$ blocks $\theta_1$, $\theta_2$, $A$ and $B$ that verify the relations
$$\theta_1=-\overline{\theta_1}^t, \quad \theta_2=-\overline{\theta_2}^t,  \quad B=\overline{A}^t.$$

Observe that  for any unitary frame $\{\omega_1,\dots,\omega_g\}
\subset H^{1,0}$, there is an associated  pseudo-unitary frame $\{e_1, \dots, e_{2g}\} \subset
H^1_\C$, with respect to the intersection form \eqref{eq:Intform}, defined as follows,
$$
\begin{aligned}
e_i&=\omega_i \,, \qquad \text{\rm  for }1\leq i\leq g\,,\\  e_j&=\bar \omega_{j-g}
 \,, \quad \text{\rm  for }g+1\leq j\leq 2g\,,
 \end{aligned}
 $$
and, with respect to the above pseudo-unitary frame, the blocks  $\theta_1$, $\theta_2$ are
equal to the connection matrices $\theta_{H^{1,0}}$, $\theta_{H^{0,1}}$ of the connections
$D_{H^{1,0}}$ and $D_{H^{0,1}}$,  respectively, that is, $\theta_1=\theta_{H^{1,0}}$ and
$\theta_2=\theta_{H^{0,1}}$, and $A=A_{H^{1,0}}$ is the matrix of the second fundamental form
(see \cite{Griffiths:Harris}, page 76).

Let us now consider the curvatures $\Theta_{H^1_{\C}}$,  $\Theta_{H^{1,0}}$ and $\Theta_{H^{0,1}}$ of the connections $D_{H_{\C}^1}$, $D_{H^{1,0}}$  and  $D_{H^{0,1}}$ on the vector bundles $H^1_{\C}$,
$H^{1,0}$ and $H^{0,1}$ respectively. It follows from the above relations that
$$
\theta_{H^1_{\C}}\wedge \theta_{H^1_{\C}} =
\left(\begin{array}{cc}  \theta_{H^{1,0}} \wedge  \theta_{H^{1,0}} + \overline{A}^t \wedge
A& \ast \\ \ast &  \theta_{H^{0,1}} \wedge  \theta_{H^{0,1}} +
A\wedge \overline{A}^t\end{array}\right)\,.
$$
By Cartan's structure equation (see for instance~\cite{Griffiths:Harris}, page 75) we have the
identity $\Theta_{H^1_{\C}}= d\theta_{H^1_{\C}} - \theta_{H^1_{\C}}\wedge \theta_{H^1_{\C}}$,
hence we conclude that
$$
\Theta_{H^1_{\C}}= \left(\begin{array}{cc} \Theta_{H^{1,0}} - \overline{A}^t_{H^{1,0}}\wedge A_{H^{1,0}}& \ast \\ \ast & \Theta_{H^{0,1}}- A_{H^{1,0}}\wedge \overline{A}^t_{H^{1,0}}\end{array}\right)
$$
It follows that (compare with \cite{Griffiths:Harris}, page 78)
$$
\Theta_{H^{1,0}}=
\Theta_{H^1_{\C}}\big\vert_{H^{1,0}}
+ A^\ast_{H^{1,0}} \wedge A_{H^{1,0}}\ .
$$
Note that $\Theta_{H^1_{\C}}$ is the curvature of the Gauss--Manin connection, which is flat. So
$\Theta_{H^1_{\C}}$ is null, and the curvature $\Theta_{H^{1,0}}$ can be written as
\begin{equation}
\label{eq:theta:H:1:0}
\Theta_{H^{1,0}}= A_{H^{1,0}}^\ast\wedge A_{H^{1,0}}\ .
\end{equation}
Similarly to the case of the  second fundamental form, we pull back the bundle $H^{1,0}$ to a point
$(S,q)$ of $\cQ_g$ and take the value of the curvature form on the tangent vectors $v, \bar v$, where the tangent vector  $v= q$ under the identification discussed above between the tangent bundle of the moduli space and the bundle of quadratic differentials. Thus, instead of a differential form of type $(1,1)$ with values in the bundle of complex endomorphisms of the bundle $H^{1,0}$,  we get a section of that bundle over $\cQ_g$, that is, we get a complex endomorphism $\Theta_q$  of the space $H^{1,0}(S)$  for any pair
 $(S, q)\in \cQ_g$.  For any Abelian differential $\omega\in \cH_g$, let $\Theta_\omega:= \Theta_q$  denote the complex endomorphism corresponding to the quadratic differential $q=\omega^2\in \cQ_g$.
For any Riemann surface $S$ and any orthonormal basis $\Omega:=\{\omega_1, \dots, \omega_g\}$ of the space $H^{1,0}(S)$, the system $\bar \Omega=\{\bar \omega_1, \dots, \bar \omega_g\}$ is a pseudo-orthonormal basis of the space  $H^{0,1}(S)$. Let $\Theta$ be the matrix of the complex endomorphism $\Theta_\omega$ with respect to the basis $\Omega$ and $A$ be the matrix of the second fundamental form operator $A_\omega$ with respect to the bases $\Omega$ and $\bar \Omega$. Formula~\eqref{eq:theta:H:1:0} can be written in matrix form as follows:
\begin{equation}
\label{eq:curvatureformula}
\Theta = - A^\ast \cdot A\ .
\end{equation}
It is immediate from the above formulas that the matrix $\Theta$ of the curvature
of the Hodge bundle is a negative-semidefinite Hermitian matrix.

%--------------------------------------------------------------------
%
\subsection{Evaluation of the second fundamental form}
\label{ss:eval:sec:fund:form}

A  formula for the second fundamental form $A_{H^{1,0}}$ was implicitly computed in
 \cite{Forni2}, \S 2. We state such a formula below.
We remark that all formulas in~\cite{Forni2} are written with different notational conventions,
which we now explain for the convenience of the reader. Any Abelian differential $\omega$ on a Riemann surface $S$ induces an isomorphism between the space $H^{1,0}(S)$ of all Abelian differentials on $S$, endowed with the Hodge norm, and the subspace of all square integrable meromorphic functions (with respect to the area form of the Abelian differential $\omega$ on $S$).
In~\cite{Forni2}, \cite{ForniSurvey} variational formulas are written in the language of
meromorphic functions. In this paper we will adopt  the language of Abelian differentials.

\smallskip
Let $(S,\omega)$ be a pair (Riemann surface $S$, holomorphic $1$-form $\omega$ on $S$).
Following \cite{Forni2}, \S 2, for any $\alpha,\beta\in H^{1,0}(S)$ we define:
\begin{equation}
\label{eq:B}
B_\omega(\alpha,\beta):=
\frac{i}{2}\int_S \frac{\alpha\, \beta}{\omega}\,\bar\omega\,.
\end{equation}
The complex bilinear form $B_\omega$  depends continuously, actually
(real) analytically, on the Abelian differential $\omega \in \cH_g$.
\begin{lemma}
\label{lemma:secfundform}
For any $\omega\in \cH_g$, the second fundamental form $A_\omega$
can be written in terms of the complex bilinear form $B_\omega$, namely
$$
(A_\omega(\alpha), \bar \beta) = - B_\omega(\alpha,\beta) \,, \quad \text{ \rm for all }\,
\alpha, \beta \in H^{1,0}(S)\,.
$$
In particular, for any orthonormal basis $\{\omega_1,\dots,\omega_g\}$ of
holomorphic forms in $H^{1,0}(S)$ and any $\alpha\in H^{1,0}(S)$,
$$
A_\omega(\alpha)=\sum_{j=1}^g B_\omega(\alpha,\omega_j) \bar\omega_j\,.
$$
\end{lemma}
\begin{proof}
The argument is a simplified version of the proof of Lemma 2.1
in~\cite{Forni2} rewritten in the language of holomorphic
differentials. As everywhere in this paper we use the same
notation for holomorphic (antiholomorphic)
forms and their cohomology classes, while for other closed 1-forms
we use square brackets to denote the cohomology class.
Let $\{(S_t, \omega_t)\}$ denote a Teichm\"uller deformation, that
is, a trajectory of the Teichm\"uller flow. Let $\alpha$ be any given
holomorphic differential on the Riemann surface $S_0$. There exists
$\epsilon>0$ such that there is a natural identification $H^1(S_t,
\C) \simeq  H^1(S_0, \C)$ by parallel transport for all
$|t|<\epsilon$, so that  locally constant sections are parallel for
the Gauss-Manin connection.

Let $\{\alpha_t\}$  be  a smooth one-parameter family of  closed
$1$-forms such that $\alpha(0) =\alpha$ and $\alpha_t$ is holomorphic
on $S_t$ for all  $|t|<\epsilon$, that is,  $\{\alpha_t\}$ is a
smooth local section of the bundle $H^{1,0}$.  Let $\pi_1:
H^1_{\C}\to H^{1,0}$ denote the natural projection (see
formula~\eqref{eq:sec:fund:form}). By definition,
\begin{equation}
\label{eq:Adef}
A_{\omega}(\alpha)=(I-\pi_1)\circ D_{H^1_{\C}}\big|_{H^{1,0}}(\alpha)
 = (I-\pi_1)\left(\left[\frac {d\alpha}{dt} (0)\right]\right)\,.
\end{equation}
Thus, for any $1$-form $\beta\in H^{1,0}$, in order to compute the
pseudo-Hermitian intersection $(A_\omega(\alpha),  \bar \beta)$, it
is sufficient to compute the derivative $d\alpha/dt (0)$ \emph{up to
exact $1$-forms and up to $1$-forms of type $(1,0)$}. In fact, it
follows from formula~\eqref{eq:Adef} that the cohomology class
$A_{\omega}(\alpha)-[d\alpha/dt (0)] \in H^1(S, \mathbb C)$ is
holomophic.
Hence,
by the definition of the pseudo-Hermitian
intersection form,
it is orthogonal to
$\bar\beta$ for any holomorphic form $\beta$,
$$
\left(A_{\omega}(\alpha)-\left[\frac{d\alpha}{dt}(0)\right]\,,\,
\bar\beta\right)=0\,,
$$
   %
% (taking into account that $\beta$ is closed) we can write
   %
which implies
\begin{equation}
\label{eq:Aid1}
(A_\omega(\alpha), \bar \beta)=
\left(\left[\frac{d\alpha}{dt}(0)\right]\,,\,\bar\beta\right)
%\frac{i}{2}
%\int_S [\frac {d\alpha}{dt} (0)] \wedge \beta
=  \frac{i}{2} \int_S \frac {d\alpha}{dt} (0) \wedge \beta\,.
\end{equation}
It is immediate from the definition of the Teichm\"uller deformation
that
\begin{equation}
\label{eq:Teichdef}
\frac{d\omega}{dt}(0) = \bar \omega\,.
\end{equation}
By writing $\alpha_t = \left(\cfrac{\alpha_t}{\omega_t}\right)
\omega_t$  and differentiating, taking~\eqref{eq:Teichdef} into
account, we derive the following formula:
\begin{equation}
\label{eq:alpha_der1}
\frac{d\alpha}{dt}(0) =
\left(\frac{d}{dt}\frac{\alpha_t}{\omega_t}(0)\right) \cdot \omega
+ \frac{\alpha}{\omega}\, \bar \omega \,.
\end{equation}
By formula~\eqref{eq:alpha_der1}, the $1$-form $d\alpha/dt (0) - (\alpha/\omega) \bar \omega$
is of type $(1,0)$, hence (taking into account that $\beta$ is holomorphic)  we have
\begin{equation}
\label{eq:Aid2}
 \frac{i}{2}  \int_S \frac {d\alpha}{dt} (0) \wedge \beta = \frac{i}{2}
\int_S \frac{\alpha}{\omega}\, \bar \omega \wedge \beta = -B_\omega(\alpha, \beta)\,.
\end{equation}
The first formula in the statement follows from
formulas~\eqref{eq:Aid1} and~\eqref{eq:Aid2}.

Finally, let $\{\omega_1, \dots, \omega_g\} \subset H^{1,0}(S)$ be any orthonormal basis (in the sense that $(\omega_i,\omega_j)=\delta_{ij}$). The system $\{\bar \omega_1, \dots, \bar \omega_g\} \subset H^{0,1}(S)$ is pseudo-orthonormal (in the sense that $(\bar \omega_i,\bar \omega_j)=-\delta_{ij}$), hence
\begin{equation}
A_\omega(\alpha)= - \sum_{j=1}^g (A_\omega(\alpha), \bar\omega_j)  \bar\omega_j
= \sum_{j=1}^g B_\omega(\alpha, \omega_j)  \bar\omega_j \,.
\end{equation}
Thus the second formula in the statement is proved.
\end{proof}
\begin{remark}  We warn the reader that  in general  (unless $\alpha \in \mathbb C \cdot \omega$)
$$
A_\omega(\alpha) \not = \frac{\alpha}{\omega} \bar \omega\,.
$$
In fact, the $1$-form $A_\omega(\alpha)$ is closed by definition,  while the $1$-form $(\alpha/\omega) \bar \omega$ is in general not closed.  In order to compute $A_\omega(\alpha)$ directly, it is necessary to consider the appropriate projection of  the $1$-form $(\alpha/\omega) \bar \omega$ onto the subspace of closed $1$-forms. We carry out such a direct calculation below, following Lemma
2.1 in \cite{Forni2}. We stress that this calculation, although not needed for first variation formulas,
is important for the correct derivation of second variation formulas along the Teichm\"uller
flow.

 Let $\partial$ and $\bar \partial$ denote respectively the type $(1,0)$ and the type $(0,1)$
 exterior differentials on the Riemann  surface $S$, defined as the projections of the (total) exterior
 differential $d$ on the subspaces of $1$-forms of type $(1,0)$ and $(0,1)$ respectively. By definition,
for all $v\in C^\infty(S)$, the $1$-form $\partial v$ is of type $(1,0)$, the $1$-form $\bar \partial v$
is of type $(0,1)$ and the following formula holds
$$
d v= \partial  v +  \bar \partial v  \,.
$$
Note that the $1$-form $(\alpha/\omega) \bar \omega$ is $\bar \partial$-closed
(but not $d$-closed, unless $\alpha \in \mathbb C\cdot\omega$), hence its
$\bar \partial$-cohomology class has a unique anti-holomorphic representative
$p_\omega(\alpha) \in H^{0,1}(S)$. In other words, there exist  a unique
anti-holomorphic form $p_\omega(\alpha) \in H^{0,1}(S)$ and a complex-valued
function $v\in C^\infty(S)$ (unique up to additive constants) such that
\begin{equation}
\label{eq:pi_omega}
\frac{\alpha}{\omega} \bar \omega = p_\omega(\alpha) + \bar \partial v
\end{equation}
(the linear operator $p_\omega: H^{1,0}(S) \to H^{0,1}(S)$ is equivalent to the
restriction to the subspace of  meromorphic functions of the orthogonal projection
from the space of square-integrable functions on $S$ onto
the subspace of anti-meromorphic functions, which appears in the formulas
of \cite{Forni1} and \cite{Forni2}).

Since $d\alpha/dt(0)$ and $p_\omega(\alpha)$  are closed forms and any closed form
of type $(1,0)$ is holomorphic, by formulas~\eqref{eq:alpha_der1}
and~\eqref{eq:pi_omega} it follows that
\begin{equation}
\label{eq:alpha_der2}
\frac{d\alpha}{dt}(0) - (p_\omega(\alpha) + dv) \in H^{1,0}(S) \,,
\end{equation}
hence, by formulas~\eqref{eq:Adef} and~\eqref{eq:alpha_der2}, we conclude that the second fundamental form has the following expression:
\begin{equation}
\label{eq:A}
A_{\omega}(\alpha) = p_\omega(\alpha) \,, \quad \text{ for all } \alpha \in H^{1,0}(S)\,.
\end{equation}
In conclusion, the form $A_{\omega}(\alpha) $ is equal to $(\alpha/\omega)\bar\omega$
only up to a $ \bar \partial$-exact  correction term. If such a correction term were
identically zero, the theory of the Kontsevich--Zorich cocycle would be much simpler!
\end{remark}

\smallskip
The second fundamental form of the Hodge bundle is related to the derivative
of the period matrix along the Teichm\"uller flow. We recall the definition of
the period matrix. Let $\{a_1, \dots, a_g, b_1,\dots, b_g\}$
be a canonical basis of the first homology group $H_1(S, \R)$ of a Riemann
 surface $S$ and let
$\{\theta_1, \dots, \theta_g\} \subset H^{1,0}(S)$ be the dual basis of holomorphic
$1$-forms, that is, the unique basis with the property that
$$
\theta_i(a_j)=\delta_{ij} \,, \quad \text{ for all }  i, j \in \{1, \dots, g\}\,.
$$
The period matrix $\Pi_{ij} (S)$ is the $g\times g$ complex symmetric matrix with
positive-definite imaginary part defined as follows:
\begin{equation}
\Pi_{ij}(S) := \theta_i(b_j)   \,\quad \text{ for all } \, i, j \in \{1, \dots, g\}\,.
\end{equation}
\begin{lemma} Let $\Cal L$ denote the Lie derivative along the Teichm\"uller flow on the space of Abelian differentials. The following formula  holds:
$$
\Cal L \Pi_{ij} (S, \omega) =   B_\omega (\theta_i, \theta_j)\,, \quad \text{ for all }\, i,j \in  \{1, \dots, g\}\,.
$$
\end{lemma}
\begin{proof}
By Rauch's formula (see for instance \cite{IT}, Prop. A.3), for any  Beltrami differential $\mu$
on $S$, we have
\begin{equation}
\frac{  d \Pi_{ij}}{d\mu} (S) =  \frac{i}{2} \int_S  \theta_i \theta_j \mu  \,,
\quad \text{ for all }\, i,j \in  \{1, \dots, g\}\,.
\end{equation}
 By definition, at any holomorphic quadratic differential $q=\omega^2$  on $S$
 the Teichm\"uller flow is tangent  to the equivalence class of Beltrami differentials represented
 by the Beltrami differential
$$
\mu = \frac{ \vert q\vert}{q} = \frac{ \bar \omega}{\omega}\,.
$$
The statement then follows immediately from Rauch's formula.
\end{proof}

For any $\omega\in \cH_g^{(1)}$ (that is, for any Abelian differential $\omega \in \cH_g$ of unit
total area, see formula~\eqref{eq:Ab1}) , the second fundamental form $B_\omega$ satisfies a uniform upper bound and a spectral gap bound, proved below.

\begin{lemma}
\label{lemma:apriori_B_bounds}
For any Abelian differential $\omega\in \cH_g^{(1)}$ on a Riemann surface $S$, the following
uniform bound holds:
\begin{equation}
\label{eq:apriori_B_bounds}
\Vert B_\omega \Vert :=\max \left\{ \frac{\vert B_\omega (\alpha, \beta) \vert}
{ \Vert \alpha \Vert \Vert \beta \Vert} : \alpha, \, \beta
 \in H^{1,0}(S)\setminus \{0\}  \right\}\, =\,1 \,,
 \end{equation}
and the maximum is achieved at $(\alpha, \beta) = (\omega, \omega)$, in fact we have
\begin{equation}
\label{eq:B_max}
B_\omega(\omega, \omega)=\Vert \omega\Vert^2=1 \,.
\end{equation}
Let $\langle \omega\rangle^\perp \subset H^{1,0}(S)$ be Hodge orthogonal
complement of the complex line $\langle \omega\rangle=\C \cdot \omega$. The
following spectral gap bound holds:
\begin{equation}
\label{eq:B_spectr_gap}
\max \left\{ \frac{\vert B_\omega (\alpha, \beta) \vert}
{ \Vert \alpha \Vert \Vert \beta \Vert} : \alpha, \, \beta
 \in H^{1,0}(S)\setminus \{0\} \text{ and } \alpha \in \langle \omega\rangle^\perp\right\}\, <\,1 \,.
 \end{equation}
\end{lemma}
\begin{proof}
By the Cauchy-Schwarz inequality in the space $L^2(S, \frac{i}{2}\, \omega\wedge \bar\omega)$,
it follows that, for all $\alpha$, $\beta\in H^{1,0}(S)$,
\begin{equation}
\label{eq:CS_ineq}
\begin{aligned}
\vert B_\omega (\alpha, \beta) \vert &= \vert \frac{i}{2} \int_S  \frac{\alpha}{\omega}\,
 \frac{\beta}{\omega}\, \omega\wedge \bar\omega\vert  \\
 &\leq \left( \frac{i}{2} \int_S \vert \frac{\alpha}{\omega}\vert ^2\, \omega\wedge \bar\omega \right)^{1/2}
\left(\frac{i}{2} \int_S \vert \frac{\beta}{\omega}\vert ^2  \,\omega\wedge \bar\omega\right)^{1/2} \\
&= \left(\frac{i}{2} \int_S \alpha\wedge \bar\alpha \right)^{1/2} \left(\frac{i}{2} \int_S \beta\wedge \bar\beta \right)^{1/2}=\Vert \alpha\Vert \Vert\beta\Vert\,.
\end{aligned}
\end{equation}
The uniform upper bound in formula~\eqref{eq:apriori_B_bounds} is therefore proved.
The bound is achieved at $(\omega, \omega)$ since it is immediate by the definition that
$$
B_\omega(\omega, \omega) =\Vert \omega \Vert^2\,.
$$
In fact, $\Vert \omega \Vert^2$ is by definition equal to the area of surface $S$ with respect to the flat metric associated to the Abelian differential $\omega \in \cH_g^{(1)}$, which, by definition of
$\cH_g^{(1)}$, is normalized (equal to $1$).

The spectral gap bound in formula~\eqref{eq:B_spectr_gap} is proved as follows. By a fundamental property of the Cauchy-Schwarz inequality, {\it equality }holds in formula~\eqref{eq:CS_ineq}  if and
only if there exists a constant $\const\in\C$ such that
$$
\frac{\alpha}{\omega} =
 \const\cdot   \frac{\overline {\beta}}{\overline{\omega}}\,.
$$
The functions $\alpha/\omega$ and $\beta/\omega$ are meromorphic on the Riemann surface $S$, hence $\bar\beta/\bar\omega$ is anti-meromorphic. Since the only meromorphic functions which are also anti-meromorphic are the constant functions, if follows that equality holds in the Cauchy-Schwarz
inequality if and only if  $\alpha$ and $\beta$ belong to the complex line  $ \langle\omega\rangle \subset H^{1,0}(S)$.  Thus,  if $\alpha\in \langle\omega\rangle^\perp \setminus \{0\}$, the Cauchy-Schwarz inequality is strict and the spectral gap bound stated above on the second fundamental form $B_\omega$  is proved.
\end{proof}

The  curvature  form  of  the  Hodge  bundle  appears  in  the second variation  formulas for the Hodge norm computed in \cite{Forni2}, \S\S 2-5,  which  we  will recall in the next section. For consistence with
the notations of that paper,  we adopt  below a sign convention for the curvature matrix which is opposite to that of formula~\pref{eq:curvatureformula}. For any Abelian differential $\omega\in \cH_g$,  let  $H_\omega$  be  the Hermitian curvature form  defined as follows: for all $\alpha, \beta \in H^{1,0}(S)$,
\begin{equation}
\label{eq:curvatureform}
H_\omega (\alpha, \beta) = -(A_\omega(\alpha), A_\omega(\beta)) = (A^\ast_\omega A_\omega(\alpha),
\beta)\,.
\end{equation}
It follows from Lemma~\ref{lemma:apriori_B_bounds} that, for any Abelian differential $\omega\in
\cH_g^{(1)}$ on a Riemann surface $S$, the second fundamental form operator (the Kodaira-Spencer map) $A_\omega: H^{1,0}(S) \to H^{0,1}(S)$ is a contraction with respect to the Hodge norm and, as a consequence, the Hermitian positive semi-definite curvature form $H_\omega$ is uniformly bounded.
In fact, the following result holds:

\begin{lemma}
\label{lemma:apriori_curv_bounds}
For any Abelian differential $\omega\in \cH_g^{(1)}$ on a Riemann surface $S$, the following
bounds hold:
\begin{equation}
\label{eq:apriori_curv_bounds}
\begin{aligned}
\Vert A _\omega \Vert  &:=  \max \left\{ \frac{\Vert A_\omega (\alpha) \Vert}{ \Vert \alpha \Vert } :
\alpha \in H^{1,0}(S)\setminus \{0\} \right\} \,=\,1\,; \\
\Vert H _\omega \Vert  &:=  \max \left\{ \frac{\vert H_\omega (\alpha,\beta) \vert}{ \Vert \alpha
\Vert \Vert \beta\Vert} : \alpha, \beta \in H^{1,0}(S)\setminus \{0\} \right\} \,=\,1\,;
\end{aligned}
\end{equation}
and the maximum is achieved at $(\alpha, \beta) = (\omega, \omega)$, in fact we have
\begin{equation}
\label{eq:curv_max}
 A_\omega (\omega) =\bar \omega \quad \text{ and }  \quad
 H_\omega(\omega, \omega) =\Vert \omega\Vert^2=1 \,.
\end{equation}
The following spectral gap result holds:
\begin{equation}
\label{eq:curv_spectr_gap}
\begin{aligned}
 &\max  \left\{\frac{  \Vert A_\omega(\alpha) \Vert }{ \Vert \alpha \Vert} :  \alpha \in
 \langle\omega \rangle^\perp\setminus \{0\}\subset H^{1,0}(S) \right\} <1\,; \\
 &\max  \left\{\frac{  \vert H_\omega(\alpha,\beta) \vert }{ \Vert \alpha \Vert \Vert \beta\Vert} :   \alpha, \beta\in  H^{1,0}(S) \setminus \{0\} \text{ and }\alpha \in \langle\omega \rangle^\perp \right\} <1\,.
 \end{aligned}
 \end{equation}
\end{lemma}
\begin{proof}
By Lemma~\ref{lemma:secfundform}, it follows that, for all $\alpha \in H^{1,0}(S)$, we have
$$
\Vert A_\omega (\alpha) \Vert  = \max_{\beta\not =0}\frac{\vert (A_\omega(\alpha), \bar\beta)\vert}
{\Vert \beta \Vert}  =  \max_{\beta\not =0}\frac{\vert B_\omega(\alpha, \beta)\vert}
{\Vert \beta \Vert} \,,
$$
and by the definition of the curvature form $H_\omega = A_\omega^{\ast} \cdot A_\omega$, we also have
$$
\vert H_\omega(\alpha, \beta) \vert = \vert \left(A_\omega(\alpha), A_\omega(\beta)\right)\vert
\leq \Vert A_\omega(\alpha)\Vert \Vert A_\omega(\beta)\Vert\,.
$$
The upper bounds in formulas~\eqref{eq:apriori_curv_bounds} and~\eqref{eq:curv_spectr_gap}  therefore follow from the corresponding results for the form $B_\omega$
established in Lemma~\ref{lemma:apriori_B_bounds}.

The identities~\eqref{eq:curv_max} follow from formula
\eqref{eq:B_max} in Lemma~\ref{lemma:apriori_B_bounds}, which states that
$B_\omega(\omega, \omega)=1$. In fact, by Lemma~\ref{lemma:secfundform}
we have
$$
A_\omega(\omega) =B_\omega(\omega, \omega) \bar \omega =\bar\omega\,.
$$
Finally, by the definition of the curvature form it follows that
$$H_\omega(\omega, \omega) = -(A_\omega(\omega), A_\omega(\omega))=
-(\bar \omega, \bar\omega)= \Vert \omega\Vert^2 =1\,.$$
The argument is complete.
\end{proof}

For any Abelian differential $\omega\in \cH_g$ on a Riemann surface $S$
the  matrix  $H$  of the Hermitian curvature form $H_\omega$ with respect to any
Hodge-orthonormal  basis $\Omega:= \{\omega_1, \dots, \omega_g\}$ can
be  written  as  follows.

Let $B$ be the matrix of the bilinear form
$B_\omega$  on  $H^{1,0}(S)$ with respect to the basis $\Omega$, that
is:
$$
B_{jk}:=
\frac{i}{2}\int_S \frac{\omega_j\, \omega_k}{\omega}\,\bar\omega\,.
$$
By formula \eqref{eq:curvatureformula} and Lemma \ref{lemma:secfundform}, the matrix
$H$ of the Hermitian curvature form $H_\omega$ of the vector bundle $H^{1,0}$ over
$\omega\in\cH_g$  in the orthonormal basis $\Omega$ can be written as follows:
\begin{equation}
\label{eq:Hform}
H =B\cdot B^\ast\ .
\end{equation}
(The above formula is the corrected version of the formula $H= B^\ast B$ which appears
as formula $(4.3)$ in \cite{Forni2} and as formula $(44)$ in \cite{ForniSurvey}. This mistake
there is of no consequence).
In particular, since the form $B_\omega$ is symmetric, the forms $H_\omega$ and
$B_\omega$ have the same rank and their eigenvalues are related. Let $ \text{\rm EV}(H_\omega)$
and  $\text{\rm EV}(B_\omega)$ denote the set of eigenvalues of the forms $H_\omega$
and $B_\omega$ respectively. The following identity holds:
$$
\text{\rm EV}(H_\omega) =
\big\{ \vert \lambda\vert^2\  \big\vert\  \lambda \in  \text{\rm EV}(B_\omega)\big\}\,.
$$
For every Abelian differential $\omega\in\cH_g^{(1)}$, the eigenvalues  of the positive-semidefinite form
$H_\omega$ on $H^{1,0}(S)$ will be denoted as follows:
\begin{equation}
\label{eigenvalues}
\Lambda_1(\omega) \equiv 1
> \Lambda_2(\omega) \geq \dots \geq \Lambda_g(\omega)\geq 0\,.
\end{equation}
 We  remark  that  the  top eigenvalue $\Lambda_1(\omega)$ is equal to $1$ and
 the second eigenvalue $\Lambda_2(\omega)<1$ for any Abelian differentials $\omega\in \cH_g^{(1)}$
 as a consequence of Lemma~\ref{lemma:apriori_curv_bounds}, in particular all of the above eigenvalues give well-defined, continuous, non-negative, bounded  functions  on  the
 moduli space $\cH_g^{(1)}$ of all (normalized) Abelian differentials.

By  the  Hodge  representation theorem for Riemann surfaces, the forms
$H_\omega$  and  $B_\omega$ induce \textit{complex-valued}
bilinear forms $H_\omega^\R$ and $B_\omega^\R$ on the real cohomology
$H^1(S,\R)$: for all $c_1$, $c_2\in H^1(S,\R)$,
\begin{align}
\notag
H_\omega^\R(c_1,c_2)&:=H_\omega(h(c_1),h(c_2))\\
\label{eq:B:R}
B_\omega^\R(c_1,c_2)&:=B_\omega(h(c_1),h(c_2))\,.
\end{align}

The  forms  $H_\omega^\R$  and  $B_\omega^\R$ on $H^1(S,\R)$ have the
same  rank,  which  is  equal  to  twice the common rank of the forms
$H_\omega$  and  $B_\omega$  on  $H^{1,0}(S)$. The bilinear form $H_\omega^\R$
induces a   real-valued, positive semi-definite quadratic form,  while the quadratic
form induced by the bilinear form $B_\omega^\R$  is complex-valued.

%-------------------------------------------------------------------------
\subsection{Variational formulas for the Hodge norm}
\label{sec:varfor}
We recall below some basic variational formulas from~\cite{Forni2},~\S2,~\S 3 and~\S 5,
reformulated in geometric terms. Such formulas generalize the fundamental
Kontsevich formula for the sum of all non-negative Lyapunov exponents of the Hodge
bundle~\cite{Kontsevich}.
\subsubsection{First Variation}
The second fundamental form of the Hodge bundle measures the first variation of the Hodge norm
along a parallel (locally constant) section for the Gauss-Manin connection. In fact, the
formula given below (implicit in the computation of \cite{Forni2}, \S 2) holds.
Let $(S,\omega)$ be a pair (Riemann surface $S$, holomorphic $1$-form $\omega$ on $S$).
For any cohomology class $c\in H^1(S,\R)$, let $h_\omega(c)$ be the unique holomorphic
$1$-form such that $c$ is the cohomology class of the closed $1$-form $\Re h_\omega(c)$
in the de Rham cohomology $H^1_{deRham}(S, \R)$. We remark that for any given $c\in H^1(S,\R)$
the holomorphic $1$-form $h_\omega(c)$ only depends on the Riemann surface $S$. However, the
Riemann surface $S$ underlying a given holomorphic $1$-form $\omega$ is unique
and it will be convenient to write below the harmonic representative as a function of the
holomorphic $1$-form $\omega$ on $S$.

\begin{lemma}
\label{lemma:varfor1}
The Lie derivative $\cL$ of the Hodge inner product  $(c_1, c_2)_\omega$ of parallel (locally
 constant) sections  $c_1, \, c_2 \in H^1(S,\R)$ in the direction of the Teichm\"uller flow can be written
 as follows:
 \begin{equation}
 \label{eq:varfor1}
\cL ( c_1, c_2)_\omega  =  2 \Re ( A_\omega (h_\omega(c_1)),  \overline{h_\omega(c_2)})  \,.
\end{equation}
 \end{lemma}
 \begin{proof}
 The argument is just a rephrasing of  the proof of Lemma  2.1' of \cite{Forni2}
 in the language of differential geometry.
Let us  recall that by definition, for any pair $c_1, c_2 \in H^1(S, \R)$, the Hodge
   inner product is defined as
$$
( c_1, c_2)_\omega = \Re\left( h_\omega(c_1), h_\omega (c_2) \right)\,.
$$
Since the Gauss-Manin connection is compatible with the Hermitian intersection
form,  we can compute
\begin{equation}
\label{eq:compatible}
\begin{aligned}
\cL ( c_1, c_2)_\omega &=       \Re\left( D_{H^1_\C} h_\omega(c_1),
h_\omega(c_2)\right)   \\ &+    \Re\left(  h_\omega(c_1),
D_{H^1_\C} h_\omega(c_2)\right)\,,
\end{aligned}
\end{equation}
and, since $(h_\omega(c_1), \overline {h_\omega(c_2)})= 0$, we also have
\begin{equation}
\label{eq:adjoint}
 (D_{H^1_\C} h_\omega(c_1), \overline{ h_\omega(c_2)}) +  (h_\omega(c_1),  D_{H^1_\C}
 \overline{h_\omega(c_2)}) =0\,.
\end{equation}
Any cohomology class $c\in H^1(S,\R)$ can be interpreted as a parallel
(constant) local section of the bundle $H^1_\C$. Since by definition of the differential
$h_\omega(c)\in H^{1,0}(S)$ we have $c = [h_\omega(c)+ \overline{h_\omega(c)}]/2$
and since the Gauss-Manin connection
is real (on real tangent vectors of the moduli space it commutes with the complex
conjugation), it follows that
\begin{equation}
\label{eq:par_transp}
D_{H^1_\C} h_\omega(c)= - D_{H^1_\C} \overline{h_\omega(c)}
 = - \overline {D_{H^1_\C}  h_\omega(c)} \,.
\end{equation}
From formulas~\eqref{eq:adjoint} and~\eqref{eq:par_transp} we can derive the identities
\begin{equation*}
\begin{aligned}
 (D_{H^1_\C} h_\omega(c_1), h_\omega(c_2)) &= -  (\overline{ D_{H^1_\C} h_\omega(c_1)},
  h_\omega(c_2)) =  \overline{(D_{H^1_\C} h_\omega(c_1), \overline{ h_\omega(c_2)})} \,; \\
  ( h_\omega(c_1), D_{H^1_\C} h_\omega(c_2)) &= -  ( h_\omega(c_1),
 \overline{ D_{H^1_\C} h_\omega(c_2)}) =  (D_{H^1_\C} h_\omega(c_1), \overline{ h_\omega(c_2)}) \,.
  \end{aligned}
\end{equation*}
In conclusion, from formula~\eqref{eq:compatible}, by the above identities and by the definition~(\ref{eq:sec:fund:form}) of the second fundamental form, it follows that
\begin{equation}
\begin{aligned}
\cL ( c_1, c_2)_\omega &= 2\Re (D_{H^1_\C} h_\omega(c_1), \overline{ h_\omega(c_2)}) \\
&= 2\Re ((I-\pi_1)D_{H^1_\C} h_\omega(c_1), \overline{ h_\omega(c_2)}) \\
&=  2\Re  \left(A_\omega( h_\omega(c_1)), \overline{ h_\omega(c_2)}\right) \,.
\end{aligned}
\end{equation}
The stated first variation formula is therefore proved.
\end{proof}
The fundamental variational formula, computed in \cite{Forni2}, Lemma 2.1',
 for the Lie derivative of the Hodge norm  of a parallel (locally constant) section
 $c\in H^1(S,\R)$ in the direction of the Teichm\"uller flow can now be derived from
 Lemma \ref{lemma:varfor1} and Lemma \ref{lemma:secfundform}.
\begin{lemma}
\label{lemma:varfor2}
The following variational formula holds:
\begin{equation}
\label{eq:varfor2}
\cL (c_1, c_2)_\omega =  -2 \Re B_\omega (h_\omega(c_1),h_\omega(c_2)) =
-2 \Re B^\R_\omega (c_1,c_2) \,.
\end{equation}
\end{lemma}
\begin{remark} Lemma~\ref{lemma:secfundform} can also be derived from the variational
formulas of Lemma~\ref{lemma:varfor1} (proved above) and Lemma \ref{lemma:varfor2}
(proved as part of Lemma 2.1' in \cite{Forni2}). In fact, by comparison of the variational formulas
of Lemmas~\ref{lemma:varfor1} and~\ref{lemma:varfor2}, for any cohomology classes
$c_1, \, c_2\in H^1(S,\R)$,
$$
\Re \left(A_\omega(h_\omega(c_1)), \overline{h_\omega(c_2)}\right)  =
 - \Re B_\omega (h_\omega(c_1),h_\omega(c_2)) \,,
$$
which implies the main identity of Lemma \ref{lemma:secfundform} since the operator
$A_\omega$ is complex linear, the intersection form is Hermitian by definition and
the form $B_\omega$ is complex bilinear.
\end{remark}
The variational formula of Lemma \ref{lemma:varfor2} implies a uniform bound
and a spectral gap result on the exponential growth of Hodge norms based on the uniform bound and on the spectral gap estimate of
Lemma~\ref{lemma:apriori_B_bounds} (see  \cite{Forni2}, Lemma~2.1'
and Corollary~2.2).

Let $\Lambda:\cH_g^{(1)} \to \R^+$ be the function defined as follows: for all $\omega\in \cH_g^{(1)}$,
\begin{equation}
\label{eq: Lambda}
\Lambda(\omega) :=  \max \left\{ \frac{  \vert B_\omega(\alpha,\alpha) \vert }{ \Vert \alpha \Vert^2} :
\alpha\in \langle\omega\rangle^\perp\setminus\{0\} \subset H^{1,0}(S) \right\}\,.
\end{equation}
By definition  $\Lambda$ is a continuous function on the moduli space of normalized (unit area) Abelian differentials
and by Lemma~\ref{lemma:apriori_B_bounds} it is everywhere strictly less than $1$,
hence it achieves its maximum (strictly less than $1$) on every compact subset. It is proved
in \cite{Forni2} that  $\Lambda$ has supremum equal to $1$ on every connected component
of every stratum of the moduli space.

For any Abelian differential $\omega \in \cH_g$ on a Riemann
surface $S$, let $\Vert c\Vert_\omega$ denote the Hodge norm of  a real cohomology
class $c\in H^1(S,\R)$, that is, the Hodge norm of the holomorphic $1$-form
$h_\omega(c) \in H^{1,0}(S)$.
\begin{corollary}
\label{cor:universal_1}
The Lie derivative of the Hodge norm along the Teichm\"uller flow admits the following bounds:  for  any
Abelian differential $\omega \in \cH_g^{(1)}$ on a Riemann surface $S$
and for any cohomology class $c\in H^1(S,\R)$,
\begin{equation}
\vert \Cal L \log  \Vert c \Vert_\omega  \vert \leq 1\,;
\end{equation}
for any cohomology class $c\in \langle[\Re(\omega)], [\Im(\omega)]\rangle^\perp$,
\begin{equation}
\vert \Cal L \log  \Vert c \Vert_\omega  \vert \leq \Lambda(\omega) <1\,.
\end{equation}
\end{corollary}
\begin{proof} By Lemma~\ref{lemma:varfor2}, for any Abelian differential $\omega \in \cH_g$ on a Riemann surface $S$ and for any
cohomology class $c\in H^1(S, \R)$ we have
$$
\Cal L \log  \Vert c \Vert_\omega = - \frac{\Re B^\R_\omega(c,c)}{\Vert c \Vert_\omega^2}\,;
$$
hence the statement follows from Lemma~\ref{lemma:apriori_B_bounds}. In fact, for any cohomology
class $c\in H^1(S,\R)$, the Abelian differential $h_\omega(c)$ belongs to the Hodge orthogonal
complemement $ \langle\omega\rangle^\perp \subset H^{1,0}(S)$ if and only if $c$ belongs
to the Hodge orthogonal complemement $\langle[\Re(\omega)], [\Im(\omega)]\rangle^\perp
\subset H^1(S,\R)$.
\end{proof}

The above universal bound and spectral gap estimate immediately extends to all exterior powers
of the Hodge bundle. For every Abelian differential $\omega\in \cH_g$ on a Riemann surface $S$ and for every $k\in \{1, \dots, 2g\}$,  the Hodge norm $\Vert \cdot \Vert_\omega$ on $H^1(S, \R)$ induces a natural norm (also called the Hodge norm) $\Vert c_1\wedge \dots \wedge c_k\Vert_\omega$ on polyvectors $c_1\wedge \dots \wedge c_k \in \Lambda^k(H^1(S, \R))$.

\begin{corollary}
\label{cor:universal_2}
The Lie derivative of the Hodge norm along the Teichm\"uller flow admits the following bounds: for  any  Abelian differential $\omega \in \cH_g^{(1)}$ on a Riemann surface $S$, for any $k\in \{1, \dots, g\}$ and for any polyvector $c_1 \wedge \dots \wedge c_k \in \Lambda^k(H^1(S, \R)) $ such that the span
$\<c_1, \dots, c_k\> \subset H^1(S, \R)$ is  an isotropic subspace,
we have:
 \begin{equation}
\vert \Cal L \log  \Vert  c_1 \wedge \dots \wedge c_k  \Vert_\omega  \vert \leq k\,;
\end{equation}
for any $k \geq 2$ the following stronger bound holds:
 \begin{equation}
\vert \Cal L \log  \Vert c_1 \wedge \dots \wedge c_k  \Vert_\omega  \vert \leq 1 + (k-1) \Lambda(\omega) < k\,;
\end{equation}
\end{corollary}

%------------------------------------------------------------
\subsubsection{Second Variation}

The $\SL$-orbit of any holomorphic Abelian differential $\omega_0\in \cH_g$
is isomorphic to the unit tangent bundle of a hyperbolic
surface (generically a copy of the Poincar\'e disk). Thus the
left quotient $\SO\backslash\big(\SL\,\omega_0\big)$ is
a hyperbolic surface, called a \textit{Teichm\"uller disk}.

There  is  a natural action of $\mathbb{C}^\ast$ on
the  space  $\cH_g$  by  multiplication  of Abelian
differentials   by    nonzero  complex  numbers.  The  corresponding
projectivization  $\PcH_g:=\cH_g/\mathbb{C}^\ast$ is foliated by
Teichm\"uller  disks  endowed  with  the hyperbolic metric. We remark
that   for   consistency   with  a  standard  normalization  for  the
Teichm\"uller  geodesic flow adopted in the literature the hyperbolic
metric is normalized to have curvature equal to $-4$.
We   have   the   following   basic   variational   formula  for  the
\textit{leafwise}  hyperbolic Laplacian $\triangle$ of the Hodge norm
$\Vert  c\Vert_\omega$  at  a ``point'' $\omega$ of the projectivized
moduli space $\PcH_g$ (see \cite{Forni2},
Lemmas~2.1' and~3.2,  \cite{ForniSurvey}, Lemma~4.3):
\begin{lemma}
\label{lemma:varlapl}
The following variational formula for the Hodge norm holds:
\begin{equation}
\label{eq:variational:formulas}
\triangle \log \Vert c \Vert_\omega
  = 4\, \frac{ H^\R_\omega (c,c)}{\Vert c \Vert_\omega^2}  -
   2\frac{ \vert B^\R_\omega (c,c)\vert^2}{\Vert c \Vert_\omega^4 }  \geq 0 \,.
\end{equation}
\end{lemma}
\begin{Remark}
\label{rm:SO:invariant}
In  fact, given a cohomology class $c$ in $H^1(S,\R)$, the Hodge norm
$\|c\|_\omega$  at  a  point  $(S,\omega)$  of  $\cH_g$ is completely
determined by the complex structure of the underlying Riemann surface
$S$.  Thus,  for  any holomorphic form $\omega'=const\cdot\omega$ one
has     $\|c\|_\omega=\|c\|_{\omega'}$.     Whenever,     in     addition
$const=\exp(ix)$  with  real  $x$,  then  $H_\omega=H_{\omega'}$  and
$\vert B_\omega\vert=\vert B_{\omega'}\vert$.  Thus,  all  the  quantities in the above
formula  are  $\SO$-invariant,  which makes it legitimate to consider
them defined on a Teichm\"uller disk in the projectivized moduli space $\PcH_g$.
It will often be convenient to pull back to $\cH_g$ the functions defined
on the projectivized moduli space $\PcH_g$.
When operating with the leafwise hyperbolic Laplacian
$\triangle$, we will always tacitly verify the $\SO$-invariance
of the functions involved.
\end{Remark}

\subsection{Variational formulas for exterior powers}
The above variational formulas can be generalized to all the
exterior powers of the Kontsevich--Zorich cocycle. For any $k\in\{1, \dots, g\}$,
let us denote by $G_k(H^1_\R)$ the total space of the Grassmannian bundle of isotropic
$k$-dimensional subspaces of the Hodge bundle. By definition, the fiber $G_k(H^1_\R)_\omega$
of the bundle $G_k(H^1_\R)$ at any Abelian differential $\omega \in \cH_g$
on a Riemann surface $S$ is the Grassmannian manifold of all $k$-dimensional
isotropic  subspaces of the symplectic vector space $H^1(S, \R)$.

Let $\Phi_k$ denote the function  on the Grassmannian bundle $G_k(H^1_\R)$
defined as follows (see \cite{Forni2}, \S 5).
Let $\omega\in \cH_g$ be an Abelian differential on a Riemann surface $S$ and let $I_k \subset H^1(S,\R)$ be any isotropic subspace of dimension $k\in\{1, \dots, g\}$. Let $\{c_1, \dots, c_k\} \subset I_k$ be any Hodge-orthonormal basis and let  $\{c_1, \dots, c_k, c_{k+1}, \dots,c_g\} \subset H^1(S,\R)$ be any
Hodge-orthonormal Lagrangian completion. Let
\begin{equation}
\label{eq:Phi_k}
\Phi_k (\omega, I_k) :=  2\sum_{i=1}^k H^\R_\omega(c_i,c_i) - \sum_{i,j=1}^k
\vert B^\R_\omega(c_i, c_j) \vert ^2\,.
\end{equation}
\begin{Lemma}[Forni~\cite{Forni2}, Lemma~5.2']
\label{lm:Fi:k}
The function $\Phi_k (\omega, I_k)$ depends only on $\omega\in \cH_g$
and on the isotropic subspace $I_k\subset H^1(S,\R)$ and is independent of the
choice of the orthonormal basis $\{c_1, \dots, c_k\} \subset I_k$ and of its Hodge-orthonormal Lagrangian completion $\{c_1, \dots, c_k, c_{k+1}, \dots,c_g\}$.
It can be also expressed as
\begin{equation}
\label{eq:Phi:alternative}
\Phi_k (\omega, I_k) =  \sum_{i=1}^g \Lambda_i(\omega) - \sum_{i,j=k+1}^g
\vert B^\R_\omega(c_i, c_j) \vert ^2\,
\end{equation}
(for $k=g$ the second sum on the right hand side is defined to be null).
For any normalized (unit area) Abelian differential $\omega\in\cH_g^{(1)}$ on a surface $S$
and for any $k$-dimensional isotropic subspace $I_k\subset H^1(S,\R)$, the following bound holds:
\begin{equation}
\label{eq:bound:for:Phi}
|\Phi_k (\omega, I_k)|\le \min(2k,g) \quad \text{and the inequality is strict for  $k\geq 2$}.
\end{equation}

 \end{Lemma}
\begin{proof}
We reproduce here the proof of equivalence of definitions~\eqref{eq:Phi_k} and~\eqref{eq:Phi:alternative} since the same kind of calculations will be repeatedly used in
the sequel. Let $\{\omega_1,\dots,\omega_g\}\subset H^{1,0}(S)$ be an orthonormal basis of
Abelian differentials representing the orthonormal basis $\{c_1,\dots,c_g\}\subset H^1(S,\R)$.
By the definition  of $H^\R_\omega$ and $B^\R_\omega$ we have that
$H^\R_\omega(c_i,c_j):=H_\omega(\omega_i,\omega_j)$ and
$B^\R_\omega(c_i,c_j):=B_\omega(\omega_i,\omega_j)$.
By formula~\eqref{eq:Hform} we also have the relation $H_\omega=B_\omega\cdot B_\omega^\ast$,
which implies that
\begin{equation}
\label{eq:Hii:in:terms:of:B}
H_\omega(\omega_i,\omega_i)=
\sum_{j=1}^g B_\omega(\omega_i,\omega_j) \overline{ B_\omega(\omega_i,\omega_j)} =
\sum_{j=1}^g |B_\omega(\omega_i,\omega_j)|^2\,.
\end{equation}
By definition~\eqref{eigenvalues} one has
\begin{equation}
\label{eq:traceH}
\Lambda_1(\omega)+\dots+\Lambda_g(\omega)=\Tr H_\omega=
\sum_{i=1}^g\sum_{j=1}^g |B_\omega(\omega_i,\omega_j)|^2
\,.
\end{equation}
By formulas~\eqref{eq:Hii:in:terms:of:B} and~\eqref{eq:traceH}
 and by taking into account that $B$ is symmetric, we
transform formula~\eqref{eq:Phi_k} as follows:
\begin{equation}
\label{eq:Pi_k:various:formulae}
\begin{aligned}
\Phi_k (\omega, I_k)  &:= 2\sum_{i=1}^k H_\omega(\omega_i,\omega_i) - \sum_{i,j=1}^k
\vert B_\omega(\omega_i, \omega_j) \vert ^2\\
&=\,
2\sum_{i=1}^k\sum_{j=1}^g |B_\omega(\omega_i,\omega_j)|^2-
\sum_{i,j=1}^k |B_\omega(\omega_i,\omega_j)|^2\\
&=\,
\sum_{i=1}^k\sum_{j=1}^g |B_\omega(\omega_i,\omega_j)|^2+
\sum_{i=1}^k\sum_{j=k+1}^g |B_\omega(\omega_i,\omega_j)|^2 \\
&=\,
\sum_{i,j=1}^g |B_\omega(\omega_i,\omega_j)|^2
- \sum_{i,j=k+1}^g \vert B_\omega(\omega_i, \omega_j) \vert ^2\\
&=\,
 \sum_{i=1}^g \Lambda_i(\omega) - \sum_{i,j=k+1}^g
\vert B_\omega(\omega_i, \omega_j) \vert ^2\,.
\end{aligned}
\end{equation}
 Formula~\eqref{eq:Phi:alternative} is therefore established.
 Note that  formula~\eqref{eq:Phi_k} does not depend
 on the choice of the system $\{c_{k+1},\dots,c_g\}$ in the orthonormal Lagrangian completion,
 while formula~\eqref{eq:Phi:alternative} does not depend on the choice
 of the orthonormal basis $\{c_1,\dots,c_k\}$. Hence, the equality between
 formulas~\eqref{eq:Phi_k} and~\eqref{eq:Phi:alternative} proves that neither
 formula depends on the choice of the orthonormal basis of $I_k$
and of its orthonormal Lagrangian completion.

 The bound in formula~\eqref{eq:bound:for:Phi} follows
 from formulas~\eqref{eq:Phi_k} and~\eqref{eq:Phi:alternative},  by taking
 into account that $H_\omega(\omega_i,\omega_i) \leq \Vert \omega_i\Vert^2=1$
 and $\Lambda_i(\omega)\leq 1$, for all $i\in \{1, \dots, g\}$,
 according to the upper bounds in formulas  \eqref{eq:apriori_curv_bounds}
 and \eqref{eigenvalues} respectively. Finally,  by the spectral gap
 bound in Lemma~\ref{lemma:apriori_curv_bounds} and by the consequent
 strict bound for the second curvature eigenvalue  in formula~\eqref{eigenvalues},
 for $k\geq 2$ the inequality in formula~\eqref{eq:bound:for:Phi} is strict .
  \end{proof}

%   The function $\Phi_k$ is
%   therefore well-defined
%  and is invariant under the action of
%  the mapping class group on the Grassmannian bundle of the
%  (real) Hodge bundle over the Teichm\"uller lspace. It induces
%  therefore a function
% on the Grassmannian bundle of the (real)
% Hodge bundle over the moduli space of all Abelian
% differentials

The  formulas below, computed in \cite{Forni2}, Lemma 5.2 and
5.2', extend the formula in Lemma~\ref{lemma:varlapl}
to all exterior powers of the Hodge bundle.
Let $\{c_1, \dots, c_k\} \subset I_k$ be any Hodge-orthonormal
basis of an isotropic subspace $I_k\subset H^1(S,\R)$ on a
Riemann surface $S$. Recall that the Euclidean structure defined by the Hodge scalar
product on $H^1(S,\R)$ defines the natural norm
$\Vert c_1 \wedge \dots \wedge c_k\Vert_\omega$
of any polyvector which we also call the Hodge norm.
Similarly to the case of the Hodge norm,
it is defined only by the complex structure of the underlying
Riemann surface. It follows from the definition~\eqref{eq:Phi_k} of
$\Phi_k(\omega,I_k)$ that this function is $\SO$-invariant
(compare to Remark~\ref{rm:SO:invariant}). Thus, for any
$(S_0,\omega_0)$ in $\cH_g$ the Hodge norm $\Vert c_1 \wedge \dots
\wedge c_k\Vert_\omega$ defines a smooth function on the
hyperbolic surface obtained as a left quotient
$\SO\backslash \SL \cdot\omega_0$  of the orbit of $\omega_0$.

\begin{lemma}
\label{lemma:var2}
For all $k\in \{1, \dots, g\}$ the following formula holds:
$$
\triangle \log \Vert c_1 \wedge \dots \wedge c_k\Vert_\omega =
2 \Phi_k (\omega, I_k)  \geq 0 \,.
$$
\end{lemma}
\begin{proof}
See   Lemma 5.2 and 5.2' in \cite{Forni2}.
\end{proof}

Recall  that  our  hyperbolic Laplacian $\triangle$ is written in the
hyperbolic metric of constant curvature $-4$. Any different choice of
the constant negative curvature would change the RHS in the above
formula by a constant  factor.

For $k=1$ the formula of Lemma \ref{lemma:var2} reduces to that of
Lemma \ref{lemma:varlapl}.

\begin{Remark}
\label{rm:Phi:g}
Note that for $k=g$
the Lagrangian subspace $I_g$ is not present in the right-hand side
of definition~\eqref{eq:Phi_k} of $\Phi_g(\omega,I_g)$:
$$
\Phi_g (\omega, I_k) :=  \sum_{i=1}^g \Lambda_i(\omega) \,.
$$
Thus,  the  function  $\Phi_g$  is  the pull-back to the Grassmannian
bundle  of  Lagrangian  subspaces  of  a function on the moduli space
$\cH_g$. Moreover, by definition~\eqref{eigenvalues} the above sum is
the trace of the Hermitian form $H_\omega$, hence it is by definition
the  curvature  of  the  Hermitian bundle $H^{1,0}$. This fundamental
fact  discovered  in~\cite{Kontsevich} is crucial for the validity of
the Kontsevich formula for the sum of exponents. A version of
this formula is stated in Corollary~\ref{cor:KZformula} below.
\end{Remark}

\section{The Kontsevich--Zorich exponents}

In this section we derive formulas for the Lyapunov exponents of the Kontsevich--Zorich
cocycle on the Hodge bundle in terms of the second fundamental form and curvature
of the Hodge bundle.

\subsection{Lyapunov exponents} Let $(T_t)_{t\in\mathbb{R}}:X\to X$ be a flow preserving a Borel ergodic probability measure $\mu$ on a locally compact topological space $X$.
Let $\pi:M\to X$ be a real or complex $d$-dimensional vector bundle. In other words,
the fiber $M_x:=\pi^{-1}(x)$ of the vector bundle above any $x\in X$ is a real or complex vector
space  isomorphic to $\mathbb{R}^d$  or $\mathbb{C}^d$ respectively. A \emph{real or complex linear cocycle} $(F_t)_{t\in\mathbb{R}}:M\to M$ over the flow $(T_t)$ is a flow on the total space $M$ of the vector bundle such that, for all $(x,t)\in X\times \mathbb R$, the map $F_t: M_x \to M_{T_t(x)}$ is well-defined and linear over $\mathbb R$ or $\mathbb C$ respectively. Suppose that there exists a family of norms $\{\vert \cdot \vert_x\}_{x\in X}$ on the fibers  of the vector bundle and, for all $(x,t)\in X\times
\mathbb R$,  let $\Vert F_t \Vert_x$ denote the operator norm of the linear map $F_t$ with respect
to the norm $\vert \cdot \vert_x$ on $M_x$ and $\vert \cdot \vert_{T_t(x)}$ on $M_{T_t(x)}$.
Under the condition of $\log$-\emph{integrability} of the cocycle $(F_t)$, that is, under the condition
that
$$
\int_X  \log\|F_{\pm1}\|_x d\mu(x)    \, < +\infty \,,
$$
 the so-called \emph{Oseledets theorem} states that there exists a collection of real numbers
 $\lambda_1>\dots>\lambda_k$, $1\leq k\leq d$, such that, for $\mu$-almost every $x\in X$, one
 has a splitting
$$M_x=E^{\mu}_{\lambda_1}(x)\oplus\dots\oplus E^{\mu}_{\lambda_k}(x)$$
 with
$$\lim\limits_{t\rightarrow\pm\infty}\frac{1}{n}\log\vert F_t(v_i)\vert_{T_t(x)}=\lambda_i$$
for every $v_i\in E_{\lambda_i}^{\mu}(x)-\{0\}$. Moreover, the subspaces $E_{\lambda_i}^{\mu}(x)$ depend measurably on $x\in X$. In the literature, the numbers $\lambda_i$ are called \emph{Lyapunov exponents} and the subspaces $E_{\lambda_i}^{\mu}(x)$ are called \emph{Oseledets subspaces}.

For the sake of convenience, one writes the list of Lyapunov exponents as $\lambda_1\geq\dots\geq\lambda_d$ (where $d$ is the dimension of the fibers of the vector bundle) by \emph{repeating}  each exponent $\lambda_i$ a number of times equal to the real or complex dimension
of the corresponding Oseledets space $E^{\mu}_{\lambda_i}(x)) \subset M_x$ (which by ergodicity
is constant almost everywhere). We will loosely refer to both the list $\lambda_1>\dots>\lambda_k$ and
$\lambda_1\geq\dots\geq\lambda_d$ as the \emph{Lyapunov spectrum}  of the linear cocycle $(F_t)_{t\in\mathbb{R}}$, although the second list also contains the information about the real or complex multiplicities of Lyapunov exponents (that is, about the real or complex dimensions of the Oseledets subspaces).

The following general facts will be relevant in this paper:
\begin{itemize}
\item the natural \emph{complexification} of any real linear cocycle has the \emph{same Lyapunov spectrum} of the original real linear cocycle; in particular, the complexified cocycle has complex multiplicities equal to the real multiplicities of the original real cocycle;
\item the Lyapunov spectrum of a real \emph{symplectic} cocycle, that is, a real cocycle preserving a family of symplectic forms on the fibers $M_x\simeq\mathbb{R}^d$, $d=2n$, of the vector bundle $\pi:M\to X$, has the form $$\lambda_1\geq\dots\geq\lambda_n\geq-\lambda_n\geq\dots\geq-\lambda_1\,;$$ in other words, the Lyapunov spectrum of a symplectic cocycle is symmetric with respect to the origin
$0\in\mathbb{R}$.
\end{itemize}

For more details on Oseledets theorem and the general theory of Lyapunov exponents, see the books \cite{BDV} and \cite{HK} (and references therein).

Coming back to the Kontsevich-Zorich cocycle, let us recall that, by Corollary~\ref{cor:universal_1},
for any (Teichm\"uller) flow-invariant Borel probability ergodic measure $\mu$ on $\cH_g^{(1)}$, the cocycle is $\log$-integrable with respect to the Hodge norm, hence it has well-defined
Lyapunov spectrum, with top exponent $\lambda_1=1$; since the cocycle on $H^1_\R$ preserves the symplectic intersection form on the ($2g$-dimensional) fibers $H^1(S,\mathbb{R})$ of $H^1_\R$, its Lyapunov spectrum is symmetric. Thus the Lyapunov spectrum of the Kontsevich-Zorich cocycle
with respect to any (Teichm\"uller) flow-invariant Borel probability ergodic measure $\mu$ on
$\cH_g^{(1)}$ has the following form:
$$
\lambda_1^\mu =1 \geq \lambda^\mu_2 \geq \dots \geq \lambda^\mu_g \geq
-\lambda^\mu_g \geq \dots \geq -\lambda^\mu_2 \geq -\lambda^\mu_1=-1\,.
$$

In particular,  the Kontsevich--Zorich spectrum always has $g$
non-negative and $g$ non-positive exponents. Let
\begin{equation}
\label{eq:no:multiplicities}
\lambda^\mu_{(1)} >  \dots > \lambda^\mu_{(n)} > -\lambda^\mu_{(n)} > \dots
 > -\lambda^\mu_{(1)}\,.
\end{equation}
be the Kontsevich--Zorich spectrum of all \textit{distinct non-zero}
Lyapunov exponents of the Hodge bundle $H^1_\R$. Applying the
Teichm\"uller flow both in forward and backward directions we get
the corresponding Oseledets decomposition
\begin{equation}
\label{eq:Oseledets:direct:sum}
E^{\mu}_{\lambda^{\mu}_{(1)}}\oplus\dots\oplus E^{\mu}_{\lambda^{\mu}_{(n)}}
\oplus E^{\mu}_{(0)}
\oplus E^{\mu}_{-\lambda^{\mu}_{(n)}}\oplus\dots\oplus E^{\mu}_{-\lambda^{\mu}_{(1)}}
\end{equation}
at $\mu$-almost every point $(S,\omega)$ of $\cH_g^{(1)}$, where
$E^{\mu}_{(0)}$ is omitted if the Lyapunov spectrum of $\mu$
does not contain zero. By definition all nonzero vectors of
each subspace $E^{\mu}_{\lambda^{\mu}_{(k)}}$ or $E^{\mu}_{(0)}$ share
the same Lyapunov exponent $\lambda^{\mu}_{(k)}$ (correspondingly
$0$) which changes sign under the time reversing.

\begin{Remark}
\label{rm:invariant:versus:SL:invariant}

By convention, when saying that a measure (function, line
subbundle, etc) is \textit{``invariant''} we mean that it is
\textit{``invariant with respect to the Teichm\"uller flow''}.
If the corresponding object is \textit{``invariant with respect
to the $\SL$-action''}, we explicitly indicate that it is
\textit{``$\SL$-invariant''}. In particular,
decomposition~\eqref{eq:Oseledets:direct:sum} is defined by any
probability measure invariant and ergodic with respect to the
Teichm\"uller flow.
\end{Remark}

\begin{Lemma}\label{l.symp-orth}
Every subspace $E^\mu_{\lambda^{\mu}_{(i)}}$ of the Oseledets direct sum
decomposition~\eqref{eq:Oseledets:direct:sum} except
$E^\mu_{(0)}$ is isotropic.
Any pair of subspaces $E^\mu_{\lambda^{\mu}_{(i)}}$, $E^\mu_{\lambda^{\mu}_{(j)}}$
such that $\lambda^{\mu}_{(j)}\neq-\lambda^{\mu}_{(i)}$ is symplectic orthogonal.
The restriction of the symplectic form to each subspace
$E^\mu_{\lambda^{\mu}_{(i)}}\oplus E^\mu_{-\lambda^{\mu}_{(i)}}$, where $i\neq 0$,
and on $E^\mu_{(0)}$ is nondegenerate.
\end{Lemma}
\begin{proof}
The absolute value of the symplectic product of any two cocycles $c_1,c_2$ in
$H^1(S,\R)$ is uniformly bounded on any compact part $\cK$ of $\cH_g^{(1)}$
by the product of their Hodge norms,
$$
|\langle c_1,c_2\rangle|\le
\const(\cK)\cdot \|c_1\|_\omega\!\cdot\!\|c_2\|_\omega
\quad
\text{ for any }\omega\in\cK\,.
$$
By ergodicity of the flow, it returns infinitely often to the compact
part $\cK$.

Consider a pair of cocycles $c_i,c_j$ such that
$c_i\in E_{\lambda^{\mu}_{(i)}}$,
$c_j\in E_{\lambda^{\mu}_{(j)}}$ and such that
$\lambda^{\mu}_{(i)}\neq-\lambda^{\mu}_{(j)}$.
By definition of $E_{\lambda^{\mu}_{(i)}}$, we have
$$
\|G_t^{KZ}(c_1)\|_\omega\cdot\|G_t^{KZ}(c_2)\|_\omega\sim
\exp\big((\lambda^{\mu}_{(i)}+\lambda^{\mu}_{(j)})\, t\big)
$$
When $\lambda^{\mu}_{(i)}+\lambda^{\mu}_{(j)}<0$ the latter expression
tends to zero when $t\to+\infty$; when
$\lambda^{\mu}_{(i)}+\lambda^{\mu}_{(j)}>0$ the latter expression tends to
zero when $t\to-\infty$. In both cases, we conclude that for
a subsequence of positive or negative times $t_k$ (chosen when
the trajectory visits the compact set $\cK$) the symplectic
product $\langle G_{t_k}^{KZ}(c_1),G_{t_k}^{KZ}(c_2)\rangle$ tends to
zero. Since the symplectic product is preserved by the flow
this implies that it is equal to zero, in particular $\langle
c_1,c_2\rangle=0$. Thus, we have proved that every subspace
$E^\mu_{\lambda^{\mu}_{(i)}}$ except $E^\mu_{(0)}$ is isotropic, and
that any pair of subspaces $E^\mu_{\lambda^{\mu}_{(i)}}$,
$E^\mu_{\lambda^{\mu}_{(j)}}$ such that $\lambda^{\mu}_{(j)}\neq-\lambda^{\mu}_{(i)}$ is
symplectic orthogonal. Hence, the cohomology space decomposes
into a direct sum of pairwise symplectic-orthogonal subspaces
$E^\mu_{(0)}$, $E^\mu_{\lambda^{\mu}_{(1)}}\oplus E^\mu_{-\lambda^{\mu}_{(1)}}$,
etc, where we couple all pairs $E^\mu_{\lambda^{\mu}_{(i)}}$ and
$E^\mu_{-\lambda^{\mu}_{(i)}}$. Since the symplectic form is
nondegenerate and the summands are symplectic-orthogonal, it is
nondegenerate on any summand.
\end{proof}

For any $k\in \{1, \dots, g\}$ such that $\lambda^{\mu}_k>\lambda^{\mu}_{k+1}\ge 0$,
let $i(k)$ be the index of the exponent $\lambda^{\mu}_k$ in the ordering without
multiplicities~\eqref{eq:no:multiplicities}, $\lambda^{\mu}_{(i(k))}=\lambda^{\mu}_k$.
Let us define the $k$-th unstable Oseledets subbundle as
$$
E^+_k:=E^{\mu}_{\lambda^{\mu}_{(1)}}\oplus\dots\oplus E^{\mu}_{\lambda^{\mu}_{(i(k))}}
$$

%--------------------------------------------------------------------
\subsection{Formulas for the Kontsevich--Zorich exponents}
\label{ssec:formulas}

Let  $V  \subseteq H^1_\R$ be any $r$-dimensional flow-invariant measurable subbundle
of  the  Hodge  bundle,  almost  everywhere defined with respect to a
flow-invariant  ergodic probability measure $\mu$ on the moduli space
$\cH_g^{(1)}$.  Let us denote
$$
\lambda_1^{V,\mu} \geq  \dots  \geq \lambda_r^{V,\mu}
$$
the  Lyapunov  spectrum  of the restriction of the Kontsevich--Zorich
cocycle   to  the  subbundle  $V\subset  H^1_\R$  with  respect  to  the
Teichm\"uller geodesic flow and the invariant measure $\mu$ on $\cH_g^{(1)}$.
Let us also denote
\begin{equation}
\label{eq:nomultiplicities:V}
\lambda_{(1)}^{V,\mu} >  \dots  > \lambda_{(s)}^{V,\mu}
\end{equation}
the Lyapunov spectrum of all distinct Lyapunov exponents and let
\begin{equation}
\label{eq:Oseledets:direct:sum:V}
V = V^\mu_1 \oplus \dots \oplus V^\mu_s
\end{equation}
be   the  corresponding  Oseledets  decomposition.  It  follows  from
Oseledets theorem that the Lyapunov exponents on $V$ form a subset of
the  Lyapunov spectrum \eqref{eq:no:multiplicities} of the cocycle on
the Hodge bundle, that is,
$$
\{\lambda_{(1)}^{V,\mu}, \dots  , \lambda_{(s)}^{V,\mu} \} \subset
\{\lambda^\mu_{(1)}, \dots, \lambda^\mu_{(n)}, -\lambda^\mu_{(1)},
\dots, -\lambda^\mu_{(n)}\} \cup \{0\}\,,
$$
and  the  Oseledets  subspaces  $V^{\mu}_1, \dots, V^{\mu}_s$ are the
non-trivial  intersections of Oseledets spaces for the cocycle on the
full  Hodge  bundle  $H^1_\R$, as in~\eqref{eq:Oseledets:direct:sum},
and the subbundle $V\subset H^1_\R$, that is,
$$
\{ V^\mu_1, \dots, V^\mu_s\} =
\bigcup_{i=1}^n \left\{ E^\mu_{\lambda^{\mu}_{(i)}} \cap V,
E^\mu_{-\lambda^{\mu}_{(i)}} \cap V\right\} \cup
\left\{ E^\mu_{(0)} \cap V\right\}\!
 \setminus \big\{\!\{0\}\!\big\},
$$
where  $E^{\mu}_{(0)}$  is  omitted if the Lyapunov spectrum of $\mu$
does not contain zero.

For any $k\in\{1, \dots, g\}$ denote by $G_k(H^1_\R)$ the total space
of  the Grassmannian bundle of isotropic $k$-dimensional subspaces of
the  real Hodge bundle $H^1_\R$. Let us denote by $\Cal N_k(\mu)$ the
space  of  all  Borel  probability  measures  on  $G_k(H^1_\R)$ which project
onto any probability measure, \textit{absolutely continuous} with respect to the
flow-invariant ergodic probability measure $\mu$ on $\cH_g^{(1)}$ under the
canonical projection.

The   Kontsevich--Zorich cocycle on the Hodge bundle $H^1_\R$ preserves
the  symplectic  form in the fibers. Hence, it defines a natural action on the
Grassmannian bundle $G_k(H^1_\R)$. Since the subbundle $V$   is   flow-invariant,
the  measurable Grassmannian subbundle   $G_k(V)$   is  also flow-invariant. Let  $\Cal  I_k(\mu)\subset \Cal N_k(\mu)$ be the subset of those
measures  in $\Cal N_k(\mu)$ that are invariant with respect to the
Kontsevich--Zorich  cocycle  $\{G^{KZ}_t\}$  on  $G_k(H^1_\R)$.
Note that all measures $\nu\in \Cal  I_k(\mu)$ project onto the flow-invariant ergodic
probability measure $\mu$ on $\cH_g^{(1)}$ under the canonical projection.
It follows that the set $\Cal  I_k(\mu)$ is a compact subset of the set
of all Borel probability measures on the locally compact space
$G_k(H^1_\R)$ endowed with the weak-star topology.

Consider  a  subbundle  $V$  of  the  Hodge  bundle
satisfying    the    properties    stated   at   the   beginning   of
section~\ref{ssec:formulas}. Ergodicity of the measure and invariance
of  the  intersection  form with respect to the flow implies that the
restrictions  of the symplectic form to $\mu$-almost any fiber of $V$
have  the  same  rank.  We  do  not  exclude  the  situation when the
resulting  form  is  degenerate.  Denote  by  $2p$  the  rank  of the
restriction  of  the  symplectic  form  to  $V$; denote by $l$ the
difference $l=r-p$ between the dimension of the fiber of $V$ and $p$.
The restriction of the   symplectic   form  to $\mu$-almost  any  fiber  of  $V$
is nondegenerate  if and only if $l=p$, otherwise the form is degenerate
and  $p<l$.  Note  also  that  $l$  is  the  maximal  dimension of an
isotropic subspace in the fiber of $V$.

Let $G_k(V)$ denote the total space of the Grassmannian bundle of all
$k$-dimensional  isotropic  subspaces contained in the fibers of $V$.
By  definition,  for  $\mu$-almost  all $\omega \in \cH_g^{(1)}$, the fiber
$G_k(V)_\omega$  of  the bundle $G_k(V)$ is equal to the space of all
$k$-dimensional isotropic subspaces of $V_\omega$.
Clearly,  for  each individual fiber, one has $G_k(V)_\omega\subseteq
G_k(H^1_\R)_\omega$,   so   the  Grassmanian  bundle  $G_k(V)$  is  a
measurable subbundle of $G_k(H^1_\R)$.

Let  $\Cal  N^V_k(\mu)\subset  \mathcal N_k(\mu)$ be the subset of all
Borel probability measures on $G_k(H^1_\R)$ essentially supported on
a subset of the measurable Grassmannian bundle $G_k(V)\subset G_k(H^1_\R)$.
Note that there is no flow-invariance assumption in the definition of
the  sets  $\Cal  N_k(\mu)$  and  $\Cal N^V_k(\mu)$. In fact, we will
prove existence of flow-invariant probability measures essentially supported
on the Grassmannian $G_k(V) \subset G_k(H^1_\R)$, which project to any
given flow-invariant ergodic probability measure $\mu$ on $\cH_g^{(1)}$, in
Lemma~\ref{lemma:inv_meas_Grass}   below.

\smallskip
We   start  with  an elementary preparatory Lemma.

\begin{lemma}
\label{lemma:bundle_meas}
Let $\bar \mu$ be any Borel probability measure absolutely continuous
with respect to $\mu$ on $\cH_g^{(1)}$. Let $V\subset H^1_\R$ be a measurable
subbundle defined $\bar \mu$-almost everywhere.
For any $k$ such that $G_k(V)$ is non-empty (i.e., for any $k\le l$),
there exist measures on $G_k(V)$ which project onto  $\bar \mu$
on $\cH_g^{(1)}$, under the canonical projection. In particular,
the space $\Cal N^V_k(\mu)$ is also non-empty.
\end{lemma}
\begin{proof}
Any Borel measurable bundle can be trivialized on the complement of a
subset  of  measure zero with respect to any given Borel measure. So,
there  exists  a  set  $\Cal E_g \subset \cH_g^{(1)}$ of full $\bar \mu$-measure
such   that   the   restriction  $V\vert_{\Cal  E_g}$  is  measurably
isomorphic  to  the  product  bundle $\Cal E_g \times \R^r$. It follows
that the restriction $G_k(V)\vert_{\mathcal E_g}$ of the Grassmannian
bundle of isotropic subspaces is isomorphic to the product bundle $\Cal
E_g \times G_k(\R^r)$. By definition,  for any Borel probability measure
$\eta$  on  $G_k(\R^r)$, the product measure $\bar\mu \times \eta$
on the space $\Cal E_g \times G_k(\R^r)$ induces (by push-forward under
the bundle isomorphism) a Borel probability measure $\nu$ on $G_k(V)$,
which projects onto $\bar \mu$ on $\cH_g^{(1)}$ under the canonical projection.
\end{proof}
Note that we do not assume that the measure $\nu\in\Cal N^V_k(\mu)$ constructed above
is flow-invariant. Let  $\Cal  I^V_k(\mu)=\Cal I_k(\mu) \cap \Cal N^V_k(\mu) $ be
the subset of those measures  in $\Cal N^V_k(\mu)$ that are invariant with respect to the
restriction of the Kontsevich--Zorich  cocycle  $\{G^{KZ}_t\}$  to  $V$.
We are going to show that
the  set  $\Cal  I^V_k(\mu)$  is  non-empty  whenever  the  set $\Cal
N^V_k(\mu)$ is non-empty.

For any $\nu \in \Cal N_k(\mu)$, let $I(\nu)\subseteq\Cal N_k(\mu)$
be  the  set of all \textit{weak limits} of the family of probability
measures
\begin{equation}
\label{eq:I(nu)}
\left\{ \nu_T:=\frac{1}{T} \int_0^T
   (G^{KZ}_t )_\ast (\nu)  dt\,\Big\vert\,\, T>0\right\}
\end{equation}
in the space of all Borel probability measures on the locally compact space
$G_k(H^1_\R)$.  Here by a ``weak limit'' we mean any limit in the weak-star topology
along some diverging sequence of positive times $T_1,T_2, \dots, T_n, \dots$.
Of course, in general distinct sequences may lead to different weak limits.

\begin{lemma}
\label{lemma:inv_meas_Grass}
For  any  $\nu  \in \Cal N^V_k(\mu)$, the set $I(\nu)$ is a non-empty
compact  subset  of the set $\Cal I^V_k(\mu)$ of probability measures
in $\Cal N^V_k(\mu)$ which are invariant under the Kontsevich--Zorich
cocycle.
\end{lemma}
\begin{proof}
The  fiber  of  the  Grassmannian  bundle  $G_k(H^1_\R)$ is the space
$G_k(\R^{2g})$  of  all  isotropic  subspaces  of  dimension $k$ in a
symplectic  space  of  dimension $2g$, so that the fiber is a compact
manifold.  Also,  the  full  Grassmannian  bundle  $G_k(H^1_\R)$ is a
\textit{continuous} bundle. Thus, the space of all Borel measures
on $G_k(H^1_\R)$ of finite total mass is a Montel space (in the
sense that all closed bounded sets are compact) with respect to
the weak-star topology. Hence,  for any diverging sequence of (positive)
times $(T_n)$ we can extract from the sequence $\{\nu_{T_n}\}$ of measures in
$\Cal  N_k(\mu)$ a  converging  subsequence. The limit
measure is a probability measure since, by the Birkhoff ergodic theorem,
the projection of the sequence $\{\nu_{T_n}\}$ under the canonical
projection converges weakly to the flow-invariant ergodic probability
measure $\mu$ on $\cH_g^{(1)}$. It follows that the  subset
$I(\nu)$  of  $\Cal  I_k(\mu)$ is nonempty.  Since  the set of all accumulation
points of any given set in  a  topological space  is closed and the set
$\Cal  I_k(\mu)$ is compact, we get that $I(\nu)$ is a nonempty compact  subset
of  $\Cal  I_k(\mu)$. As  the flow $\{G_t^{KZ}\}$  is  continuous on $G_k(H^1_\R)$,
by the usual (relative) Bogolyubov--Krylov  argument (see, e.g., page 135 of the book \cite{HK}),  one  has  that any measure
$\hat\nu\in I(\nu)$  is  $G^{KZ}_t$-invariant. We reproduce the argument below
for the convenience of the reader. For any (fixed) $s\in \R$ and for all $T>0$, we have
$$
\begin{aligned}
(G^{KZ}_s)_\ast (\nu_T) -\nu_T &= \frac{1}{T}\left( \int_0^T (G^{KZ}_{t+s})_\ast(\nu)dt
- \int_0^T (G^{KZ}_t)_\ast (\nu)dt \right)\\
&= \frac{1}{T}\left(\int_s^{T+s} (G^{KZ}_t)_\ast(\nu)dt
- \int_0^T (G^{KZ}_t)_\ast(\nu)dt \right) \\
&=  \frac{1}{T}\left(\int_T^{T+s} (G^{KZ}_t)_\ast(\nu)dt  - \int_0^s (G^{KZ}_t)_\ast(\nu)dt \right) \,.
\end{aligned}
$$
It follows that (for fixed $s\in \R$) the total mass $\Vert (G^{KZ}_s)_\ast (\nu_T) -\nu_T \Vert$ of
the signed measure
$(G^{KZ}_s)_\ast (\nu_T) -\nu_T$ converges to  zero as $T\to \infty$. In fact,
$$
\Vert (G^{KZ}_s)_\ast (\nu_T) -\nu_T \Vert \leq \frac{2s}{T}  \to 0\,.
$$
Let then $\{\nu_{T_n}\}$ be a sequence converging weakly to a measure $\hat \nu\in I(\nu)$.
Since  the map $G^{KZ}_s$ is continuous on $G_k(H^1_\R)$, the sequence
$\{(G^{KZ}_s)_\ast(\nu_{T_n})\}$ converges weakly to $(G^{KZ}_s)_\ast(\hat \nu)$, hence,
for all $s\in \R$, we have
$$
(G^{KZ}_s)_\ast (\hat \nu) -\hat \nu =\lim_{n\to+\infty}  (G^{KZ}_s)_\ast (\nu_{T_n}) -\nu_{T_n} = 0\,.
$$
We conclude that any measure $\hat\nu\in I(\nu)$  is  $\{G^{KZ}_t\}$-invariant  as stated.

It  remains  only  to show that $I(\nu)\subset \Cal N^V_k(\mu)$, that is,  any  measure
$\hat\nu\in  I(\nu)$  is essentially supported on $G_k(V)$, in the sense that $\hat\nu(G_k(V))=1$.

Let $\bar{\mu}$ denote the projection of $\nu\in\Cal N_k(\mu)$ on $\Cal H_g$. By definition, $\bar{\mu}$ is absolutely continuous with respect to $\mu$: in particular, given $\varepsilon>0$, we can choose $\delta(\varepsilon)>0$ such that $\mu(A)<\delta(\varepsilon)$ implies $\bar{\mu}(A)<\varepsilon$ for all measurable $A\subset\Cal H_g$. On the other hand, by Luzin's theorem (see, e.g., page 2 of the book \cite{Mane}), given $\epsilon>0$, we can fix $K_\epsilon\subset\cH_g^{(1)}$ a compact subset such that $\mu(K_\epsilon)>1-\delta(\epsilon)$ and $V|_{K_{\epsilon}}$ is a \textit{continuous} subbundle of the measurable bundle $V$. In particular, $G_k(V)|_{K_{\epsilon}}$ is a compact subset of $G_k(H^1_\R)$. Let $\phi$ be any real-valued continuous function on $G_k(H^1_\R)$ such that $0\leq \phi\leq 1$ and $\phi$ is identically
equal to $1$ on $G_k(V)|_{K_{\epsilon}}$. Since, by definition, $\nu\in\Cal N^V_k(\mu)$ means that $\nu$ is supported on $G_k(V)$ and it projects to $\bar{\mu}\ll\mu$ on $\cH_g^{(1)}$, it follows from our choice of $\delta(\varepsilon)>0$ above that, for any $T>0$,
$$\int_{G_k(H^1_\R)} \phi d\nu_T \geq 1-\epsilon\,.$$
Hence, for any weak limit $\hat\nu\in I(\nu)$, one has
$$\int_{G_k(H^1_\R)} \phi d\hat\nu\geq 1-\epsilon\,.$$
Because this holds for every $\phi$ as above, we conclude that
$$\hat\nu\left(G_k(V)|_{K_{\epsilon}}\right) \geq 1-\epsilon$$
and hence $\hat\nu(G_k(V))=1$, as claimed.
\end{proof}

\smallskip
For any measure $\nu \in \Cal N^V_k(\mu)$ we define below the average
Lyapunov  exponent  $\Lambda^{(k)}(\nu)$ over the Grassmannian bundle
$G_k(V)$. Let us consider  an  isotropic subspace $I$ in the fiber of the Hodge
bundle over  some  point  $\omega \in \cH_g^{(1)}$.  Let  $\{c_1, \dots, c_k\}$ and
$\{c'_1, \dots,  c'_k\}$  be  a  pair  of  bases  in it. Let us consider the parallel
transport of these vectors to a neighborhood of $\omega$. Clearly, at
any point of the neighborhood of $\omega$ the polyvectors $c_1 \wedge
\dots  \wedge  c_k$  and  $c'_1  \wedge  \dots  \wedge  c'_k$  remain
proportional  with  the  same  constant coefficient. Hence, the Hodge
norms  of  these  polyvectors are proportional with the same constant
coefficient.  This  implies,  in  particular,  that  the  logarithmic
derivative  of the Hodge norm of a polyvector along the Teichm\"uller
geodesic  flow,  depends only on the isotropic subspace $I$,
$$
\Cal L \log  \Vert c_1 \wedge \dots \wedge c_k  \Vert_\omega=
\Cal L \log  \Vert c'_1 \wedge \dots \wedge c'_k  \Vert_\omega\,.
$$
Thus,  slightly  abusing  notations,  in the following we  shall  sometimes denote the
logarithmic derivative as above by
$$
\frac{d\log\| G^{KZ}_t(\omega, I) \|}{dt} \vert_{t=0}:=\Cal L\log\|I\|_\omega:=
\Cal L \log  \Vert c_1 \wedge \dots \wedge c_k  \Vert_\omega\,.
$$
The  \mbox{Oseledets} theorem establishes that, for $\mu$-almost every Abelian
differential $\omega \in \cH_g^{(1)}$ on a Riemann surface $S$ and for every polyvector
$c_1\wedge\dots\wedge c_k$ in $\Lambda^k(H^1(S,\R))$, the Lyapunov exponent
$$
\lambda^\mu_\omega(c_1, \dots,c_k) := \lim_{T\to +\infty}
\frac{1}{T}  \log \Vert  G^{KZ}_T (c_1\wedge\dots\wedge c_k)\Vert
$$
is  well defined. Let us assume that the vectors $c_1,\dots,c_k$ form a
basis  of  an  isotropic subspace $I_k\subset V_\omega$, and that the
norm   is   the   Hodge   norm.   Then, according to the above discussion
about the logarithmic derivative of the Hodge norm of an isotropic
subspace, for $\mu$-almost all $\omega\in \cH_g^{(1)}$, the Lyapunov exponent
$\lambda_\omega^\mu(c_1, \dots,c_k)$ depends only on the
isotropic subspace $I_k\subset V$ and can be written as follows:
\begin{multline*}
\lim_{T\to +\infty}
\frac{1}{T}  \log \Vert  G^{KZ}_T (c_1\wedge\dots\wedge c_k)\Vert\\
=\lim_{T\to +\infty}
\frac{1}{T} \int_0^T
\frac{d}{dt}\log \Vert  G^{KZ}_t (c_1\wedge\dots\wedge c_k)\Vert\,dt\\
=\lim_{T\to +\infty}
\frac{1}{T} \int_0^T
\frac{d}{dt}\log \Vert  G^{KZ}_t (\omega,I_k)\Vert\,dt\,.
\end{multline*}
By  Corollary~\ref{cor:universal_2}  the  function
$|\frac{d}{dt}\log\Vert   G^{KZ}_t   (\omega,I_k)\Vert|$
is  bounded  above by  $k\in \N$ for any point $(\omega,I_k)$  of the Grassmannian
$G_k(H^1_\R)$.   Hence,   for   any  $(\omega,I_k) \in G_k(H^1_\R)$
 and any $T>0$, we get the following uniform estimate:

\begin{equation}
\label{eq:uniform:estimate:integral}
-k\le \frac{1}{T} \int_0^T
\frac{d}{dt}\log \Vert  G^{KZ}_t (\omega,I_k)\Vert\,dt\le k\,.
\end{equation}

By averaging  the Lyapunov  exponent  $\lambda_\omega^\mu(c_1, \dots,c_k)$ over $G_k(V)$
with respect to measure $\nu\in\mathcal N^V_k(\mu)$, we define the \textit{average Lyapunov exponent}:

\begin{equation}
\begin{aligned}
\label{eq:av_Lyap_exp}
\Lambda^{(k)}(\nu)&:=\int_{G_k(V)} \lim_{T\to +\infty} \left(\frac{1}{T} \int_0^T
\frac{d}{dt}\log \Vert  G^{KZ}_t (\omega,I_k)\Vert\,dt \right)d\nu\\
&=\lim_{T\to +\infty}\int_{G_k(V)}\left(\frac{1}{T} \int_0^T
\frac{d}{dt}\log \Vert  G^{KZ}_t (\omega,I_k)\Vert\,dt\right) d\nu\\
&=\lim_{T\to +\infty}\frac{1}{T}\int_0^T\int_{G_k(V)}
\frac{d}{dt}\log \Vert  G^{KZ}_t (\omega, I_k)\Vert  d\nu dt \,.
\end{aligned}
\end{equation}

Note that  the interchange of the limit with the integral and the change in the
order of integration (Fubini theorem) in formula \eqref{eq:av_Lyap_exp} are
justified by the uniform upper bound established above in
formula~\eqref{eq:uniform:estimate:integral}. Also, note that the definition of $\Lambda^{(k)}(\nu)$ doesn't assume the flow-invariance of $\nu\in \mathcal N^V_k(\mu)$.

\smallskip
Suppose  now  that  $\mu$  is  an ergodic $\SL$-invariant probability
measure  on  $\cH_g^{(1)}$ (and not just flow-invariant as above).
 Let  $V\subset  H^1_\R$  be  any $\SO$-invariant measurable subbundle
defined  $\mu$-almost everywhere. A measure $\nu \in \Cal N^V_k(\mu)$
will  be called \textit{$\SO$-invariant} if it is invariant under the
natural  lift  of  the  action of the group $\SO$ to the Grassmannian
bundle $G_k(V)$. The subset $\Cal O^V_k(\mu) \subset \Cal N^V_k(\mu)$
consisting  of  $\SO$-invariant  probability  measures  is  non-empty
whenever  $\Cal  N^V_k(\mu)$  is.  In  fact,  since by assumption the
measure   $\mu$   is   $\SL$-invariant   and   the   bundle   $V$  is
$\SO$-invariant,  and since $\SO$ is an amenable (compact) group, the
$\SO$-average  of  any  probability measure in $\Cal N^V_k(\mu)$ is a
well-defined probability measure in $\Cal O^V_k(\mu)$.

\begin{theorem}
\label{thm:partialsum} Let $\mu$ be any $\SL$-invariant Borel probability ergodic
measure on the moduli space $\cH_g^{(1)}$ of normalized Abelian differentials. Let
$V\subset H^1_\R$ be any $\SL$-invariant measurable subbundle defined $\mu$-almost everywhere.
For any  $\SO$-invariant probability measure $\nu \in {\Cal O}^V_k(\mu)$, the following formula holds:
\begin{equation}
\label{eq:partial:sum:V}
\Lambda^{(k)}(\nu)   =
 \int_{G_k(V)}  \Phi_k(\omega, I_k ) d \hat \nu  \,, \quad \text{ \rm for any }
 \, \hat \nu \in I(\nu)\,.
\end{equation}
\end{theorem}
\begin{proof} Let $D$ denote the Poincar\'e disk and let  $(t,\theta)\in D$  denote  the geodesic
polar coordinates on the Teichm\"uller   disk  $\SO\backslash  \SL\cdot  \omega$  centered  at
an Abelian differential $\omega\in \cH_g^{(1)}$ on a Riemann surface $S$, defined as follows.
Let $\SO:=\{R_\theta\,  \vert\,  \theta \in [0, 2\pi)\}$ and let
$$
(S_{(t,\theta)}, \omega_{(t,\theta)}):=   (G_t \circ R_\theta)(S,\omega) \,, \quad \text{ \rm for all } (t,\theta)
\in \R^+\times [0,2\pi)\,.
$$
For $\mu$-almost all Abelian differential $\omega\in \cH_g^{(1)}$ on a Riemann surface $S$ and for any  $k$-dimensional isotropic  subspace  $I_k  \subset  V_{\omega}$, let $\{c_1,  \dots,  c_k\} \subset I_k$
be any Hodge orthonormal basis  at  $(S,  \omega)$  and  let  $\Vert  c_1  \wedge \dots \wedge c_k\Vert_{(\omega,t,\theta)}$ denote the Hodge norm of the polyvector $c_1   \wedge  \dots  \wedge  c_k\in \Lambda^k(H^1(S,\R))$  at  the  point  $(S_{(t,\theta)}, \omega_{(t,\theta)})\in  \SO\backslash  \SL \cdot \omega$ of coordinates
$(t,\theta) \in D$.

From    the    variational    formulas    of Lemma~\ref{lemma:var2} for the hyperbolic Laplacian
of the Hodge norm of a polyvector on a Teichm\"uller disk, by the Green formula (or,  equivalently,
by explicit  integration  of the Poisson equation for the  hyperbolic Laplacian on the Poincar\'e disk $D$), we derive the  formula stated below (see formula $(5.10)$ in \cite{Forni2}).   Let $D_t$ denote the disk of hyperbolic  radius  $t>0$  centered  at  the origin of the Poincar\'e disk,  let  $\vert D_t\vert$ denote its hyperbolic area and let $\Cal A_P$  denote  the  Poincar\'e  area  form.  Let  us  also adopt  the
convention, introduced above, on the logarithmic derivative of the Hodge norm of isotropic subspaces.
We have

\begin{multline}
\label{eq:integral1}
\frac{1}{2\pi} \int_0^{2\pi}  \frac{\partial}{\partial t} \log
   \Vert (G^{KZ}_t \circ R_\theta) (\omega, I_k)  \Vert \,d\theta\\
   =\frac{1}{2\pi} \int_0^{2\pi}  \frac{\partial}{\partial t}
\log \Vert c_1 \wedge \dots \wedge c_k  \Vert_{(\omega,t,\theta)} \,d\theta \\
=\frac{\tanh t}{\vert D_t\vert} \int_{D_t}  (\Phi_k \circ G^{KZ}_\tau \circ R_\theta )
(\omega, I_k)\,  d\Cal A_P(\tau,\theta) \,.
\end{multline}

 Let  us  now integrate formula~\eqref{eq:integral1} over the Grassmannian
$G_k(V)$ with respect to the $\SO$-invariant probability measure $\nu
\in     \mathcal     O^V_k(\mu)$.     Note    that    by    Corollary
\ref{cor:universal_2} all the integrands are uniformly bounded, hence
it  is  possible to exchange the order of integrations.

On the left-hand side (LHS for short) of formula~\eqref{eq:integral1},  by the $\SO$-invariance of the measure $\nu$ on $G_k(V)$, we compute as follows:

\begin{equation}
\label{eq:LHS}
\begin{aligned}
&\int_{G_k(V)} \frac{1}{2\pi} \int_0^{2\pi}   \frac{\partial}{\partial t}
\log \Vert (G^{KZ}_t \circ R_\theta) (\omega, I_k)  \Vert \, d\theta \, d\nu \\
&=\frac{1}{2\pi} \int_0^{2\pi}  \int_{G_k(V)}  \frac{\partial}{\partial t}
\log \Vert (G^{KZ}_t \circ R_\theta) (\omega, I_k)  \Vert \, d\nu \, d\theta \\
& =  \frac{1}{2\pi}  \int_0^{2\pi}  \int_{G_k(V)}  \frac{d}{dt}
\log \Vert G^{KZ}_t  (\omega, I_k)  \Vert \,d\nu d\theta \\
& =  \int_{G_k(V)}
\frac{d}{dt}  \log \Vert G^{KZ}_t  (\omega, I_k)  \Vert \,d\nu  \,.
\end{aligned}
\end{equation}
On the right-hand side (RHS for short) of formula~\eqref{eq:integral1} we compute as follows. Let
us  recall  that  the  Poincar\'e  area form can be written as $d\Cal
A_P(\tau,\theta)  =  d(\sinh^2  \tau)  d\theta$  in geodesic polar
coordinates  $(\tau,\theta)  \in  \R^+  \times  [0,  2\pi)$ and, as a
consequence, the Poincar\'e area of the disk $D_t$ of geodesic radius
$t>0$ is $\vert D_t\vert= 2\pi \sinh^2 t$.

By taking into account the uniform bound~\eqref{eq:bound:for:Phi}
for $|\Phi_k(\omega, I_k)|$, and by the   above  elementary  formulas  of
hyperbolic  geometry, and  the $\SO$-invariance  of  the  measure  $\nu$
on $G_k(V)$, we proceed as follows:

\begin{equation}
\label{eq:RHS}
\begin{aligned}
&  \frac{1}{2\pi} \int_{G_k(V)}
\int_{D_t}  (\Phi_k \circ G^{KZ}_\tau\circ R_\theta )
(\omega, I_k)\,  d\Cal A_P(\tau,\theta)\,  d\nu\\
&=  \frac{1}{2\pi} \int_0^{2\pi}  \int_{G_k(V)}
\int_0^t (\Phi_k \circ G^{KZ}_\tau \circ R_\theta)
(\omega, I_k)  d(\sinh^2 \tau)\,  d\nu  d\theta\\
&=\int_{G_k(V)}  \int_0^t  (\Phi_k \circ G^{KZ}_\tau)
(\omega, I_k)  d(\sinh^2 \tau)\, d\nu\,.
\end{aligned}
\end{equation}

To sum up our computations so far, by  integration of formula~\eqref{eq:integral1} over the Grassmannian $G_k(V)$ with respect to the $\SO$-invariant probability measure $\nu \in
\mathcal O^V_k(\mu)$ we have:
\begin{multline}
\label{eq:integral2}
 \int_{G_k(V)} \, \frac{d}{dt} \log \Vert G^{KZ}_t (\omega, I_k) \Vert \, d \nu\\
  =\frac{\tanh t}{\sinh^2 t}\, \int_{G_k(V)}  \int_0^t \Phi_k \circ G^{KZ}_\tau
 \, d (\sinh^2 \tau)\, d\nu    \,.
\end{multline}
Let us then average the above formula over the interval $[0,T]\subset
\R$  and take the limit as $T \to +\infty$. The average of the LHS of
formula~\eqref{eq:integral2}  converges, by the definition in formula
\eqref{eq:av_Lyap_exp},  to  the  average  Lyapunov  exponent  of the
$\SO$-invariant measure $\nu \in \mathcal O^V_k(\mu)$, that is,
\begin{equation}
 \label{eq:LHS_lim}
\Lambda^{(k)} (\nu) =
\lim_{T\to +\infty} \frac{1}{T} \int_0^T \int_{G_k(V)} \frac{d}{dt}
\log \Vert G^{KZ}_t (\omega, I_k) \Vert \, d \nu\, dt \,.
\end{equation}
We claim that for any probability measure $\hat \nu \in I(\nu)$ there
exists  a  diverging  sequence  $\{T_n\}$  such that the average over
$[0,T_n]$  of  the RHS  of  formula~\eqref{eq:integral2} converges to the
integral
\begin{equation}
\label{eq:RHS_lim}
 \int_{G_k(V)} \Phi_k (\omega,I_k)  d\, \hat \nu\,,
\end{equation}
so that, taking into account the limit in formula~\eqref{eq:LHS_lim},
the  theorem  follows  from  formula  \eqref{eq:integral2}. The above
claim  is  proved as follows. For any continuous function $\phi$ with
compact support on $G_k(H^1_\R)$, the function
$$
 \frac{\tanh t}{\sinh^2 t}    \int_0^t \phi \circ G^{KZ}_\tau
  d (\sinh^2 \tau)   -  \phi \circ  G^{KZ}_t
$$
converges to zero uniformly as $t\to + \infty$. In fact, the hyperbolic tangent
converges  to $1$, the function $\phi$ is  uniformly
continuous, and for any $\epsilon >0$, the mass assigned by
 the probability measure $d(\sinh^2 \tau)/\sinh^2 t$ over $[0,t]$  to the interval $[0, t-\epsilon]$
converges to zero as $t\to +\infty$.

 It follows that the measure
$$
\frac{1}{T} \int_0^T \frac{\tanh t}{\sinh^2 t}    \int_0^t ( G^{KZ}_\tau)_\ast (\nu)
  d (\sinh^2 \tau) \, - \,  \frac{1}{T} \int_0^T  ( G^{KZ}_t)_\ast (\nu)
$$
converges to zero weakly as $T\to +\infty$.
Thus for any $\hat \nu \in I(\nu)$  there exists a diverging sequence $\{T_n\}$ such that the
sequence of measures
 $$
 \frac{1}{T_n} \int_0^{T_n} \frac{\tanh t}{\sinh^2 t}    \int_0^t ( G^{KZ}_\tau)_\ast (\nu)
  d (\sinh^2 \tau)
 $$
 converges weakly to the measure $\hat \nu$ on $G_k(H^1_\R)$, essentially supported on
 $G_k(V) \subset G_k(H^1_\R)$. Since the function $\Phi_k$ is continuous
 and bounded on $G_k(H^1_\R)$ it follows that the average over $[0,T_n]$ of  the RHS of
 formula~\eqref{eq:integral2} converges to the integral in formula~\eqref{eq:RHS_lim},
 as claimed, and the proof of the theorem is completed.
  \end{proof}

For any $k\in \{1, \dots, r-1\}$ such that $\lambda_k^{V,\mu}  > \lambda_{k+1}^{V,\mu}$, let $j(k)$ be the index such that $\lambda_{(j(k))}^{V,\mu}= \lambda_k^{V,\mu}$. Let us define
$$
V^+_k := V^\mu_1 \oplus \dots \oplus V^\mu_{j(k)}\,.
$$
In general, the subbundle $V^+_k$ does not need to be a  bundle of isotropic
subspaces, that is, a measurable section of the Grassmannian $G_k(V)$. However,  note
that if $\lambda_{k+1}^{V,\mu}\geq 0$, then the bundle $V^+_k$  is a subbundle of the unstable Oseledets bundle which is isotropic (since the Kontsevich--Zorich cocycle is symplectic),
hence it is itself isotropic.

\begin{corollary}
\label{cor:partialsum1}
Let $\mu$ be any $\SL$-invariant Borel probability ergodic measure on the moduli space $\cH_g^{(1)}$ of normalized Abelian differentials. Let $V\subset H^1_\R$ be any $\SL$-invariant measurable subbundle
defined $\mu$-almost everywhere.  Assume that there exists $k\in\{1, \dots, \text{dim}(V)-1\}$
such that $\lambda_k^{V,\mu}  > \lambda_{k+1}^{V,\mu}$ and that the subbundle
$V^+_k$ is a bundle of $k$-dimensional isotropic subspaces (that is, it defines a measurable section of the Grassmanian $G_k(V)$). Then the following formula holds:
\begin{equation}
\label{eq:partial:sum:equals:Pi:k1}
\lambda^{V,\mu}_1 + \dots +\lambda^{V,\mu}_k =
 \int_{\cH_g^{(1)}}  \Phi_k\left(\omega,V^+_k(\omega)\right) d \mu(\omega) \,.
\end{equation}
\end{corollary}
\begin{proof}  Let $\nu_k\in \mathcal O^V_k(\mu)$ be an $\SO$-invariant probability
measure on $G_k(V)$ such that all of its conditional measures on the
fibers $G_k(V)_\omega$ of the Grassmannian bundle are equivalent to the Lebesgue
measure, for $\mu$-almost all $\omega\in \cH_g^{(1)}$. Theorem~\ref{thm:partialsum} in this special
case  implies formula~\eqref{eq:partial:sum:equals:Pi:k1}. In fact, by the assumption that
$\lambda_k^{V,\mu}  > \lambda_{k+1}^{V,\mu}$  the family of measures given in formula~\eqref{eq:I(nu)} converges to the unique probability measure $\hat \nu_k$  on $G_k(V)$  given by the condition
that for  $\mu$-almost all $\omega\in \cH_g^{(1)}$ the conditional measure $\nu_k \vert G(V)_\omega$
is the Dirac measure  at  the  point  $V^+_k(\omega)$ (in other terms, the measure
$\nu_k$  is  defined  as  the  push-forward  of  the measure $\mu$ on $\cH_g^{(1)}$  under  the section
$V^+_k: \cH_g^{(1)} \to G_k(V)$). In other words, the set $I(\nu_k)$ of all weak limits of the family of measures given in formula~\eqref{eq:I(nu)} is equal to $\{\hat \nu_k\}$. By the Oseledets theorem, the average Lyapunov exponent $\Lambda^{(k)}(\nu_k)$  of the Kontsevich--Zorich cocycle  with respect to the measure
$\nu_k$ on the bundle $G_k(V)$ is given by the formula
$$
\Lambda^{(k)}(\nu_k)= \lambda^{V,\mu}_1 + \dots +\lambda^{V,\mu}_k\,.
$$
Thus formula~\eqref{eq:partial:sum:equals:Pi:k1} is indeed a particular case of
formula~\eqref{eq:partial:sum:V}\,.
\end{proof}

In the particular case of the full Hodge bundle we derive below a result first proved in \cite{Forni2},
Corollary 5.5, for the canonical absolutely continuous invariant measures on connected components
of strata of the moduli space.

 \begin{corollary}
\label{cor:partialsum2}
Let $\mu$ be any $\SL$-invariant Borel probability ergodic
measure on the moduli space $\cH_g^{(1)}$ of normalized Abelian
differentials. Assume that there exists  $k\in\{1, \dots, g-1\}$ such that
$\lambda_k^{\mu}  > \lambda_{k+1}^{\mu} \geq 0$. Then the following formula holds:
\begin{equation}
\label{eq:partial:sum:equals:Pi:k2}
\lambda^{\mu}_1 + \dots +\lambda^{\mu}_k =
 \int_{\cH_g^{(1)}}  \Phi_k\left(\omega,E^+_k(\omega)\right) d \mu(\omega) \,.
\end{equation}
\end{corollary}

%Corollary \ref{cor:partialsum2} was proved in \cite{Forni2},
%Corollary 5.5, for the canonical absolutely continuous
%invariant measures on connected components of strata of the
%moduli space.

%\smallskip
By Remark ~\ref{rm:Phi:g}, Corollary \ref{cor:partialsum2} for $k=g$ holds
without any assumptions on the Lyapunov exponents and provides a version of the Kontsevich
formula for their sum  (see \cite{Kontsevich} and \cite{Forni2}, Corollary 5.3):

\begin{corollary}
\label{cor:KZformula}
Let $\mu$ be any $\SL$-invariant Borel probability ergodic
measure on the moduli space $\cH_g^{(1)}$ of normalized Abelian
differentials. The following formula holds:
\begin{equation}
\label{eq:sum:all}
\lambda^\mu_1 + \dots +\lambda^\mu_g=
\int_{\cH_g^{(1)}} (\Lambda_1+ \dots + \Lambda_g) d\mu \,.
\end{equation}
\end{corollary}

%--------------------------------------------------------------
\subsection{Reducibility of the second fundamental form}
\label{ssec:reduc}

Let $V\subset H^1(S,\R)$ be a subspace invariant under the
Hodge operator. Since for any nonzero $c\in V$ one has
$\|c\|=\langle c, \ast c\rangle>0$, this implies that $c$ cannot be
symplectic-orthogonal to $V$. Thus, invariance of $V$ under the
Hodge star-operator implies, in particular, that restriction
of the symplectic form to $V$ is nondegenerate, in particular,
$V$ is even-dimensional.
For any Hodge star-invariant subspace $V\subseteq H^1(S,\R)$,
let us define $V^{1,0}\subset H^{1,0}(S)$ and $V^{0,1}\subset H^{0,1}(S)$
to be the subspaces of cohomology classes of all holomorphic,
respectively, anti-holomorphic forms $\omega$ such that
$[\Re(\omega)] \in V$. Invariance of $V$ under the Hodge
operator implies that the sets $V^{1,0}$ and $V^{0,1}$
are indeed complex vector spaces, that $V_\C=V^{1,0}\oplus V^{0,1}$ and that
$\overline{V^{1,0}} = V^{0,1}$. In particular, $\dim_{\R} V=2\dim_{\C} V^{1,0}$.
Let us denote by $H_\omega\vert_{V^{1,0}}$ and $B_\omega\vert_{V^{1,0}}$ the
restrictions of the forms $H_\omega$ and $B_\omega$
to $V^{1,0}\subseteq H^{1,0}(S)$  and by $H_\omega^\R\vert_V$ and $B_\omega^\R\vert_V$ the
restrictions of the forms $H_\omega^\R$ and $B_\omega^\R$ to $V
\subseteq H^1(S,\R)$ respectively.

\begin{Lemma}
\label{lm:equivalence:Hodge:and:symp:orthogonal:under:star:invariance}
A subspace $V\subset H^1(S,\R)$ is invariant under the
Hodge star-operator if and only if the subspace $V^\perp$,
Hodge-orthogonal to $V$, coincides with the subspace $V^\dagger$,
symplectic-orthogonal to $V$.  In that case the subspace $V^\perp=
V^\dagger$ is  Hodge star-invariant and
$\left(V^{1,0}\right)^\perp=\left(V^\perp\right)^{1,0}$.
\end{Lemma}
\begin{proof}
Assume  that  a subspace $V\subset H^1(S,\R)$ is Hodge star-invariant.
Let $V^\dagger$ be the subspace symplectic-orthogonal to $V$.
By~\eqref{eq:c1:scalar:c2} for any $c_1\in V^\dagger$ and
$c_2\in V$ one has $( c_1, c_2) = \langle c_1, \ast
c_2\rangle$. Since $V$ is Hodge star-invariant, we have $\ast
c_2\in V$ and hence the right expression is equal to zero.
It follows that $V^\dagger \subset V^\perp$. The converse inclusion
is proved similarly. In fact, for any $c_1\in V^\perp$
and $c_2\in V$ one has $ \langle c_1, c_2\rangle= - (c_1, \ast c_2)$.
Again since $V$ is Hodge star-invariant, we have $\ast c_2\in V$,
hence the right  expression is equal to zero. Thus
$V^\perp \subset V^\dagger$ and equality holds. Conversely, assume
that $V^\perp= V^\dagger$. Let $c_1 \in V$
and let $c_2 \in V^\perp$. By~\eqref{eq:c1:scalar:c2}, one has
$(\ast c_1, c_2) = \langle c_1, c_2\rangle=0$, hence $\ast c_1\in
(V^\perp)^\perp= V$. Thus $V$ is Hodge star-invariant.

\smallskip
Let us show that if $V$ is Hodge star-invariant, then $V^\perp$ is
also Hodge star-invariant. Let $c_2\in
V^\perp$ and take any $c_1\in V$. By the same
equation~\eqref{eq:c1:scalar:c2} one has $\langle c_1, \ast
c_2\rangle=(c_1, c_2)$. Since $c_1$ and $c_2$ belong to
Hodge-orthogonal subspaces the right expression is equal to
zero. Hence $\ast c_2$ is symplectic-orthogonal to any $v$ in
$V$, which implies that $\ast c_2\in V^\perp$.

Finally,
for any $\omega_1\in V^{1,0}$ and any
$\omega_2\in \left({V^\perp}\right)^{1,0}$
let $c_1=[\Re(\omega_1)]\in V$ and
let $c_2=[\Re(\omega_2)]\in V^\perp$.
By formula~\eqref{eq:h:c1:h:c2}
$$
\big(\omega_1,\omega_2\big):=\big(h(c_1),h(c_2)\big) = (c_1, c_2) +
i \langle c_1,  c_2\rangle\,,
$$
which is equal to zero since $V$ and $V^\perp$ are
both symplectic-ortho\-go\-nal and Hodge-orthogonal.
This implies that
$\left(V^{1,0}\right)^\perp=\left(V^\perp\right)^{1,0}$.
\end{proof}

\begin{NNRemark}
Note that the above property is not related to either
the $\SL$-action or the Teichm\"uller flow.
\end{NNRemark}

\begin{Proposition}
\label{pr:B:reducible}
Let  $V\subset H^1(S_0,\R)$ be a Hodge star-invariant subspace
in the fiber of the Hodge bundle over $(S_0,\omega_0)\in\cH_g^{(1)}$ and let $V^\dagger \subset
H^1(S_0,\R)$ denote its symplectic orthogonal. Let
$U=\left]-\epsilon,\epsilon\right[$ be any open interval along the trajectory of the Teichm\"uller flow passing through $\omega_0$. Let us identify the fibers of the Hodge bundle over $U$ by parallel transport with respect to the Gauss--Manin connection. The following properties are equivalent:

\noindent
(i) For any $t\in U$ the subspace $V$ stays Hodge star-invariant at $(S_t,\omega_t)$.

\noindent
(ii) For any $t\in U$ the subspaces $V$ and $V^\dagger$ are $B^\R_{\omega_t}$-orthogonal.

\noindent
An analogous equivalence holds when $U$ is replaced by
a small open ball in $\SL$ containing the identity element, or by
a small open neighborhood of the initial point
$(S_0,\omega_0)$ in its Teichm\"uller disc.
\end{Proposition}
\begin{proof}
Suppose that property (i) is satisfied.
By Lemma~\ref{lm:equivalence:Hodge:and:symp:orthogonal:under:star:invariance}
we have a direct sum decomposition $H^1(S_t,\R)=V \oplus V^\dagger$ where
$V$ and $V^\dagger$ are simultaneously symplectic-orthogonal and Hodge-orthogonal
with respect to the Hodge-inner product $(\cdot, \cdot)_{\omega_t} $ on $H^1(S_t, \R)$.
Since the symplectic structure is preserved by the Gauss--Manin connection, $V^\dagger$ is
constant over $U$ under our identification of the real cohomology spaces  $H^1(S_t,\R)$
given by the connection. Hence for any pair $(v,v^\dagger)\in V\times V^\dagger$,  the
Hodge inner products $(v,v^\dagger)_{\omega_t}  = (\ast v,v^\dagger)_{\omega_t}   =0$
for all $t\in U$, so that by Lemma \ref{lemma:varfor2}
\begin{equation}
\label{eq:Hodgeorthog}
\begin{aligned}
\frac{d}{dt} (v,v^\dagger)_{\omega_t}   &= -2 \Re B^\R_{\omega_t} (v, v^\dagger) =0
\quad \text{ \rm and } \\
\frac{d}{dt} (\ast v,v^\dagger)_{\omega_t}   &= -2 \Re B^\R_{\omega_t} (\ast v, v^\dagger)  =
2 \Im  B^\R_{\omega_t} (v, v^\dagger) =0 \,.
\end{aligned}
\end{equation}
It follows that $V$ and $V^\perp$ are $B^\R_{\omega_t}$-orthogonal for all $t\in U$.

\smallskip
Conversely, suppose that property (ii) is satisfied. Let $V^\dagger$ be the symplectic
orthogonal of $V$. Since the Gauss-Manin connection preserves the symplectic structure, the
space $V^\dagger$ is constant over $U$. In addition, since $V$ is Hodge star-invariant,
the symplectic-orthogonal $V^\dagger$ and the Hodge-orthogonal $V^\perp$ coincide
at $t=0$.  It follows from formulas \eqref{eq:Hodgeorthog} that since $V$ and $V^\dagger$ are
$B^\R_{\omega_t}$-orthogonal for all $t\in U$ and they are Hodge-orthogonal for $t=0$,
then they are Hodge-orthogonal for all $t\in U$. Thus the symplectic orthogonal and the Hodge-orthogonal subspaces of $V$ coincide, hence by Lemma~\ref{lm:equivalence:Hodge:and:symp:orthogonal:under:star:invariance}
 the space $V$ is Hodge star-invariant,  for all $t\in U$.
%Again, by Lemma~\ref{lm:equivalence:Hodge:and:symp:orthogonal:under:star:invariance}
%any $v\in V_{t_0}^{1,0}$ transported to $H^1(S_{t_1},\C)$ by Gauss--Manin connection
%(we assume $t_0, t_1\in U$) stays in $V_\C=V_t^{1,0}\oplus V_t^{0,1}$, and hence the
%projection of the resulting vector to $H^{0,1}(S_t)$ belongs to $V_t^{0,1}$. Similarly, any
%$v\in {V^\perp_{t_0}}^{1,0}$ transported to $H^1(S_{t_1},\C)$ by Gauss--Manin connection
%stays in $V_\C^\perp={(V_t^\perp)}^{1,0}\oplus{(V_t^\perp)}^{0,1}$, and hence the projection
%of the resulting vector to $H^{0,1}(S_t)$ belongs to ${V^\perp_t}^{0,1}$. Infinitesimally, this implies
%that for any orthonormal basis $\{\omega_1,\dots,\omega_g\}$ in $H^{1,0}(S_t)$, where $t\in U$,
%such that $\{\omega_1,\dots,\omega_n\}$ is a basis in $V_t^{1,0}$ and $\{\omega_{n+1},\dots,
%\omega_g\}$ is a basis in ${V^\perp_t}^{1,0}$ the second fundamental form is block-diagonal. This
%means that $B_{\omega_t}$ is block-diagonal in such basis, which is equivalent to saying
%that $V$ and $V^\perp$ are $B^\R_{\omega_t}$-orthogonal.
%To prove the reverse implication, we use the Hodge star-invariance of $V$ at the starting point to
%construct decompositions in direct products of the corresponding spaces at $\omega_0$, and then
 %we reverse the arguments ``integrating'' them form $0$ to $t$.
\end{proof}

In the setting of Proposition~\ref{pr:B:reducible} let
$$
\Lambda^V_1(\omega_t) \geq  \dots \geq \Lambda^V_n(\omega_t)\ge 0 \,
$$
be the eigenvalues of the positive-semidefinite Hermitian form
$H_{\omega_t}\big\vert_{V_t^{1,0}}$ restricted to $V_t^{1,0}$,
where $n=\dim_\C V_t^{1,0}$.

\begin{Corollary}
In the setting of Proposition~\ref{pr:B:reducible}
the following sets with multiplicities coincide:
$$
\{\Lambda_1(\omega),\dots,\Lambda_g(\omega)\}=
\{\Lambda^V_1(\omega),\dots,\Lambda^V_n(\omega)\} \sqcup
\{\Lambda^{V^\perp}_1(\omega),\dots,\Lambda^{V^\perp}_{g-n}(\omega)\}
$$
\end{Corollary}
\begin{proof} Let $\{\omega_1,\dots,\omega_g\}$ be an orthonormal basis such that $\{\omega_1,\dots,\omega_n\}$ spans $V$, $\{\omega_{n+1},\dots,\omega_g\}$ spans $V^{\perp}$ and $B_\omega$ has block-diagonal matrix in the basis $\{\omega_1,\dots,\omega_g\}$.
By formula~\eqref{eq:Hform} the matrix $H_{\omega_t}$ is also
block-diagonal in the basis $\omega_1,\dots,\omega_g$.
%constructed in the proof of Proposition~\ref{pr:B:reducible}.
Hence,
$$
 H_{\omega_t} = H_{\omega_t}\big\vert_{V_t^{1,0}} + H_{\omega_t}\big\vert_{{(V^\perp_t)}^{1,0}}\,,
$$
which is exactly the statement of the Corollary.
\end{proof}

\begin{NNRemark}
Suppose that at some point $(S,\omega)$ of the moduli space of
normalized Abelian differentials $\cH_g^{(1)}$ all the eigenvalues
$\Lambda_1(\omega),\dots,\Lambda_g(\omega)$ are distinct. The
corollary above implies that there is only a finite number of
subspaces (namely $2^g$) which might a priori serve as fibers
of Hodge star-invariant subbundles, namely, those spanned by
$\Re(\omega_{i_j}), \Im(\omega_{i_j})$ for some subcollection
$\{\omega_{i_1},\dots,\omega_{i_k}\}$ of eigenvectors
$\{\omega_1,\dots,\omega_g\}$ of $H_\omega$.

The condition that $B_\omega$ is block-diagonal in the
corresponding basis in a small neighborhood $U$ of the initial
point is a necessary and sufficient condition for extension of
the corresponding subspace to a local Hodge star-invariant
subbundle over $U$.
\end{NNRemark}

Let  $V\subset H^1_\R$ be an $\SL$-invariant and Hodge
star-invariant subbundle of dimension $2n$ over a full measure
set for an $\SL$-invariant ergodic probability measure $\mu$ on
the moduli space $\cH_g^{(1)}$ of normalized Abelian differentials. The Hodge
star-invariance is a very strong property of an $\SL$-invariant
subbundle. In particular, the restriction of the
Kontsevich--Zorich cocycle to $V$ mimics most of the properties
of the cocycle on Hodge bundle $H^1_\R$, where $n$ plays a role of
a ``virtual genus''. Let us give several illustrations of this
general philosophy.

We have seen that the symplectic structure on $H^1_\R$ restricts
to a nondegenerate symplectic structure on $V$, hence the
Kontsevich--Zorich cocycle on $V$ is symplectic. It follows
that for any $\SL$-invariant Borel probability ergodic measure
$\mu$ on $\cH_g^{(1)}$ the Kontsevich--Zorich spectrum on $V$ is
symmetric:
$$
\lambda^{V,\mu}_1 \geq  \dots \geq \lambda^\mu_n \geq -\lambda^{V,\mu}_n \geq \dots
 \geq -\lambda^{V,\mu}_1\,.
$$

We get the following generalization of the Kontsevich formula for the
sum  of  positive  Lyapunov  exponents (compare analogous formulas
in~\cite{Kontsevich},   \cite{Eskin:Kontsevich:Zorich}  and,  for  the
specific case of Teichm\"uller curves, in~\cite{Bouw:Moeller}).

\begin{corollary}
\label{cor:KZformgen}
Let $\mu$ be any $\SL$-invariant ergodic Borel probability measure on
the  moduli  space  $\cH_g^{(1)}$.  The  following  formula  holds  for the
Kontsevich--Zorich  exponents  of  any  subbundle $ V\subset H^1_\R$ (of
dimension  $2n$),  $\mu$-almost  everywhere $\SL$-invariant and Hodge
star-invariant:
\begin{equation}
\label{eq:KZformgen}
\lambda^{V,\mu}_1 + \dots +\lambda^{V,\mu}_n=
\int_{\cH_g^{(1)}} (\Lambda^V_1+ \dots + \Lambda^V_n)\, d\mu \,.
\end{equation}
\end{corollary}
\begin{proof}
Let us consider a maximal isotropic subspace $I_n$ of $V_\omega$, and
some Hodge-orthonormal basis $\{c_1, \dots, c_n\}$ of $I_n$. Let
$\{\omega_1,\dots,\omega_n\}$ be Abelian differentials in
$H^{1,0}(S)$ such that $c_j=[\Re(\omega_j)]$ for $j=1,\dots,n$.
Since $V$ is Hodge star-invariant, we get $\omega_j\in V^{1,0}$
for $j=1,\dots,n$. Since $\{c_1, \dots, c_n\}$ are
symplectic-orthogonal and Hodge-orthonormal, the collection
$\{\omega_1,\dots,\omega_n\}$ is orthonormal,
see~\eqref{eq:h:c1:h:c2}. Complete the latter collection
of Abelian differentials to an
orthonormal basis $\{\omega_1,\dots,\omega_g\}$ in $H^{1,0}(S)$.
By construction $\{\omega_{n+1},\dots,\omega_g\}$ is an
orthonormal basis in $(V^\perp)^{1,0}$.  Finally, let
$c_j=[\Re(\omega_j)]$ for $j=n+1,\dots,g$. We have constructed a
Hodge-orthonormal basis $\{c_1, \dots, c_g\}$ of a
Lagrangian subspace in $H^1(S,\R)$ which completes the initial
Hodge-orthonormal basis in the isotropic subspace $I_n\subset V$.
By formula~\eqref{eq:Hii:in:terms:of:B} we have
$$
\sum_{i=1}^n H^\R_\omega(c_i,c_i)\,:=\,
\sum_{i=1}^n H_\omega(\omega_i,\omega_i)
\,=\,
\sum_{i=1}^n\sum_{j=1}^g |B_\omega(\omega_i,\omega_j)|^2
$$
By Proposition~\ref{pr:B:reducible} the matrix $B_\omega$ is block-diagonal
in the chosen basis, so we obtain the following relations:
\begin{multline*}
\Lambda^V_1+\dots+\Lambda^V_n=
\Tr \left(H_\omega\big|_{V^{1,0}}\right)=
\sum_{i=1}^n H_\omega(\omega_i,\omega_i)
\,=\\=\,
\sum_{i=1}^n\sum_{j=1}^g |B_\omega(\omega_i,\omega_j)|^2
\,=\,
\sum_{i,j=1}^n |B_\omega(\omega_i,\omega_j)|^2
\end{multline*}
Plugging the latter formula in definition~\eqref{eq:Phi:alternative}
we finally obtain the following expression for $\Phi_n(\omega, I_n)$:
\begin{equation}
\Phi_n(\omega, I_n) = \sum_{i=1}^n \Lambda^V_i(\omega) \,,
\quad \text{for $\mu$-almost all }\,  \omega\in
\cH_g^{(1)}\,.
\end{equation}
Since the function $\Phi_n$ has no dependence on the maximal
isotropic subspace $I_n \subset V$ and the subbundle $V$ is
$\SL$-invariant, the statement follows from the variational
formula given in Lemma~\ref{lemma:var2}
  %
  % as in
  %
and mimics
the      proof      of      the      Kontsevich      formula     (see
Corollary~\ref{cor:KZformula} above and~\cite{Forni2}, Corollary 5.3).
Alternatively, the statement can be now immediately obtained from the
more  general  formula~\eqref{eq:partial:sum:equals:Pi:k2} on partial
sums    of   exponents   in   Corollary~\ref{cor:partialsum2}   above
(see~\cite{Forni2}, Corollary~\mbox{5.5} for the proof).
\end{proof}

%-----------------------------------------------------------------
\section{Degenerate Kontsevich--Zorich spectrum}
\label{compldegspectrum}

In this section we collect several results which address the occurrence of
zero Kontsevich--Zorich exponents. In particular, we prove that all
the exponents are zero on a given $\SL$-invariant subbundle if and only if the
cocycle is isometric and that happens whenever the second fundamental
form vanishes on that subbundle. In all known
examples the vanishing of the second fundamental form can be derived
from symmetries (automorphisms) of (almost) all Abelian differentials
in the support of an $\SL$-invariant measure. We conclude with a
partial converse which gives a lower bound on the number of strictly
positive exponents in terms of the rank of the second fundamental form.

\subsection{Isometric   subbundles}   By  the  variational  formulas,
whenever  the  second  fundamental  form  vanishes identically on any
flow-invariant  subbundle  then  the  Kontsevich--Zorich cocycle acts
isometrically,  hence  all  of its exponents are zero. We prove below
partial  converse  results  in  the  special case of $\SL$-invariant,
Hodge star-invariant subbundles.

\begin{lemma}
\label{lemma:degeneratespectrum}
Let $V \subset  H^1_\R$ be a flow-invariant subbundle over a full
measure set for a flow-invariant ergodic Borel probability measure $\mu$
on the moduli space $\cH_g^{(1)}$ of normalized Abelian differentials.
Consider the following two properties:
\begin{enumerate}
\item the bilinear form $B_\omega^\R\vert_V$ vanishes
for $\mu$-almost all $\omega \in \cH_g^{(1)}$;
\item the restriction of the Kontsevich--Zorich  cocycle  to $V$ is
isometric with respect to the Hodge norm.
 \end{enumerate}
 Then, one has that $(1)$ implies $(2)$. Moreover, if one also assumes that $V$ is Hodge star-invariant, then $(2)$ implies $(1)$.
\end{lemma}
\begin{proof}
By Lemma~\ref{lemma:varfor2}, the real part of the bilinear form
$B_\omega^\R$ gives the derivative of the Hodge inner product
under the action of the Kontsevich--Zorich cocycle at any
$\omega\in \cH_g^{(1)}$.  Whenever $B_\omega^\R\vert_V$ vanishes for
$\mu$-almost all $\omega\in \cH_g^{(1)}$, by continuity it vanishes
identically on the support  of the measure, and, hence, it
follows from the variational
formula~\eqref{eq:variational:formulas} that the
Kontsevich--Zorich cocycle acts isometrically on $V$ with
respect to the Hodge inner product.

Let us show now that $(2)$ implies $(1)$ (assuming also that $V$ is Hodge star-invariant). By the variational formula~\eqref{eq:varfor2} the
real part of the  bilinear form $B$ vanishes on $V^{1,0}$, i.e.  $\Re B(\alpha, \beta)=0$ for all
$\alpha, \beta \in V^{1,0}$. Note that, since $B_\omega$ is complex bilinear,
$$
B_\omega(e^{i\phi}\alpha,e^{i\phi}\beta)=
e^{2i\phi}B_\omega(\alpha,\beta)\,.
$$
Since the Hodge star-invariance of $V$ implies that $V^{1,0}$ is a complex space, we see that if  $ \Re B(\alpha, \beta)=0$ for all $\alpha, \beta \in V^{1,0}$, then
$B_\omega(\alpha,\alpha)=0$.
\end{proof}

\begin{remark}
Under either condition $(1)$ or condition $(2)$ of Lemma~\ref{lemma:degeneratespectrum} above,
all Lyapunov exponents of the restriction of the Kontsevich--Zorich  cocycle to $V$ are equal to zero.
In fact, any isometric cocycle has a Lyapunov spectrum reduced to the single exponent zero.
\end{remark}

Under the extra assumption that the invariant subbundle $V$ and
the ergodic measure $\mu$ are invariant not only with respect
to the Teichm\"uller flow, but with respect to the action of
$\SL$, one can prove a converse statement and prove
that vanishing of all Lyapunov exponents of an invariant
subbundle implies vanishing of the second fundamental
form on this subbundle, see Theorem~\ref{th:Ann}
below.

\smallskip
In the particular case that $V$ is Hodge star-invariant, the converse
result   becomes  a  straightforward  corollary  of  the  generalized
Kontsevich     formula     (see    formula~\eqref{eq:KZformgen}    in
Corollary~\ref{cor:KZformgen}).

\begin{corollary}
\label{cr:rank:0}
Let $V \subset  H^1_\R$ be an $\SL$-invariant, Hodge star-invariant subbundle over a full
measure set for an $\SL$-invariant ergodic Borel probability measure $\mu$
on the moduli space $\cH_g^{(1)}$ of Abelian differentials.
The following properties are equivalent:
\begin{enumerate}
\item the bilinear form $B_\omega^\R\vert_V$ vanishes
for $\mu$-almost all $\omega \in \cH_g^{(1)}$;
\item the restriction of the Kontsevich--Zorich  cocycle  to $V$ is
isometric with respect to the Hodge norm;
\item the non-negative Lyapunov spectrum of $V$ has the form
$$
\lambda^{V,\mu}_1 =  \dots = \lambda^{V,\mu}_n =0 \,.
$$
\end{enumerate}
\end{corollary}
\begin{proof} The first two statements are equivalent by Lemma~\ref{lemma:degeneratespectrum}.
The second statement implies the third by the definition of Lyapunov exponents. All the above
statements hold for any flow-invariant subbundle. If $V$ is $\SL$-invariant and Hodge star-invariant,
the third statement implies the first statement by the generalized Kontsevich formula. In fact,
by that formula (see Corollary~\ref{cor:KZformgen}) the vanishing of all Lyapunov
exponents of $V$ implies that
$$
\Lambda^V_1(\omega) =\dots  =\Lambda^V_n(\omega)=0 \,, \quad \text{ for $\mu$-almost all }\,
\omega\in \cH_g^{(1)}\,.
$$
It follows that $H_\omega\vert_{V^{1,0}}$ vanishes on the support of the measure $\mu$ in $\cH_g^{(1)}$.
By formula~\eqref{eq:Hform} this implies that the bilinear forms $B_\omega\vert_{V^{1,0}}$ and, thus,
$B_\omega^\R\vert_V$ also vanish for $\mu$-almost all $\omega \in \cH_g^{(1)}$.
\end{proof}

An important particular case of the above Corollary~\ref{cr:rank:0} is given
below. In Appendix~\ref{s:rk:B:cyclic:covers} we shall see other
examples.

\smallskip
The Hodge bundle $H^1_\R$ over $\cH_g^{(1)}$ splits into a direct sum of two subbundles. The first one
(the \textit{tautological subbundle} ) has dimension two; its fiber
are spanned by cohomology classes $[\Re(\omega)]$ and
$[\Im(\omega)]\in H^1(S, \R)$. The second subbundle,
$W$, is the orthogonal complement to the first one with respect
to the symplectic intersection form (and with respect to the Hodge
inner product) on the Hodge bundle $H^1_\R$.
Clearly, both the tautological subbundle and its
orthogonal complement $W$ are $\SL$-invariant and  Hodge star-invariant.
In particular, for all $\omega \in \cH_g^{(1)}$ the space $W^{1,0}_\omega$ is  the  orthogonal
complement to $\omega$ in $H^{1,0}(S)$ with respect to the Hermitian
form~\eqref{eq:Intform}.

By Corollary~\ref{cr:rank:0} we have the following result (see Corollary 7.1  in~\cite{ForniSurvey}):
\begin{corollary}
\label{cr:rank:1}
Let $\mu$ be an $\SL$-invariant ergodic Borel probability measure
on the moduli space $\cH_g^{(1)}$ of normalized Abelian differentials.
The second fundamental form $B_\omega$ has rank equal to $1$ for $\mu$-almost all
$\omega \in \cH_g^{(1)}$ if and only if all the non-trivial Lyapunov
exponents of the Kontsevich--Zorich cocycle with respect to
$\mu$ vanish, that is, if and only if
$$
\lambda_2^\mu = \dots = \lambda_g^\mu =0\,.
$$
\end{corollary}
\begin{proof} It can be verified explicitly that the Lyapunov spectrum of the tautological bundle
is $\{1, -1\}$. It follows that the non-negative Lyapunov spectrum of its symplectic
orthogonal complement  $W\subset H^1_\R$ is $\{\lambda_2^\mu, \dots, \lambda_g^\mu\}$.
Since by definition $B_\omega(\omega, \omega)=1$, the rank of $B_\omega$ is $1$ if and
only if the rank of $B_\omega\vert_{W^{1,0}}$ is zero, for all $\omega\in \cH_g^{(1)}$. The statement
then follows from Corollary~\ref{cr:rank:0} for the $\SL$-invariant, Hodge star-invariant bundle
$W\subset H^1_\R$.

\end{proof}

 %------------------------------------------------------------------
\subsection{A symmetry criterion}
\label{symmetry}
We recall below a simple symmetry criterion for the vanishing of the second
fundamental form found in \cite{ForniSurvey}, \S 7.  Let $\omega\in \cH_g^{(1)}$
be an Abelian differential on a Riemann surface $S$.
Suppose that $S$ has a holomorphic automorphism $T$ and
that the holomorphic 1-form $\omega$ is an eigenvector of the
induced action $T^\ast: H^{1,0}(S)\to H^{1,0}(S)$. Denote by
$u(T)$ the corresponding eigenvalue, $T^\ast\omega=u(T)\omega$.
Note that the action $T^\ast: H^{1,0}(S)\to H^{1,0}(S)$
preserves the restriction of the Hermitian intersection form
(\ref{eq:Intform}) which is positive-definite on $H^{1,0}(S)$,
which implies that $\vert u(T)\vert=1$ and
$T^\ast\vert_{H^{1,0}(S)}$ is diagonalizable. Consider a basis
$\{\omega_1,\dots,\omega_g\}$ of eigenvectors of $T^\ast$ in
$H^{1,0}(S)$ and denote the corresponding eigenvalues by
$u_1(T),\dots,u_g(T)$. The following statement is a simplified
version of Lemma 7.2 in~\cite{ForniSurvey}.

\begin{Theorem}
\label{thm:degenerateB}
Let $\cM$ be an $\SL$-invariant suborbifold in some stratum of Abelian
differentials in genus $g$. Let $\cM$ be endowed with an ergodic
probability measure. Suppose that almost every flat surface
$(S,\omega)$ in $\cM$ is endowed with a holomorphic automorphism
$T:S\to S$, and that $\omega$ is an eigenvector of $T^\ast$ with an
eigenvalue $u(T)$. Denote by $u_1(T),\dots, u_g(T)$ all eigenvalues of
$\ T^\ast:H^{1,0}(S)\to H^{1,0}(S)$.

If for all but one couple of indices $(i,j)$, where $1\le i\le j\le g$,
one has $u_i(T)u_j(T)\neq u^2(T)$, then the rank of the bilinear
form $B_\omega$ on $H^{1,0}$ is equal to $1$ for all $\omega \in \cM$,
and, hence, all the non-trivial Lyapunov exponents of the Hodge bundle
with respect to the Teichm\"uller geodesic flow on $\cM$
vanish:
$$
\lambda_2=\dots=\lambda_g=0\ .
$$
\end{Theorem}
\begin{proof}
Consider  a holomorphic automorphism $T:S\to S$. For any two
holomorphic differentials  $\omega_i,\omega_j$ of our basis of
eigenvectors of the linear map $T^\ast$ in $H^{1,0}(S)$, by
definition~\eqref{eq:B} of the form $B_\omega$
  % in terms of the $L^2$ inner product (\ref{eq:L2})
and by change of coordinates, we get
\begin{equation*}
\begin{aligned}
B_\omega(\omega_i,\omega_j)&=   \frac{i}{2} \int_S
\frac{\omega_i \omega_j }{\omega}  \bar \omega=  \frac{i}{2} \int_S
\frac{T^\ast\omega_i T^\ast\omega_j }{T^\ast\omega} T^\ast \bar \omega=  \\
&=  \frac{i}{2} \int_S \frac{ u_i(T) u_j(T)}{u^2(T) } \frac{\omega_i \omega_j }{\omega}  \bar \omega=
 \frac{ u_i(T) u_j(T)}{u^2(T) } B_\omega(\omega_i, \omega_j)\,.
\end{aligned}
\end{equation*}
Hence, for every pair of indices $i,j$ such that $u_i(T)u_j(T)\neq u^2(T)$,
the element $B_\omega(\omega_i,\omega_j)$ of the matrix of the form
$B_\omega$ is equal to zero. The result
now follows from Corollary~\ref{cr:rank:1}.
\end{proof}

In Appendix \ref{s:rk:B:cyclic:covers} (more precisely, Subsection~\ref{ss:FMt} below), we will give an application of this symmetry criterion (for the vanishing of exponents) provided by Theorem~\ref{thm:degenerateB} in a particular interesting case (of an arithmetic Teichm\"uller disk of a square-tiled cyclic cover in genus $4$).

\subsection{The central Oseledets subbundle}
\label{ssec:central}

An important example of a Hodge star-invariant subspace is given by the kernel of the
second fundamental form. In fact, the following elementary result holds.
For any $(S,\omega)\in \cH_g$, let  $\Ann(B_\omega^\R)$  denote
 the kernel of the form $B^\R_\omega$ on $H^1(S, \R)$, that is:
$$
\Ann(B_\omega^\R):=\{ c\in H^1(S,\R)\ |\
B_\omega^\R(c,c')=0\quad \forall c'\in H^1(S,\R)\}\ .
$$

\begin{Lemma}
\label{lemma:Ann:Hodgeinvariance}
For any $(S,\omega)\in\cH_g$ the subspace $\Ann(B_\omega^\R)$
in $H^1(S,\R)$ is Hodge star-invariant. Moreover,
${(\Ann(B_\omega^\R))}^{1,0}=\Ann B_\omega$.
\end{Lemma}

\begin{proof}
Let $c_1,c_2\in H^1(S,\R)$ be any two cohomology classes, and
let $\omega_1,\omega_2$ be holomorphic $1$-forms such that
$c_1=[\Re(\omega_1)]$ and $c_2=[\Re(\omega_2)]$. By definition
$\ast c_1=[\Im(\omega_1)]$, so we have $\ast
c_1=[\Re(-i\omega_1)]$. Thus, by the definition of
$B^\R_\omega$ given in
Section~\ref{ss:eval:sec:fund:form} and by bilinearity of
the form $B_\omega$ defined in~\eqref{eq:B} one gets
$$
B^\R_\omega(\ast c_1,c_2):=B_\omega(-i\omega_1,\omega_2)
=-i B_\omega(\omega_1,\omega_2)
=-i B^\R_\omega(c_1,c_2)\,.
$$
Hence, if for some $c_1\in H^1(S,\R)$ one has the identity
$B_\omega^\R(c_1,c_2)=0$ for all $c_2\in H^1(S,\R)$, one also
has $B_\omega^\R(\ast c_1,c_2)=0$ for all $c_2\in H^1(S,\R)$.
The last statement is a direct corollary of the definition
of $B^\R_\omega$.
\end{proof}

Another remarkable, however simple, property of the bundle  $\Ann(B_\omega^\R)$ is
described below:

\begin{lemma}
\label{lemma:invcompl}
For any  flow-invariant (resp., $\SL$-invariant) measurable subbundle $V \subset \Ann(B_\omega^\R)$ the Hodge orthogonal splitting $H^1_\R= V \oplus V^\perp$ is flow-invariant (resp.,
$\SL$-invariant).
\end{lemma}
\begin{proof} By  the variational formula (\ref{eq:varfor2}) of Lemma \ref{lemma:varfor2} for the Hodge
innner product, the condition  $V \subset \Ann(B_\omega^\R)$ implies that the Hodge product $(v,w)$
is constant for any parallel (locally constant) sections $v\in V$ and $w\in H^1_\R$.  In particular, the
equation $(v,w)=0$ is invariant under parallel transport.
\end{proof}

We prove below our strongest result on the central Oseledets bundle.

\begin{Theorem}
\label{th:Ann}
Let $\mu$ be a flow-invariant ergodic probability measure on the moduli space $\cH_g^{(1)}$ of normalized Abelian differentials.

If $V \subset H^1_\R$ is a flow-invariant subbundle of the Hodge bundle such that  $\Ann(B^\R\vert V)$ is $\mu$-almost everywhere flow-invariant, then the dimension
$\dim\Ann(B^\R\vert V_\omega)$ is the same for $\mu$-almost all $\omega\in \cH_g^{(1)}$ (so
$\Ann(B^\R\vert V)$ defines a vector bundle over the support of $\mu$)
and
$$\Ann(B^\R\vert V ) \subseteq E^\mu_{(0)} \cap V \quad (\mu\text{-almost everywhere})\,.$$

If $\mu$ is $\SL$-invariant and $W$ is any $\SL$-invariant measurable subbundle of the Oseledets (measurable) bundle
$E^\mu_{(0)}$, then
$$W \subseteq \Ann(B^\R)\quad (\mu\text{-almost everywhere})\,.$$
In particular, if $\mu$ is $\SL$-invariant and $E^\mu_{(0)}\cap V$ is $\mu$-almost everywhere $\SL$-invariant, then
$E^\mu_{(0)}\cap V \subseteq \Ann(B^\R)\quad (\mu\text{-almost everywhere})\,.$

Consequently, if $\mu$ is $\SL$-invariant and $V\subset H^1_\R$ is a Hodge star-invariant and $\SL$-invariant subbundle of the Hodge bundle such that
$\Ann(B^\R\vert V)$ is $\mu$-almost everywhere
flow-invariant and $E^\mu_{(0)}\cap V$ is
$\mu$-almost everywhere $\SL$-invariant, then
$$E^\mu_{(0)}\cap V= \Ann(B^\R\vert V) =  \Ann(B^\R)\cap V \quad (\mu\text{-almost everywhere})\,.$$
\end{Theorem}
\begin{proof}
Since $\dim\Ann(B^\R\vert V)$ is, by hypothesis, a flow-invariant integer-valued
function on $\cH_g^{(1)}$, the ergodicity of $\mu$ implies that it is constant
$\mu$-almost everywhere. (The dimension can jump and become
larger, say, on suborbifolds of nontrivial codimension). Thus,
$\Ann(B^\R\vert V)$ defines a flow-invariant vector bundle over the support of $\mu$ and we can apply
Lemma~\ref{lemma:degeneratespectrum} to this bundle to prove the first
statement.

The last statement is a trivial combination of the first two,
so the essential part of the Theorem is the second statement,
which is proved below.

Let  ${\mathcal  N}^W_k(\mu)$ be  the space of all probability
measures  on  the  Grassmannian  bundle  $G_k (W)$ of $k$-dimensional
isotropic  subspaces projecting on $\cH_g^{(1)}$ to some $\bar{\mu}$ absolutely continuous with respect to $\mu$.
Let us denote  by $2p$ the rank of the restriction of the symplectic form to
the subbundle $W$  and  by $l$ the difference $l=\dim    W-p$.   By
Lemma~\ref{lemma:bundle_meas},  for  any  $k\in \N$ satisfying the relations
$1\le k\le l$, the  space  ${\mathcal  N}^W_k(\mu)$ is non-empty.
Let  ${\mathcal  O}^W_k(\mu)  \subset  {\mathcal  N}^W_k(\mu)$ be the
subset  of  all  $\SO$-invariant  measures.
  As we have already seen in section~\ref{ssec:formulas},
since   $\SO$  is  a  compact  amenable  group,  the  set  ${\mathcal
O}^W_k(\mu)$  is  non-empty  whenever  ${\mathcal  N}^W_k(\mu)  $  is
non-empty:  the  $\SO$-average  of  any  measure  in  ${\mathcal
N}^W_k(\mu) $ is a measure in ${\mathcal O}^W_k(\mu)$.

Since by assumption $W\subset E_{(0)}^\mu$, all the Lyapunov exponents
of the restriction of Kontsevich--Zorich cocycle to $W$ are zero, hence
by the Oseledets theorem the average Lyapunov exponent $\Lambda^{(k)}(\nu)$
of any probability measure $\nu\in\mathcal N^W_k(\mu)$ is equal to zero.
By Theorem~\ref{thm:partialsum} it follows that, for any measure $\nu
\in  {\mathcal  O}^W_k(\mu)$  and  for  any  weak limit $\hat \nu \in I(\nu)$,
\begin{equation}
\label{eq:integral3}
 0 = \int_{G_k(W)}   \Phi_k d\,\hat \nu\,.
\end{equation}
Since  $\Phi_k$  is  by  definition  non-negative,  the above formula
implies that $\Phi_k$ vanishes $\hat\nu$-almost everywhere.
Hence, for $\hat\nu$-almost all $(\omega,  I_k)\in G_k(W)$
 we  get  the following conclusion. By applying the
 identity in the middle of formula~\eqref{eq:Pi_k:various:formulae}
 to any Lagrangian  Hodge-orthonormal completion $\{c_1,\dots,c_k, c_{k+1}, \dots, c_g\}$
of  a  Hodge-orthonormal  basis $\{c_1, \dots, c_k\}$ of the $k$-dimensional
isotropic subspace $I_k\subset W_\omega$ we  get
$$
0=\Phi_k (\omega, I_k) =
\sum_{i=1}^k\sum_{j=1}^g |B_\omega(\omega_i,\omega_j)|^2+
\sum_{i=1}^k\sum_{j=k+1}^g |B_\omega(\omega_i,\omega_j)|^2
\,.
$$
By the definition~\eqref{eq:B:R} of  the  form  $B^\R$, it follows that $B^\R_\omega(c_i, c_j) =0$
for all $i\in  \{1,  \dots,  k\}  \text{  \rm and } j\in\{1, \dots, g\}$, or, equivalently,   that
 $I_k  \subset  \text{  \rm  Ann}(B^\R_\omega)$.  Thus  we have proved that, for any
 Borel probability measure  $\nu \in {\mathcal O}^W_k(\mu)$ and for any weak limit
 $\hat \nu \in I(\nu)$,
 $$
I_k \subset  \text{ \rm Ann} (B^\R_\omega) \,, \quad \text{ for $\hat \nu$-almost all } \, (\omega, I_k)
\in G_k(W)\,.
$$
For  every Borel probability measure $\nu \in {\mathcal O}^W_k(\mu)$,
we then define, over  the  support  of  the  measure $\mu$ on $\cH_g^{(1)}$,
a flow-invariant Borel measurable subbundle $F(\nu) \subset W$ as follows.
For $\mu$-almost all Abelian differentials $\omega \in \cH_g^{(1)}$ we let
the fiber $F_\omega(\nu)$ be the linear span of all isotropic subspaces
$I_k\subset W_\omega$ such that $(\omega,  I_k)$ belongs to the
essential  support  of  at least one measure $\hat \nu \in I(\nu)$ on
the  Grassmannian $G_k(W)$. Since $I(\nu)$ is a compact set of probability measures, and, by Rokhlin's disintegration theorem \cite{Rokhlin},
for each $\hat{\nu}\in I(\nu)$, the conditional measures $\hat{\nu}_{\omega}$ of $\hat{\nu}$ on the fibers $G_k(H^1_{\R})_{\omega}$ depend measurably
on $\omega\in\cH_g^{(1)}$, one can check that $F_{\omega}(\nu)$ depends measurably on $\omega\in\cH_g^{(1)}$. By  construction, since all measures
$\hat \nu \in I(\nu)$ are flow-invariant, the family of subspaces $F_{\omega}(\nu)$
is defined $\mu$-almost everywhere and  flow-invariant. Since the measure
$\mu$ is  ergodic,  this implies that the dimension $\dim F_{\omega}(\nu)$ is
$\mu$-almost  everywhere constant, hence $F(\nu) \subset W$ is a $\mu$-measurable
flow-invariant subbundle. By  construction,  every such $F(\nu)$ is a subbundle of
$\text{  \rm Ann} (B^\R)$, since at $\mu$-almost all $\omega\in \cH_g^{(1)}$ the
fiber $F_\omega(\nu)$ is spanned by subspaces $I_k$ of $\text{ \rm
Ann} (B^\R_\omega)$.

We  then  define,  for  every $1\leq k\leq l$ one more flow-invariant
measurable  subbundle  $F^\mu_k\subset  W$ over the support of $\mu$ on $\cH_g^{(1)}$.
For  $\mu$-almost  any  $\omega\in \cH_g^{(1)}$, the fiber $(F^\mu_k)_\omega$ is
defined   as   a   linear   span  of  the  family  of  vector  spaces
$\{F_{\omega}(\nu)\}_{\nu\in  {\mathcal  O}^W_k(\mu)}$. Finally, we let
$F^\mu \subset W$ be the flow-invariant measurable subbundle, defined
$\mu$-almost everywhere,   such   that  the  fiber  $F^\mu_\omega$  over
$\mu$-almost  any $\omega\in \cH_g^{(1)}$ is the linear span of vector subspaces
$\{(F^\mu_k)_\omega\}_{1\le  k\le l}$. As above, since  $\mu$ is ergodic,
the bundles $F^\mu_k$ and $F^\mu$ are indeed measurable subbundles.
By  construction,  the  measurable  bundle  $F^\mu$ is a subbundle of
$\text{  \rm  Ann} (B^\R)$, since its fiber at $\mu$-almost any Abelian
differential $\omega\in \cH_g^{(1)}$  is spanned by subspaces of $\text{ \rm Ann}
(B^\R_\omega)$.

%Note  that  the restriction of the Hodge bundle to any $\SO$ orbit in
%$\cH_g$ is trivialized by the Gauss--Manin connection.

 Let us argue by contradiction. Let us assume that the subset $\mathcal P \subset \cH_g^{(1)}$
of all $\omega \in \cH_g^{(1)}$ such that $W_\omega \not \subseteq \text{\rm  Ann}(B^\R_\omega)$
has positive measure (with respect to the $\SL$-invariant ergodic probability measure  $\mu$ on
$\cH_g^{(1)}$). Since the subbundles $\text{\rm Ann}(B^\R)$ and $W$ are $\SO$-invariant, it follows
that the set $\mathcal P$ is $\SO$-invariant.
For $\mu$-almost all $\omega\in \cH_g^{(1)}$, let $\SO F^\mu_\omega \subset  W_\omega$  denote
the  smallest $\SO$-invariant linear subspace which contains the subspace $F^\mu_\omega \subset
W_\omega$. The vector space $\SO F^\mu_\omega$ can be defined  as
the intersection of the (non-empty) family of all $\SO$-invariant subspaces of the vector space
$W_\omega$ which contain $F^\mu_\omega$ or , equivalently, as the span of the union
of the family $\{F^\mu _{\omega'}\vert  \omega' \in \SO \cdot \omega\}$.
By construction $F^\mu_\omega \subset \text{\rm Ann} (B^\R_\omega)$, hence
$\SO F^\mu_\omega \subset \text{\rm  Ann} (B^\R_\omega)$, which implies that
$\SO F^\mu_\omega \not = W_\omega$, for all $\omega\in \mathcal P$.
Note that it  is not restrictive to assume that the dimension of the vector space $\SO F^\mu_\omega$ is constant for all $\omega \in  \mathcal P$. In fact, the positive measure set $\mathcal P$ has a finite partition
$$
 \mathcal P := \bigcup_{d=1}^{\text{dim}(W)-1}  \{ \omega\in \mathcal P \vert
\text{dim}(\SO F^\mu_\omega)=d \}\,,
$$
hence at least one of these sets has positive measure. We can therefore assume that the collection
of vector spaces $\{\SO F^\mu_\omega \vert \omega \in \mathcal P\}$ forms a  proper measurable $\SO$-invariant subbundle $\SO F^\mu$ of the restriction $W\vert  \mathcal P$ of the subbundle $W$ to
$\mathcal P$. By the above construction, it follows that the bundle $(\SO F^\mu)^\perp \cap W\vert {\mathcal P}$ is a non-trivial  $\SO$-invariant subbundle of the restriction  of the subbundle $(F^\mu)^\perp \cap W$ to $\mathcal P$.

We claim that there exists $k_0\in \{1, \dots, l \}$  such that for all $k\leq k_0$ there exists an $\SO$-invariant probability measure  $\nu^\ast_{\mathcal P}\in {\mathcal O}^W_k(\mu)$  essentially supported on  the Grassmannian $G_k((F^\mu)^\perp \cap W)$.  In fact, let $\mu_{\mathcal  P}$ be the restriction of the  measure $\mu$ on $\cH_g^{(1)}$  to the positive measure set $\mathcal P$, normalized to have unit total mass. There exists $k_0 \in \N\setminus\{0\}$ such that the Grassmannian of $k_0$-dimensional isotropic subspaces of the (non-trivial) subbundle $(\SO F^\mu)^\perp \cap W_{\mathcal  P}$ is non-empty. By Lemma~\ref{lemma:bundle_meas},  for any $k\leq k_0$ there
exists a probability measure   $\nu_{\mathcal P}$ on $G_k(H^1)$ essentially supported on
$G_k((\SO F^\mu)^\perp \cap W)$, which projects onto the probability measure
$\mu_{\mathcal  P}$ on $\cH_g^{(1)}$ under the canonical projection. Let $\nu^\ast_{\mathcal P}$ be the
$\SO$-average of the measure $\nu_{\mathcal P}$. By construction  $\nu^\ast_{\mathcal P}$
belongs to the set ${\mathcal O}^W_k(\mu)$: in fact, it is $\SO$-invariant, it projects onto the probability measure $\mu_{\mathcal P}$, which is absolutely continuous with respect to $\mu$ on $\cH_g^{(1)}$,
and it is essentially supported on $G_k((F^\mu)^\perp \cap W)$. The above claim is therefore proved.

Since $F^\mu \subset \text{Ann}(B^\R)$ and by construction it is flow-invariant, by
Lemma \ref{lemma:invcompl} there is a flow-invariant  Hodge-orthogonal splitting
$$
H^1_\R =  F^\mu \oplus   (F^\mu)^\perp\,.
$$
It follows that, on the one hand, $F(\nu^\ast_{\mathcal P}) \subset (F^\mu)^\perp$ by construction,
since $\nu^\ast_{\mathcal P}$ is essentially supported on $G_k((F^\mu)^\perp \cap W)$ and the
bundle $(F^\mu)^\perp \cap W$ is flow invariant;
on the other hand, $F(\nu^\ast_{\mathcal P}) \subset F^\mu$ by the definition of the bundle $F^\mu$ given above. Of course, this is a contradiction since $F^\mu \cap   (F^\mu)^\perp=\{0\}$. The argument
is therefore complete.
\end{proof}

\begin{Remark}
Concerning the invariance assumptions (under Teichm\"uller flow and/or $\SL$) in the previous theorem, we would like to stress out that they are really necessary. More precisely, while in Appendix~\ref{s:rk:B:cyclic:covers} below we introduce a class of ergodic $\SL$-invariant probability measures supported on the $\SL$-orbit of \textit{square-tiled cyclic covers} such that we can show $E^{\mu}_{(0)}=\textrm{Ann}(B^\R)$ is $\SL$-invariant  (see Theorem~\ref{thm:cyclic:covers} below), we construct in
Appendix~\ref{s:rk:B:Z} below an example of a closed $\SL$-invariant locus $\mathcal{Z}$ supporting an ergodic $\SL$-invariant probability measure such that $E_{(0)}^{\mu}$ is not $\SL$-invariant , $\textrm{Ann}(B^\R)$ is not Teichm\"uller invariant and hence $E_{(0)}^{\mu}\neq\textrm{Ann}(B^{\R})$ \textit{even though} they have the same dimension! In other words, although $\textrm{Ann}(B^\R)$ doesn't coincide with $E_{(0)}^{\mu}$ in the case of the locus $\mathcal{Z}$, the number of zero exponents is still predicted by the corank of $B^{\R}$. Partly motivated by the features of these examples, we pose below the problem to establish whether the corank of $B^{\R}$ always give the correct number of central exponents (see Problem~\ref{pb:rank} below).
\end{Remark}

Note that invariance of a Hodge star-invariant vector subbundle
$V\subset H^1_\R$ of the Hodge bundle under $\SL$ does not imply
invariance of the corresponding subbundle $V^{1,0}\subset
H^{1,0}$ under $\SL$ or under the flow. For example, the
tautological bundle, spanned by $[\Re\omega]$ and $[\Im\omega]$
is $\SL$-invariant , while the line bundle $\C\cdot\omega\subset
H^{1,0}$ is not. The Lemma below shows that $\Ann(B^\R)$ is special
in this sense.

\begin{Proposition}
Assume that the kernel $\Ann(B^\R)$ of the bilinear form
$B^\R$ on $H^1_\R$ is invariant over an open interval $U=\left]-\epsilon,\epsilon\right[$ along
the trajectory of the Teichm\"uller flow passing through $(S_0,\omega_0)\in \cH_g^{(1)}$. Then the kernel
$\Ann(B)$ of the form $B$ on $H^{1,0}$ is also invariant over $U$.

An analogous statement holds when $U$ is replaced by
a small open ball in $\SL$ containing the identity element, or by
a small open neighborhood of the initial point
$(S_0,\omega_0)$ in the Teichm\"uller disc of $(S_0,\omega_0)$.
\end{Proposition}
\begin{proof}
Since by Lemma~\ref{lemma:Ann:Hodgeinvariance} the subspace $V:=\Ann(B_\omega^\R)$ is
Hodge star-invariant, the subspace $V_\C$ is invariant over the interval $U$ and admits a (a priori
non-invariant) decomposition $V_\C=V_t^{1,0}\oplus V_t^{0,1}$, for all $t\in U$. By Lemma~\ref{lemma:Ann:Hodgeinvariance} $V_t^{1,0}=\Ann(B_{\omega_t})$, hence by the definition (\ref{eq:sec:fund:form}) of the second fundamental form the projection of the covariant derivative $D_{H^1} (\omega)$ onto $V_t^{0,1}$ vanishes for every $\omega \in V^{1,0}_t $ and for all $t\in U$. It follows that the subspace $V^{1,0}_t$ is constant over $t\in U$, as stated.
\end{proof}

  %   \begin{Problem}
  %      %
  %   Suppose that for some finite $\SL$-invariant ergodic measure
  %   the kernel $\Ann(B^\R)$ of the bilinear form $B^\R$ is
  %   $\SL$-invariant  on a set of full measure $\mu$. Does that imply
  %   that the kernel $\Ann(B)$ of the form $B$ on $H^{1,0}$ is also
  %   $\SL$-invariant ?
  %      %
  %   \end{Problem}

%-----------------------------------------------------------------------
\subsection{Non-vanishing of the Kontsevich--Zorich exponents}

In this subsection we prove a general lower bound on the number of
strictly positive Lyapunov exponents of the Kontsevich--Zorich
cocycle in terms of the rank of $B$. Such a bound is considerably weaker than a
sharp estimate (roughly by a factor $2$) and no upper bound other than
Lemma~\ref{lemma:degeneratespectrum} is known.

  \begin{theorem}
\label{thm:lowerbound}
Let $\mu$ be any $\SL$-invariant ergodic probability measure on
the moduli space of Abelian differentials.  Let $V\subset H^1_\R$ be any
 Hodge star-invariant, $\SL$-invariant subbundle
(of dimension $n\in \N$)  defined $\mu$-almost everywhere.
If  for some $k\in\{1, \dots, n-1\}$ the bottom $n-k$ exponents of the restriction of the
Kontsevich--Zorich cocycle to $V$ with respect to the measure $\mu$
vanish, that is,
\begin{equation}
\label{eq:partialvanishing1}
\lambda^{V,\mu}_{k+1}= \dots = \lambda^{V,\mu}_n  =0\,,
\end{equation}
then the rank of the bilinear form $B\vert V^{1,0}_\omega$
satisfies the inequality:
$$
\rk(B\vert V^{1,0}_\omega) \leq 2k \,, \quad \text{ for }
\mu\text{-almost all }\omega\in \cH_g^{(1)}\,.
$$
\end{theorem}
\begin{proof}
The proof follows closely the argument given in \cite{Forni2},
Corollary 5.4, which was stated for the Hodge bundle with respect to
the canonical absolutely continuous invariant measures only. In that case,
the bilinear form $B$ has maximal rank (equal to $g$) on a subset of
positive (Lebesgue) measure on each connected component of
every stratum, as proved in \cite{Forni2}, \S 4, which implies
that
$$
\lambda_1^\mu >\lambda^\mu_2 \geq \dots  \geq \lambda^\mu_{[(g+1)/2]} >0\,.
$$
The generalized argument proceeds as follows.  If all Lyapunov
exponents vanish, then the $\rk(B\vert V^{1,0}_\omega)$ also vanishes
by Corollary~\ref{cr:rank:0}, hence the statement holds in this case.
It follows that without loss of generality we may assume that $\lambda^{V,\mu}_k>0$:
if this assumption only holds for a smaller value of $k\in \{1, \dots, n-1\}$, we would
prove a statement even stronger than the claim.

Let us assume then that $\lambda^{V,\mu}_k>\lambda^{V,\mu}_{k+1}=0$. We are in the
setting of Corollary~\ref{cor:partialsum1} and we can apply formula~\eqref{eq:partial:sum:equals:Pi:k1}.
For $\mu$-almost all $\omega\in \cH_g^{(1)}$, let $\{c_1, \dots, c_k\} \subset V^+(\omega)$ be a Hodge
orthonormal basis.  Let $\{c_1, \dots, c_g\} \subset H^1(S, \R)$ be a completion of the system
$\{c_1, \dots, c_k\}$ to a orthonormal basis of any Lagrangian subspace of $H^1(S, \R)$.
By formula~\eqref{eq:partial:sum:equals:Pi:k1} and by the definition of the function $\Phi_k$
in formula~\eqref{eq:Phi_k} we have
\begin{multline*}
\lambda^{V,\mu}_1 + \dots +\lambda^{V,\mu}_k =
 \int_{\cH_g^{(1)}}  \Phi_k\left(\omega,V^+_k(\omega)\right) d \mu(\omega)
=\ \\ \ =
\int_{\cH_g^{(1)}}\left(\sum_{i=1}^g \Lambda_i(\omega) - \sum_{i,j=k+1}^g
\vert B^\R_\omega(c_i, c_j) \vert ^2\right)d \mu(\omega) \,.
\end{multline*}
Since $V$ is Hodge star-invariant, by the reducibility of the second fundamental form
(see Proposition~\ref{pr:B:reducible}) it follows that
$$
\lambda^{V,\mu}_1 + \dots +\lambda^{V,\mu}_k=
\int_{\cH_g^{(1)}}\left(\sum_{i=1}^n \Lambda^V_i(\omega) - \sum_{i,j=k+1}^n
\vert B^\R_\omega(c_i, c_j) \vert ^2\right)d \mu(\omega) \,.
$$

By the generalized Kontsevich formula (see Corollary~\ref{cor:KZformgen}) we get
$$
(\lambda^{V,\mu}_1 + \dots +\lambda^{V,\mu}_k)+
(\lambda^{V,\mu}_{k+1}+\dots+\lambda^{V,\mu}_n)=
\int_{\cH_g^{(1)}} \sum_{i=1}^n \Lambda^V_i(\omega)\, d\mu \,.
$$
Since $\lambda^{V,\mu}_{k+1}=\dots=\lambda^{V,\mu}_n=0$ the above two expressions
coincide, which implies that
$$
\int_{\cH_g^{(1)}}\
\sum_{i,j=k+1}^n
\vert B^\R_\omega(c_i, c_j) \vert ^2\,d \mu(\omega)\, =\,0\,.
$$
   %
  %
  %  Taking formula
  %  (\ref{eq:Phi_k}) into account, the formulas given in Corollary
  %  \ref{cor:partialsum2}  for (partial) sums of the Lyapunov
  %  exponents of the Kontsevich--Zorich cocycle imply the following
  %  statement. Let $\mu$ be any $\SL$-invariant measure. For any
  %  $s\in \{1, \dots, g-1\}$, the vanishing of {\it exactly
  %  }$\,g-s\,$ Lyapunov exponents implies  that
This means that $\mu$-almost everywhere we have
$B^\R_\omega(c_i, c_j)=0$ for any pair $(i,j)$ such that $i,j\ge k+1$.
Thus, for $\mu$-almost all $\omega\in \cH_g^{(1)}$
there exists an orthonormal basis
$\{\omega_1, \dots, \omega_n\}$ of the subspace $V^{1,0}_\omega$
of holomorphic differentials on the Riemann surface $S$ such that
$$
B_\omega(\omega_i, \omega_j) =0 \,,
\quad \text{ \rm for all }\, i,j \in \{k+1, \dots, n\}\,.
$$
Hence, the matrix of the form $B\vert V^{1,0}_\omega$ with respect to the
basis $\{\omega_1, \dots, \omega_n\}$ has a $(n-k)\times (n-k)$
zero diagonal block. It follows that $B\vert V^{1,0}_\omega$ has rank at
most $2k$.
  %  In other words, if there are {\it exactly} $s\in
  %  \{1, \dots, g-1\}$ strictly positive exponents, then $B_\omega$
  %  has rank at most $2s$. Thus, if the vanishing identities
  %  (\ref{eq:partialvanishing}) hold, then the number of strictly
  %  positive Lyapunov exponents is at most $k\in \{1, \dots,
  %  g-1\}$, hence the rank of $B_\omega$ is at most  $2k$.
\end{proof}

\begin{corollary} [see \cite{Forni2}, Cor. 5.4]
\label{cor:lowerbound}
Let $\mu$ be any $\SL$-invariant ergodic probability measure on
the moduli space of Abelian differentials. If  for some $k\in
\{1, \dots, g-1\}$ the bottom $g-k$ exponents of the
Kontsevich--Zorich cocycle with respect to the measure $\mu$
vanish, that is,
\begin{equation}
\label{eq:partialvanishing2}
\lambda^\mu_{k+1}= \dots = \lambda^\mu_g  =0\,.
\end{equation}
then the rank of the bilinear form $B_\omega$ on $H^{1,0}$
satisfies the inequality:
$$
\rk(B_\omega) \leq 2k \,, \quad \text{ for }
\mu\text{-almost all }\omega\in \cH_g^{(1)}\,.
$$
\end{corollary}

\begin{remark}
In \cite{ForniSurvey} the first author introduced analytic subvarieties
$\cR_g^{(1)}(k)$ of the moduli space $\cH_g^{(1)}$, defined as follows:
$$
\cR_g^{(1)}(k)  := \{ \omega\in \cH_g^{(1)} \vert  \rk (B_\omega)  \leq k \} \,, \quad
\text{ \rm where } k\in \{1, \dots, g-1\} \,.
$$
In terms of such subvarieties Theorem \ref{thm:lowerbound} can be formulated as follows:
if the support of the $\SL$-invariant measure $\mu$ on $\cH_g^{(1)}$ is {\it not }contained in the
subvariety $\cR_g^{(1)}(2k)$, then formula~\eqref{eq:partialvanishing2} does not hold, that is,
$$
\lambda_1^\mu > \lambda^\mu_2 \geq \dots \geq \lambda^\mu_{k+1} >0\,.
$$
\end{remark}

 \begin{remark} A lower bound for the number of strictly positive exponents which holds
 in general for $\SL$-invariant probability measures supported on \textit{regular} orbifolds
 is proved in \cite{Eskin:Kontsevich:Zorich}. Such a bound cannot be derived from the
 above Corollary~\ref{cor:lowerbound} without assumptions on the minimal rank of the
 fundamental form $B$ on the moduli space $\cH_g^{(1)}$. However, to the authors best knowledge,
 lower bounds on the rank of the second fundamental forms for $g\geq 5$ are not
 available. It could be conjectured that it grows linearly with respect to the genus of the surface.
 \end{remark}

\begin{remark}
The argument given in the proof of Theorem \ref{thm:lowerbound}
is {\it not} sufficient to prove the non-vanishing of {\it all
}the Kontsevich--Zorich exponents. In fact, any improvement on
Theorem \ref{thm:lowerbound} based on the formulas for sums of
Lyapunov exponents, given in Corollary \ref{cor:partialsum2},
seems to require some control \textit{a priori} on the position
of the unstable bundle $E^+_g$ of the cocycle on a set of
positive measure of Abelian differentials. In the case of a
canonical absolutely continuous invariant measure such a set
can be found near the boundary of the moduli space and the full
non-vanishing of the Lyapunov spectrum can thus be proved (see
\cite{Forni2},  \S 4 and \S 8.2). Later A.~Avila and M.~Viana
\cite{Avila:Viana} proved the simplicity of the Lyapunov
spectrum, that is, that all the exponents are non-zero and
distinct. The simplicity of the top exponent, that is, the
strict inequality $\lambda_1^\mu >\lambda^\mu_2$  is much
simpler.  W.~Veech \cite{Veech} proved it for the canonical
absolutely continuous invariant measures on connected
components of strata. This result was generalized by the first
author in \cite{Forni2} to an arbitrary  Teichm\"uller invariant
ergodic probability measure.

Recently, Forni has developed his method from \cite{Forni2}
to give a general criterion for the non-uniform hyperbolicity of the
Kontsevich--Zorich spectrum for a wide class of  $\SL$-invariant
measures \cite{Fornicriterion}. The criterion is based on a topological
condition on completely periodic directional foliations contained in the
support of the measure.
\end{remark}

Closing the considerations of this section, we observe that the following fundamental question is wide open:

\begin{Problem}
\label{pb:rank}
Does there exist any finite $\SL$-invariant ergodic measure such
that the number of strictly positive exponents for the
Kontsevich--Zorich cocycle differs from the maximal rank of the
bilinear form $B_\omega$ at a positive measure set in the space
of Abelian differentials?
\end{Problem}

%By Theorem~\ref{th:Ann} a negative solution to
%Problem~\ref{pb:SLinv} implies a negative solution to
%Problem~\ref{pb:rank}.

%-----------------------------------------------------------------------

\appendix

\section{Lyapunov spectrum of square-tiled cyclic covers}
\label{s:rk:B:cyclic:covers}

Arithmetic Teichm\"uller curves of square-tiled cyclic covers (see e.g. the subsection below and/or \cite{Forni:Matheus:Zorich1} for definitions) provide a basic model for the discussion of relations between the Lyapunov spectrum and the geometry of the Hodge bundle. In particular, we will see that, \textit{in the case of square-tiled cyclic covers}, the annihilator $\textrm{Ann}(B^{\mathbb{R}})$ of the second fundamental form $B^{\mathbb{R}}$ \textit{coincides} with the neutral Oseledets bundle $E^{\mu}_{(0)}$. Therefore, in the context of square-tiled cyclic covers, the rank of the second fundamental form (at the generic point with respect to an $\SL$-invariant measure) equals the number of strictly positive Lyapunov exponents, and the corank of the second fundamental form is always completely explained by the symmetries of the underlying surfaces.

However, we advance that these relations \textit{are not true in general}, as an example in Appendix~\ref{s:rk:B:Z} below will show: actually, the best \textit{general} results available (on the relationship between $\textrm{Ann}(B^{\mathbb{R}})$ and $E^{\mu}_{(0)}$) were already given above in Theorem~\ref{th:Ann}, Theorem~\ref{thm:lowerbound} and Theorem~\ref{thm:degenerateB}.

\subsection{Square-tiled cyclic covers}
\label{ss:Cyclic:covers}
In the sequel, we recall the definition of square-tiled cyclic covers and some of its basic properties. For more details and proofs of the statements below, see e.g. \cite{Forni:Matheus:Zorich1}.

Let $N>1$ be an integer, and consider $(a_1,\dots,a_4)$ be a $4$-tuple of integers satisfying the following conditions:
\begin{equation}
\label{eq:a1:a4}
0<a_i\le N\,;\quad
\gcd(N, a_1,\dots,a_4) =1\,;\quad
\sum\limits_{i=1}^4 a_i\equiv 0 \modN\ .
\end{equation}
Let $z_1,z_2,z_3,z_4\in \C$ be four distinct points.
Conditions~\eqref{eq:a1:a4} imply that, possibly after a desigularization,
a Riemann surface defined by equation
$$
w^N=(z-z_1)^{a_1}(z-z_2)^{a_2}(z-z_3)^{a_3}(z-z_4)^{a_4}
$$
is closed, connected and nonsingular. We denote this Riemann surface
by $M=M_N(a_1,a_2,a_3,a_4)$. %By construction, $M_N(a_1,a_2,a_3,a_4)$
%is a ramified cover over the Riemann sphere
%$\Proj^1(\C)$ branched over the points $z_1, \dots, z_4$. Moreover, for each $z_i$, $i=1,2,3,4$, the surface
%$M_N(a_1,a_2,a_3,a_4)$ has $\gcd(N,a_i)$ ramification points over
%$z_i$, and each ramification point has degree $N/\gcd(N,a_i)$. In particular, by Riemann--Hurwitz formula, the genus $g$ of
%$M_N(a_1,a_2,a_3,a_4)$ is
%satisfies
   %
%\begin{multline*}
%2-2g=2N-
%\sum_{\substack{\text{ramification}\\
%                \text{points}}}
%(\text{degree of ramification}-1)
%\ =\\=\
%2N-\sum_{i=1}^4 \gcd(N,a_i)\cdot\big(N/\gcd(N,a_i)-1\big)=
%\sum_{i=1}^4 \gcd(N,a_i)-2N
%\end{multline*}
   %
%and hence,
   %
%\begin{equation}
%\label{eq:genus}
   %
%g=N+1- \frac{1}{2} \sum\limits_{i=1}^4\textrm{gcd}(a_i,N)\,.
%\end{equation}
Puncturing
the ramification points we obtain a regular $N$-fold cover over
$\Proj^1(\C)\setminus\{z_1,z_2,z_3,z_4\}$. The group of deck
transformations of this cover is the cyclic group $\Z/N\Z$ with a generator
$T:M\to M$ given by
\begin{equation}
\label{eq:T}
T(z,w):=(z,\zeta w) \,,
\end{equation}
where
$\zeta$
is a primitive $N$th root of unity, $\zeta^N=1$.

%Throughout this
%paper by a \textit{cyclic cover} we call a Riemann surface
%$M_N(a_1,\dots,a_4)$, with parameters $N,a_1,\dots,a_4$ satisfying
%relations~\eqref{eq:a1:a4}.

%\begin{Convention}
%To be more precise, if for some $i=1,2,3,4$ one has equality $a_i=N$,
%the cover as above is unbranched over such $z_i$. In this case
%we mark all such ``fake'' ramification points on the cover
%and in all our considerations we treat them in the same way
%as the ``true'' ramification points.
%\end{Convention}

%---------------------------------------------------------------

%\subsection{Square-tiled flat structure on a cyclic cover}
%\label{s:Square:tiled:structure}

Next, we recall that any meromorphic quadratic differential
$q(z)(dz)^2$ with at most simple poles
on a Riemann surface defines a flat metric
$g(z)=|q(z)|$ with conical
singularities at zeroes and poles of $q$.
Consider a meromorphic quadratic differential
\begin{equation}
\label{eq:q:on:CP1}
q_0=\frac{(dz)^2}{(z-z_1)(z-z_2)(z-z_3)(z-z_4)}
\end{equation}
on $\Proj^1(\C)$. %It has simple poles at $z_1,z_2,z_3,z_4$ and no
%other zeroes or poles. The quadratic differential $q_0$ defines a
%flat metric on a sphere obtained by identifying two copies of the
%appropriate
%rectangle
%by their boundary.
For a convenient choice
of parameters $z_1,\dots,z_4$ the
rectangles
become unit squares.
%Metrically, we get a square pillow with four corners corresponding to
%the four poles of $q_0$.
Therefore, given $\MNa$ as above and denoting by
$p:\MNa\to\Proj^1(\C)$ the canonical projection $p(z,w)=z$, we have that the quadratic differential
$q=p^\ast q_0$ on $\MNa$ induces a flat structure naturally tiled by unit squares (by construction).
In other words, we get in this way a \textit{square-tiled surface} (or \textit{origami} or \textit{arithmetic translation surface}) $(M_N(a_1,\dots,a_4),q=p^\ast q_0)$.
%,
%see~\cite{Eskin:Okounkov}, \cite{Zorich:square:tiled} (also called an
%\textit{origami}, \cite{Lochak}, \cite{Schmithusen}; also called an
%\textit{arithmetic translation surface}, see~\cite{Gutkin:Judge}).

%---------------------------------------------------------------
%\subsection{Singularity pattern of a square-tiled surface}
%\label{ss:Singularity:pattern}

%For every $z_i$, where $i=1,2,3,4$, the surface
%$M_N(a_1,a_2,a_3,a_4)$ has $\gcd(N,a_i)$ ramification points over
%$z_i$; each ramification point has degree $N/\gcd(N,a_i)$. The flat
%metric has $4$ conical points $z=z_i$, $i=1,\dots,4$, on the base
%sphere with a cone angle $\pi$ at each conical point. Hence, the
%induced flat metric on $M_N(a_1,a_2,a_3,a_4)$ has $\gcd(N,a_i)$
%conical points over $z_i$; each conical point has cone angle
%$\big(N/\gcd(N,a_i)\big)\pi$.

%If one of the cone angles is an odd multiple of $\pi$, then the flat
%metric has nontrivial holonomy, or, in other words,
%the quadratic differential $q=p^\ast q_0$ \textit{is not} a global
%square of a holomorphic 1-form. Note, however, that the condition
%that all cone angles are even multiples of $\pi$ is necessary, but
%yet not sufficient for triviality of the holonomy of the flat metric.

During the present discussion, we will focus \textit{exclusively} on the orientable case $q=\omega^2$ for an Abelian differential $\omega$. By
%Concerning the singularity pattern of $q$, we have the following result (see
Lemma 2 of \cite{Forni:Matheus:Zorich1}, this amounts to assume that $N$ is \textit{even}, and $a_i$ is \textit{odd}, $i=1,2,3,4$ in what follows. Finally, closing this preliminary subsection on square-tiled cyclic covers, we recall (for later use) the following property:

\begin{Lemma}[Lemma 5 of \cite{Forni:Matheus:Zorich1}]
\label{lm:minus:omega}
In the case $q:=p^\ast q_0 = \omega^2$ for an Abelian differential $\omega$, one has that
the form $\omega$ is anti-invariant
with respect to the action of a generator of the group of deck
transformations,
\begin{equation}
\label{eq:omega:antiinvariant}
T^\ast\omega=-\omega\,.
\end{equation}
\end{Lemma}
   %
%\begin{proof}
%By construction, the quadratic differential $q=p^\ast q_0$ is
%invariant under the action of deck transformations on
%$\MNa$. Hence, when $N$ is even,
%all $a_i$, $i=1,2,3,4$ are odd, and  $q=p^\ast q_0=\omega^2$,
%the holomorphic 1-form $\omega$ is either invariant or antiinvariant
%under the action of a generator of the group of deck
%transformations~\eqref{eq:T}.

%Invariance of $\omega$ under $T^\ast$
%would mean that $\omega$ can be pushed forward to $\Proj^1(\C)$, which
%would imply in turn that $q_0$ is a global square of a holomorphic
%$1$-form. This is not true. Hence, $\omega$ is antiinvariant.
%\end{proof}

%By Riemann--Hurwitz formula, the genus $g$ of
%$M_N(a_1,a_2,a_3,a_4)$ satisfies
   %
%\begin{multline*}
%2-2g=2N-
%\sum_{\substack{\text{ramification}\\
%                \text{points}}}
%(\text{degree of ramification}-1)
%\ =\\=\
%2N-\sum_{i=1}^4 \gcd(N,a_i)\cdot\big(N/\gcd(N,a_i)-1\big)=
%\sum_{i=1}^4 \gcd(N,a_i)-2N
%\end{multline*}
   %
%and hence,
   %
%\begin{equation}
%\label{eq:genus}
   %
%g=N+1- \frac{1}{2} \sum\limits_{i=1}^4\textrm{gcd}(a_i,N)\,.
%\end{equation}

%The same result can be obtained by summing up the degrees of
%zeroes giving $2g-2$ for a holomorphic 1-form
%in~\eqref{eq:singularities:one:form}, and $4g-4$ for a quadratic
%differential in~\eqref{eq:singularities:quadratic:differential}.

\subsection{Maximallly degenerate spectrum in genus four}
\label{ss:FMt}

Before entering into the discussion of the relation between the annihilator $\textrm{Ann}(B^{\mathbb{R}})$ of the second fundamental form and the neutral Oseledets bundle $E^{\mu}_0$ in the context of square-tiled cyclic covers, we recall below one of the \textit{only} two examples of Teichm\"uller curves of square-tiled cyclic covers with maximally degenerate Kontsevich--Zorich spectrum. The first
example to be discovered, in genus three, is the Teichm\"uller curve of the well-known Eierlegende Wollmilchsau (see~\cite{ForniSurvey}, \S 7, and \cite{Herrlich:Schmithuesen}). The second example,
in genus four,  was announced in~\cite{Forni:Matheus} and will be presented below.

These examples were the motivation for a full investigation of the spectrum of square-tiled cyclic covers carried out in \cite{Eskin:Kontsevich:Zorich:cyclic} and,
from a slightly different perspective, in the next subsections of this appendix.
In \cite{Forni:Matheus:Zorich1} we showed that there are no other (maximally degenerate) examples
among square-tiled cyclic covers (see \cite{Moeller}, for a stronger result in the class of all
Veech surfaces and \cite{Aulicino} for progress on the general case).  Both examples were discovered as an application of the symmetry criterion given by Theorem~\ref{thm:degenerateB} to the arithmetic Teichm\"uller curves of the square-tiled cyclic covers $M_4(1,1,1)$ and $M_6(1,1,1,3)$ respectively.

Below, we will follow the presentation in~\cite{Forni:Matheus} and apply the symmetry criterion
to the arithmetic Teichm\"uller curve of the square-tiled cyclic cover $M_6(1,1,1,3)$ to derive that it  is  maximally degenerate.

\begin{corollary}
The Lyapunov spectrum of the Hodge bundle over the geodesic flow
on the Teichm\"uller curve of cyclic covers $M_6(1,1,1,3)$ is
maximally degenerate, $\lambda_2=\lambda_3=\lambda_4=0$.
\end{corollary}
\begin{proof}
By Lemma~\ref{lm:minus:omega}, the generator $T$ of the group of deck
transformations of $M_6(1,1,1,3)$ acts on $\omega$ as
$T^\ast\omega=-\omega$, see~\eqref{eq:omega:antiinvariant}; in
particular $u(T)=-1$, and $u^2(T)=1$. We have an explicit
basis
$$\frac{(z-z_4)^2 dz}{w^5}, \, \frac{(z-z_4)^3 dz}{w^5}, \, \frac{(z-z_4)dz}{w^4} \, \textrm{ and }\, \omega:=\frac{(z-z_4)dz}{w^3}$$
in the space of holomorphic
$1$-forms on $M_6(1,1,1,3)$; the corresponding eigenvalues are
$$
\{u_1(T),u_2(T),u_3(T),u_4(T)\}=
\{\zeta,\zeta,\zeta^2,\zeta^3\}\,\quad\text{where }\zeta=\sqrt[3]{-1}\ .
$$
It is easy to see that for all couples of indices $1\le i\le j\le 4$ except
$(i,j)=(4,4)$ one has $u_i(T) u_j(T)\neq 1=u^2(T)$. By Theorem \ref{thm:degenerateB}, the proof is complete.
\end{proof}

%The above square-tiled cyclic covers are the only known examples of maximally degenerate
%spectrum for Abelian differentials (see \cite{Forni:Matheus:Zorich}, \cite{Moeller}).
In both cases of $M_4(1,1,1,1)$ and $M_6(1,1,1,3)$, the symmetries given by the deck transformations force the rank of the second fundamental
form to be equal to $1$ (its lowest possible value) and there is exactly $1$ strictly positive exponent,
the top exponent of the tautological bundle. In the sequel, we will extend this picture to the \textit{whole} class of square-tiled cyclic covers.

\subsection{Rank of the second fundamental form and positive exponents}

The identity between the rank of the second fundamental form and
the number of strictly positive Kontsevich--Zorich exponents always holds for square tiled-cyclic covers.
In fact, we have the following result.

\begin{Theorem}
\label{th:n:of:exp:equals:rk}
Let $\MNa$ be a square-tiled cyclic cover with $N$
even and all $a_i$ odd. The rank of $B_\omega$ is
constant for all $\omega$ in the corresponding $\SL$-orbit and
it is equal to the number of strictly positive Kontsevich--Zorich exponents
on the Hodge bundle $H^1_\R$.
\end{Theorem}
The result was inspired by~\cite{Eskin:Kontsevich:Zorich:cyclic}, and it can be obtained
as a corollary of the results of that paper. For the sake of completeness,
we present in the remaining part of this appendix a proof of Theorem~\ref{th:n:of:exp:equals:rk}.
Such a proof is based on Theorem~\ref{th:Ann}
and a remarkable property of square-tiled cyclic covers,
namely, the existence of an explicit $\SL$-invariant ,
$B^\R$-orthogonal splitting of the Hodge bundle $H^1_\R$ over a
Teichm\"uller curve of a square-tiled cyclic cover into
subbundles of small dimension ($2$ or $4$).
   %
   % We describe below the above-mentioned splitting.
We start with the description of the splitting, see~\cite{Bouw:Moeller} and~\cite{Eskin:Kontsevich:Zorich:cyclic} for more details.

Consider a generator $T$ of the group of deck transformations
of the cyclic cover, see~\eqref{eq:T}. The induced
linear map
$$
T^\ast: H^1(S, \C)\to H^1(S, \C)
$$
verifies $(T^\ast)^N=\operatorname{Id}$, hence its eigenvalues  are $N$-th roots of unity, that is,
they form a subset of the set $\{1,\zeta,\dots,\zeta^{N-1}\}$, where $\zeta$ is an $N$-th primitive root
of unity. For all $k\in \{1, \dots, N-1\}$, let
$$
V_k:=\operatorname{Ker}(T^\ast-\zeta^k\operatorname{Id})
\subset H^1(S, \C)\ .
$$
Since the deck transformation $T$ commutes with the $\SL$-action, each $V_k$ is $\SL$-invariant subbundle of $H^1_\C$.
%The following observation is fundamental for the sequel.
    %
%\begin{Lemma}
%\label{lemma:SLinv}
%The eigenspaces $V_k\,$ form $\SL$-invariant subbundles of the bundle
%$H^1_\C$ over the Teichm\"uller curve of a square-tiled cyclic cover.
%\end{Lemma}
  %
%\begin{proof}
%The group $\SL$ acts
%through monodromy of the flat Gauss--Manin connection
%on the  complex
%cohomology bundle $H^1_\C$ over the Teichm\"uller curve of a
%square-tiled cyclic cover,
%and the monodromy commutes with the group of deck transformations.
%\end{proof}
  %
%We remark that the map $T^\ast$ has a well-defined restriction
%to the subbundle $H^{1,0} (S, \C)$ determined by the cohomology
%classes of holomorphic $1$-forms. The eigenvalues of $T^\ast
%\vert_{H^{1,0} (S, \C)}$ form a subset of
%$\{\zeta,\dots,\zeta^{N-1}\}$, where $\zeta$ is an $N$th
%primitive root of unity. We excluded the root $\zeta^0=1$ since
%any holomorphic form invariant under deck transformations would
%be a pullback from the Riemann sphere, which does not have
%holomorphic forms. For all $k\in \{1, \dots, N-1\}$, let
   %

On the other hand, $T^\ast$ has a well-defined restriction to $H^{1,0} (S, \C)=H^{1,0}$, and $(T^\ast\vert_{H^{1,0} (S, \C)})^N=\textrm{Id}$, so that we can also define
$$
V^{1,0}_k:=\operatorname{Ker}(T^\ast-\zeta^k\operatorname{Id})
\subset H^{1,0}(S, \C)\ .
$$
   %
%Note that $T^\ast$ does admit a basis of eigenvectors in
%$H^{1,0}(S)$ since it preserves the positive-definite
%Hermitian form~\eqref{eq:Intform}, hence
   %
%$$
%H^{1,0}(S,\C) = \bigoplus_{k=1}^{N-1}  V^{1,0}_k \,.
%$$
   %
In \textit{general}, the subspaces $V^{1,0}_k$ do \textit{not} form
$\SL$-invariant subbundles of the complex cohomology bundle.

Let us consider
the decomposition $$H^1(S,\C) = H^{1,0}(S) \oplus H^{0,1}(S)$$
%of the complex cohomology into a holomorphic and an
%anti-holomorphic part.
Since the operator $T^\ast$ preserves
the subspace $H^{1,0}(S)$ and commutes with complex
conjugation, it follows that
$$
H^1(S,\C) = \bigoplus_{k=1}^{N-1}  V_k =
\bigoplus_{k=1}^{N-1}  \left(V^{1,0}_k \oplus V^{0,1}_k\right) =
\bigoplus_{k=1}^{N-1} \left(V^{1,0}_k \oplus
  \overline{V^{1,0}_{N-k}}\right) \,.
$$
In particular, by defining
$$
W_k:=\begin{cases}
V_k\oplus V_{N-k} &\text{ for }k=1,\dots,\frac{N}{2}-1\\
V_{N/2} &\text{ for } k=\frac{N}{2}
\end{cases}\,.
$$
and
$$
H^1_k(S,\R) := W_k  \cap H^1(S, \R) \, ,
$$
we see that, for any $k< N/2$, one has
\begin{equation}
\label{eq:W:k}
W_k= \left(V^{1,0}_k\oplus V^{1,0}_{N-k}\right) \oplus
\left(\overline{V^{1,0}_k\oplus V^{1,0}_{N-k}}\right)
\end{equation}
and, for $k=N/2$, one gets
\begin{equation}
\label{eq:W:N:over:2}
W_{N/2}= V^{1,0}_{N/2}\oplus \overline{V^{1,0}_{N/2}}\,.
\end{equation}
Thus,
%This implies that for any real cohomology class $c\in H^1_k(S,\R)$,
%the Abelian differential $h(c)$ representing $c$ belongs to $W_k$, and that
each $H^1_k(S,\R)$ is $\SL$-invariant  and Hodge star-invariant.

Concerning the dimensions of these subbundles, we have the following lemma. Denote
\begin{equation}
\label{eq:tk}
t(k):=
\sum_{i=1}^4\left\{\frac{k a_i }{N}\right\}\ , \quad \text{for }k=1,\dots,N-1\, ,
\end{equation}
where $\{x\}$ denotes a fractional part of $x$.
Conditions~\eqref{eq:a1:a4} imply that $t(k)$ is integer, hence,
clearly, $t(k)\in\{1,2,3\}$.
\begin{NNLemma}[\cite{Bouw}, Lemma 4.3]
%Consider a cyclic cover
%$\MNa$.
For any $k\in\{1,\dots,N-1\}$, one has
$$
\dim_{\C}V^{1,0}_k=t(N-k)-1 \in \{0, 1, 2\}\ .
$$
\end{NNLemma}

Define the following two complementary subsets $\cI_0$ and $\cI_1$ of the set
$\{1,\dots,N/2\}$:
$$
\cI_0:=\{k \ |\ 1\le k\le N/2,
\text{ and at least one of }
%\ V^{1,0}_k=\{0\} \, \text{ or } \, V^{1,0}_{N-k}=\{0\}
V^{1,0}_k \text{ or }  V^{1,0}_{N-k}
%\text{ or both \}
\text{ vanish}\}
\, ,
$$
$$
\cI_1:=\{1,\dots,N/2\}\setminus \cI_0\ ,
$$
and consider the subspaces
$$
H_{\cI_0}^1(S, \R):=\bigoplus_{k\in \cI_0} H^1_k(S,\R) \,,\qquad
H_{\cI_1}^1(S,\R):=\bigoplus_{k\in \cI_1} H^1_k(S,\R)\ .
$$
of the real cohomology.
By definition, $$H^1(S,\R)=H_{\cI_0}^1(S,\R) \oplus H_{\cI_1}^1(S,\R)\,.$$
In this language, a consequence of the previous lemma is:
\begin{corollary}
\label{cor:Q}
An integer $k$, such that $1\le k\le N-1$,
belongs to the subset $\cI_1$ if and only if
$$
\dim V^{1,0}_k = \dim V^{1,0}_{N-k} = 1\,.
$$
\end{corollary}
\begin{proof} The definition~\eqref{eq:tk} of the integers $t(k)$ and conditions~\eqref{eq:a1:a4}
on the sum of $a_i$ imply that $t(k)+t(N-k)\in\{2,3,4\}$, or, equivalently,
$$
\dim V_k^{1,0}+\,\dim V_{N-k}^{1,0}\in\{0,1,2\} \,.
$$
It follows that if $k\in \cI_1$ then the integers $\dim V_k^{1,0}$ and $\dim V_k^{1,0}$ are both different
from zero, hence they have to be both equal to $1$, as claimed.
\end{proof}

In any case, for a square-tiled cyclic cover, we can use these subbundles to compute the complex bilinear form $B$
on $H^{1,0}$ as follows.

\begin{Lemma}
\label{lm:B:i:j}
Let $\omega_j$, $\omega_k$ be eigenvectors of the linear map
$T^\ast\vert_{H^{1,0}(S,\C)}$ with eigenvalues $\zeta^j$ and $\zeta^k$
respectively. The following formula holds:
$$
B_\omega(\omega_j,\omega_k)
\begin{cases}
\neq 0&\text{ if }j=N-k\\
=0    &\text{ otherwise}\ .
\end{cases}
$$
\end{Lemma}
\begin{proof}
Symmetry arguments analogous to the ones in the proof of
Theorem~\ref{thm:degenerateB} show that, in our case, if the
eigenvalues $u_j=\zeta^j$ and $u_k=\zeta^k$ are not complex-conjugate, the
value $B_\omega(\omega_j,\omega_k)$ on the corresponding
eigenvectors is equal to zero. Since
$\zeta^k=\overline{\zeta^{N-k}}=1/\zeta^{N-k}$, one has
that $j\neq N-k$ implies $B_\omega(\omega_j,\omega_k)=0$.

Now, consider the action of $T^\ast$ on the Abelian differential
$(\omega_k\cdot\omega_{N-k})/\omega$. By Lemma \ref{lm:minus:omega}, we get
$$
T^\ast \left(\frac{\omega_k\,\omega_{N-k}}{\omega}\right)=
\frac{\zeta^k\omega_k\cdot\zeta^{N-k} \omega_{N-k}}{(-\omega)}=
-\omega_k\,\omega_{N-k}\ ,
$$
that is, $(\omega_k\cdot\omega_{N-k})/\omega$ is $T^\ast$ anti-invariant. In other words, since $\zeta^{N/2}=-1$, we have that $(\omega_k\cdot\omega_{N-k})/\omega\in V^{1,0}(N/2)$. By Corollary \ref{cor:Q}, this implies that $(\omega_k\cdot\omega_{N-k})$ is proportional to $\omega$ with a nonzero constant coefficient
$const$. Thus,
$$
B_\omega(\omega_k,\omega_{N-k}):=
\frac{i}{2}\int_S \frac{\omega_k\, \omega_{N-k}}{\omega}\,\bar\omega=
const\cdot\frac{i}{2}\int_S \omega\bar\omega=const\cdot 1\neq 0
$$
This completes the proof of the lemma.
\end{proof}

%Let
   %
%$$
%W_k:=\begin{cases}
%V_k\oplus V_{N-k} &\text{ for }k=1,\dots,\frac{N}{2}-1\\
%V_{N/2} &\text{ for } k=\frac{N}{2}
%\end{cases}\,.
%$$
   %
%Define
   %
%$$
%H^1_k(S,\R) := W_k  \cap H^1(S, \R) \,.
%$$
Our discussion so far can be summarized by the following lemma:
\begin{Lemma}
\label{lemma:realsplitting}
The real Hodge bundle $H^1_\R$ over an arithmetic Teichm\"uller
curve of a square-tiled cyclic cover splits into a direct sum
\begin{equation}
\label{eq:realsplitting}
H^1(S,\R) = \bigoplus_{k=1}^{N/2} H^1_k(S,\R)\,.
\end{equation}
of\, $\SL$-invariant , $B^\R$-orthogonal, Hodge star-invariant
subbundles.
\end{Lemma}

\begin{Remark}
By Lemma~\ref{lm:equivalence:Hodge:and:symp:orthogonal:under:star:invariance}
and Lemma~\ref{lemma:realsplitting}  the subspaces
$H^1_k(S,\R)$ in the splitting~\eqref{eq:realsplitting}
are symplectic-orthogonal and Hodge-orthogonal.
Of course, it can be also immediately seen directly.
\end{Remark}

%Define the following two complementary subsets $\cI_0$ and $\cI_1$ of the set
%$\{1,\dots,N/2\}$:
   %
%$$
%\cI_0:=\{k \ |\ 1\le k\le N/2,
%\text{ and that at least one of }
%\ V^{1,0}_k=\{0\} \, \text{ or } \, V^{1,0}_{N-k}=\{0\}
%V^{1,0}_k \text{ or }  V^{1,0}_{N-k}
%\text{ or both \}
%\text{ vanish}\}
%\, ,
%$$
   %
%$$
%\cI_1:=\{1,\dots,N/2\}\setminus \cI_0\ .
%$$
   %
%Let us consider the following subspaces of the real cohomology:
  %
%$$
%H_{\cI_0}^1(S, \R):=\bigoplus_{k\in \cI_0} H^1_k(S,\R) \,,\qquad
%H_{\cI_1}^1(S,\R):=\bigoplus_{k\in \cI_1} H^1_k(S,\R)\ .
%$$
   %
%Our definition of the subspaces $H_{\cI_0}^1(S, \R)$ and $H_{\cI_1}^1(S,
%\R)$ imply that $$H^1(S,\R)=H_{\cI_0}^1(S,\R) \oplus H_{\cI_1}^1(S,\R)\,.$$
Given a Teichm\"uller curve $\mathcal{C}$ associated to a square-tiled cyclic cover, we denote $H_{\cI_0}^1 \subset H^1_\R$ the bundle over $\mathcal{C}$ formed by the
subspaces $H_{\cI_0}^1(S, \R)$, and $E_0 \subset H^1_\R$
be the Oseledets bundle corresponding to the zero Lyapunov
exponents (with respect to the unique $\SL$-invariant probability supported on $\mathcal{C}$).
The Lemma below is a consequence of Theorem~\ref{th:Ann} and Lemma~\ref{lm:B:i:j}.
\begin{Lemma}
\label{lemma:inclusion}
The following inclusions hold:
$$
H_{\cI_0}^1 \subseteq \Ann(B^\R) \quad
\text{ and } \quad H_{\cI_0}^1 \subseteq  E_0 \,.
$$
\end{Lemma}
\begin{proof}
%By construction, the bundle $H^1_{\cI_0}$ is $\SL$-invariant  and
%Hodge star-invariant. By Lemma \ref{lm:B:i:j} and by the
%definition of $H^1_{\cI_0}$, the inclusion $H^1_{\cI_0} \subseteq \Ann(B^\R)$
%holds. In fact,
For each $k\in \cI_0$, one has $V^{1,0}_k
\oplus V^{1,0}_{N-k} =V_{N-k}^{1,0} \textrm{ or } V^{1,0}_k $.
Hence  $(B\vert_{V^{1,0}_k \oplus V^{1,0}_{N-k}})=0$ by Lemma
\ref{lm:B:i:j}. It follows that $B^\R\vert_{H^1_{\cI_0}} =0$, that
is, $H^1_{\cI_0} \subseteq \Ann(B^\R)$.  On the other hand, we have that, by construction, $H^1_{\cI_0}$ is $\SL$-invariant  and
Hodge star-invariant. Thus, the inclusion $H_{\cI_0}^1 \subseteq  E_0$ follows from Theorem~\ref{th:Ann} (and the inclusion
$H^1_{\cI_0} \subseteq \Ann(B^\R)$).
\end{proof}

%The completion of the analysis seemingly requires precise
%information on the dimensions of the subspaces $V^{1,0}_k$
%which is provided by the result below. Denote
   %
%\begin{equation}
%\label{eq:tk}
%t(k):=
%\sum_{i=1}^4\left\{\frac{k a_i }{N}\right\}\ , \quad \text{for }k=1,\dots,N-1\, ,
%\end{equation}
   %
%where $\{x\}$ denotes a fractional part of $x$.
%Conditions~\eqref{eq:a1:a4} imply that $t(k)$ is integer, hence,
%clearly, $t(k)\in\{1,2,3\}$.
%
%\begin{NNLemma}[\cite{Bouw}, Lemma 4.3]
%Consider a cyclic cover
%$\MNa$. For any $k\in\{1,\dots,N-1\}$ one has
   %
%$$
%\dim_{\C}V^{1,0}_k=t(N-k)-1 \in \{0, 1, 2\}\ .
%$$
   %
%\end{NNLemma}
%
%\begin{corollary}
%\label{cor:Q}
%An integer $k$, such that $1\le k\le N-1$,
%belongs to the subset $\cI_1$ if and only if
%$$
%\dim V^{1,0}_k = \dim V^{1,0}_{N-k} = 1\,.
%$$
%\end{corollary}
%\begin{proof} The definition~\eqref{eq:tk} of the integers $t(k)$ and conditions~\eqref{eq:a1:a4}
%on the sum of $a_i$ imply that $t(k)+t(N-k)\in\{2,3,4\}$, or, equivalently,
%$$
%\dim V_k^{1,0}+\,\dim V_{N-k}^{1,0}\in\{0,1,2\} \,.
%$$
%It follows that if $k\in \cI_1$ then the integers $\dim V_k^{1,0}$ and $\dim V_k^{1,0}$ are both different
%from zero, hence they have to be both equal to $1$, as claimed.
%\end{proof}
%
In what follows, we want to show that the inclusions in the previous lemma are actually equalities. The first step is to prove the following fact:
\begin{lemma}
\label{cor:HP1}
The following identity holds:
$$
H_{\cI_0}^1 = \Ann(B^\R)\,.
$$
\end{lemma}
\begin{proof}
For any $k\in \cI_1$, take $\omega_k\in V^{1,0}_k-\{0\}$ and $\omega_{N-k}\in V^{1,0}_{N-k}-\{0\}$. It follows from
Corollary~\ref{cor:Q} that such $\omega_k$ and $\omega_{N-k}$ exist, and $V^{1,0}_k=\mathbb{C}\omega_k$, $V^{1,0}_{N-k}=\mathbb{C}\omega_{N-k}$.
By~\eqref{eq:W:k} and~\eqref{eq:W:N:over:2}, we obtain that $\Re\omega_k$, $\Re\omega_{N-k}$, $\Im\omega_k$, $\Im\omega_{N-k}$
is a basis of $H^1_k(S,\R)$ when $k\neq N/2$, and
$\Re\omega_{N/2}$, $\Im\omega_{N/2}$ is a basis of $H^1_{N/2}(S,\R)$.
By Lemma \ref{lm:B:i:j}, we deduce that the $B^\R\vert_{H^1_k(S,\R)}$ is non-degenerate for any $k\in \cI_1$. Since $H^1_{\cI_1}(S,\R)$ is a $B^\R$-orthogonal sum
(see Lemma \ref{lemma:realsplitting}), it
follows that the restriction of $B^\R$ to $H^1_{\cI_1}(S,\R)$
is non-degenerate. Finally, since $H^1(S,\R)$ splits as a
$B^\R$-orthogonal sum of $H^1_{\cI_1}(S,\R)$ and
$H^1_{\cI_0}(S,\R)$ (see again Lemma \ref{lemma:realsplitting}), it follows that $\Ann(B^\R) \subseteq
H_{\cI_0}^1$. Because the converse inclusion was proved as part of
Lemma~\ref{lemma:inclusion}, the proof of this lemma is complete.
\end{proof}

Next, we invoke the following key result:
\begin{Theorem}[Theorem 2.6, item (iii), of \cite{Eskin:Kontsevich:Zorich:cyclic}]
\label{lemma:Qexp}
For every $k\in \cI_1$ the Kontsevich--Zorich cocycle has no zero exponents
on the $\SL$-invariant subbundle $H^1_k\subset H^1_\R$. Moreover, for each $k\in\cI_1-\{N/2\}$, the Lyapunov spectrum of Kontsevich-Zorich cocycle
restricted to $H^1_k$ has the form $\{\lambda_k, \lambda_k,-\lambda_k,-\lambda_k\}$ with $\lambda_k>0$ (i.e., there is only one double positive Lyapunov exponent).
\end{Theorem}
\begin{remark}This theorem could be deduced \textit{directly} from the properties of $B^{\mathbb{R}}\vert_{H^1_k}$ discussed in this paper. However, we prefer to skip the presentation of this proof because, contrary to the arguments in \cite{Eskin:Kontsevich:Zorich:cyclic}, it doesn't yield the \textit{precise} value $\lambda_k = 2\cdot\min\{t_j(k), 1-t_j(k): j=1,2,3,4\}$.
\end{remark}

At this stage, we are ready to get the following result stating in particular that for square-tiled cyclic covers the central Oseledets
subbundle indeed coincides with the kernel of the second fundamental form:
\begin{theorem}
\label{thm:cyclic:covers}
Suppose that $N$ is even
and all $a_i$ are odd. Consider the $\SL$-orbit of the square-tiled cyclic cover $\MNa$.  The following identities hold:
$$
E_0= \Ann(B^\R)=H_{\cI_0}^1 \,.
$$
In particular the number of strictly positive Lyapunov exponents of the Kontsevich--Zorich
cocycle is given by the following formula:
$$
\#\{ k\in \{1, \dots, N-1\}  \vert \dim_{\C}V^{1,0}_k=\dim_{\C}V^{1,0}_{N-k}=1\}\,.
$$
\end{theorem}
\begin{proof}
The decomposition~\eqref{eq:realsplitting} of the Hodge bundle
into a direct sum of $\SL$-invariant subbundles implies that
the Lyapunov spectrum of the Hodge bundle is the union of the
spectra of the subbundles. Lemmas~\ref{lemma:inclusion} and ~\ref{cor:HP1} show that
$\Ann(B^\R)=H_{\cI_0}^1\subseteq E_0$, while
Theorem~\ref{lemma:Qexp}, shows that the inclusion $H_{\cI_0}^1\subseteq E_0$ is, actually
an equality. Finally, the formula counting the number of
nonnegative Lyapunov exponents is obtained by a combination of
Theorem~\ref{lemma:Qexp} with Corollary~\ref{cor:Q}.
\end{proof}

Theorem~\ref{th:n:of:exp:equals:rk} is now an immediate corollary
of the more precise Theorem~\ref{thm:cyclic:covers}.

\section{Lyapunov spectrum of a higher-dimensional $\SL$-invariant locus $\mathcal{Z}$}\label{s:rk:B:Z}
This appendix is devoted to the description of an $\SL$-invariant locus $\mathcal{Z}\subset\mathcal{H}(8,2^5)$ supporting an ergodic $\SL$-invariant probability $\mu$ with the following properties:
\begin{itemize}
\item there are precisely $6$ vanishing exponents out of $10$ non-negative  Kontsevich--Zorich exponents (with respect to $\mu$),
\item the corank of $B_{\omega}$ is $6$ for every $\omega\in\mathcal{Z}$, but
\item $E^{\mu}_{(0)}\neq\textrm{Ann}(B^\R)$, $E_{(0)}^\mu$ is not $\SL$-invariant  and $\textrm{Ann}(B^\R)$ is not flow-invariant.
\end{itemize}
The construction of $\mathcal{Z}$ is partly motivated by McMullen's paper~\cite{McMullen}.

Below, we will only \textit{sketch} the proof of these properties. The details are part of a forthcoming paper \cite{FMZ-III}.

\subsection{Description of the locus $\mathcal{Z}$} Denote by $\mathcal{Z}$ the family of Riemann surfaces
$$C_6=\{w^6=(z-z_1)\dots(z-z_6)\}$$
equipped with the Abelian differentials
$$\omega=\frac{(z-z_1)dz}{w^3}.$$
Here, $z_1,\dots,z_6\in\Proj^1(\C)$ are $6$ pairwise distinct points of the Riemann sphere. A quick inspection reveals that
these Riemann surfaces have genus 10, and any such $\omega$ has a zero of order $8$ at the branch point over $z_1$ and a double zero at the branch point over $z_j$, $j=2,\dots,6$. In other words, $\mathcal{Z}\subset\mathcal{H}(8,2^5)$.

Observe that $\omega^2=g^*q$ where $g:C_6\to\Proj^1(\C)$, $g(z,w)=z$, and $q$ is the following quadratic differential with a simple zero and 5 simple poles of the Riemann sphere:
$$q=\frac{(z-z_1) dz^2}{(z-z_2)\dots(z-z_6)}\in\mathcal{Q}(1,-1^5)$$
Alternatively, $\omega=h^*\widehat{\omega}$ where $h:C_6\to C_2$, $h(z,w)=(z,w^3)$, $C_2=\{y^2=(z-z_1)\dots(z-z_6)\}$ is a genus two Riemann surface and $\widehat{\omega}$ is the following Abelian differential on $C_2$:
$$\widehat{\omega}=\frac{(z-z_1)dz}{y}\in\mathcal{H}(2).$$
Since $\mathcal{Z}$, $\mathcal{Q}(1,-1^5)$ and $\mathcal{H}(2)$ are $4$-dimensional loci, it follows that $\mathcal{Z}\simeq\mathcal{Q}(1,-1^5)\simeq\mathcal{H}(2)$ and $\mathcal{Z}$ is the closure of the $\textrm{GL}^+(2,\mathbb{R})$-orbit of $(C_6,\omega)$ for a generic choice of $z_1,\dots, z_6$.

\subsection{Decomposition of Hodge bundle over $\mathcal{Z}$} Similarly to the discussion of the case of square-tiled cyclic covers in Appendix \ref{s:rk:B:cyclic:covers} above, we notice that the group of deck transformations of the cover $g:C_6\to\Proj^1(\C)$ (ramified over $z_1,\dots,z_6$) is generated by
$$T(z,w)=(z,\varepsilon w)$$
where $\varepsilon$ is a primitive $6$th root of unity. For sake of concreteness, we take $\varepsilon=\exp(2\pi i/6)$. Again, we denote by $T^*:H^1(C_6,\C)\to H^1(C_6,\C)$ the induced linear map. Of course, since $(T^*)^6=\textrm{Id}$, its eigenvalues are a subset of $\{1,\varepsilon,\dots,\varepsilon^5\}$. For every $k=1,\dots,5$, we put
$$V_k = \textrm{Ker}(T^*-\varepsilon^k\textrm{Id}).$$
Again, because $\SL$ acts by monodromy of the flat Gauss-Manin connection, these eigenspaces $V_k$ form $\SL$-invariant subbundles of the Hodge bundle over $\mathcal{Z}$.

Next, we consider the restriction of $T^*$ to $H^{1,0}(C_6,\C)$. Because there is no Abelian differential on $\Proj^1(\C)$, we see that the eigenvalues of $T^*|_{H^{1,0}(C_6,\C)}$ form a subset of $\{\varepsilon^j: j=1,\dots,5\}$ and we denote $$V_k^{1,0}=\textrm{Ker}(T^*|_{H^{1,0}(C_6,\C)}-\varepsilon^k\textrm{Id}),$$
so that
$$H^{1,0}(C_6,C)=\bigoplus\limits_{k=1}^{5}V_k^{1,0}.$$
A quick computation shows that
$$\left\{\frac{z^j dz}{w^k}: 0<j<k<5\right\}$$
is a basis of holomorphic differentials on the genus 10 Riemann surface $C_6$. In particular,
$$\textrm{dim}_{\C}V_k^{1,0}=k-1$$
for each $k=1,\dots,5$.

Finally, we form the subspaces
$$W_k=\left\{\begin{array}{cc}V_k\oplus V_{6-k} & \textrm{if } k\neq 3 \\ V_3 & \textrm{if }k=3\end{array}\right.$$
and
$$H^1_k(C_6,\R):=W_k\cap H^1(C_6,\R).$$

Since $V_k\oplus V_{6-k}$ and $V_3$ are invariant under complex conjugation, each $V_k$ is $\SL$-invariant  and $V_k=V_k^{1,0}\oplus V_k^{0,1} = V_k^{1,0}\oplus \overline{V_{6-k}^{1,0}}$, we have the following Hodge-$\ast$ and $\SL$-invariant  splitting
$$H^1(C_6,\R)=\bigoplus\limits_{k=1}^3 H^1_k(C_6,\R).$$
Moreover, a direct application of the arguments of the proof of Theorem~\ref{thm:degenerateB} shows that the $H^1_k(C_6,\R)$ are pairwise $B^{\R}$-orthogonal.

In the sequel, we will use these decompositions to analyze $\textrm{Ann}(B^{\R})$ and the neutral Oseledets bundle of $\SL$-invariant ergodic probabilities supported on $\mathcal{Z}$.

\subsection{Neutral Oseledets bundle versus $\textrm{Ann}(B^{\R})$ over $\mathcal{Z}$} In the present case, we just saw that $\textrm{dim}_{\C}V_k^{1,0}=k-1$, so that $V_1=\overline{V_5^{1,0}}$, $V_5=V_5^{1,0}$ and, \textit{a fortiori}, $V_5^{1,0}\subset\textrm{Ann}(B)$ and $H_1^1(C_6,\R)\subset\textrm{Ann}(B^{\R})$ (compare with Theorem \ref{thm:degenerateB} and Lemma \ref{lm:B:i:j}). Furthermore, from Theorem~\ref{th:Ann}, it follows also that $H_1^1(C_6,\R)\subset E_{(0)}^{\mu}$ for any ergodic $\SL$-invariant $\mu$ supported on $\mathcal{Z}$, that is, any such measure automatically possesses at least $4$ vanishing exponents among the $10$ non-negative exponents of Kontsevich-Zorich cocycle restricted to the Hodge bundle over $\mathcal{Z}$.

\begin{remark}The fact $V_5^{1,0}=V_5$ says that the Jacobian of $C_6$ has a \textit{fixed part} (of dimension $4=\textrm{dim}_{\C}V_5$), i.e., the complex torus $A=V_5^{1,0}/V_5^{1,0}(\Z)$ obtained from the quotient of $V_5^{1,0}$ by the lattice $V_5^{1,0}(\Z)=V_5^{1,0}\cap H^1(C_6,\Z)$ is a fixed part (rigid factor) of the Jacobian $\textrm{Jac}(C_6)$ of $C_6$ in the sense that we have a isogeny $\textrm{Jac}(C_6)\to J(C_6)\times A$. The fact that $V_5^{1,0}$ is a fixed part of the Jacobian was already known by C. McMullen~\cite[Theorem 8.3]{McMullen} and it was our starting point to study the locus $\mathcal{Z}$.
\end{remark}

Next, we pass to the analysis of $H_3^1(C_6,\R)$. Because $h^*(H^{1,0}(C_2,\C))$ equals $V_3^{1,0}$ (where $h:C_6\to C_2$, $h(z,w)=(z,w^3)$ is the covering map used above to construct the isomorphism $\mathcal{Z}\simeq\mathcal{H}(2)$), by the results of Bainbridge~\cite{Bainbridge} and Eskin, Kontsevich and Zorich~\cite{Eskin:Kontsevich:Zorich}, we conclude that the non-negative exponents of the Kontsevich-Zorich cocycle restricted to $H_3^1(C_6,\R)$ are $1$ and $1/3$ for any ergodic $\SL$-invariant probability supported on $\mathcal{Z}$ and the rank of the second fundamental form $B$ restricted to $V_3^{1,0}$ has full rank (equal to $2$).

Therefore, it remains to study the restriction of the second fundamental form $B$ to $W_2^{1,0}:=V_2^{1,0}\oplus V_4^{1,0}$ and the restriction of the Kontsevich--Zorich cocycle to $H_2^1(C_6,\R)$.

\begin{remark}For sake of concreteness, we observe that $H_2^1(C_6,\R)$ is a copy of the first homology group $H^1(C_3,\R)$ of the genus $4$ Riemann surface
$$C_3:=\{x^3=(z-z_1)\dots(z-z_6)\}=C_6/\langle T^2\rangle$$
Moreover, the square of the Abelian differential $\omega$ on $C_6$ projects into the quadratic differential $q_{(4)}=(z-z_1)^2 dz^2/x^3$. In other words, the study of the restriction of the Kontsevich--Zorich cocycle over $\mathcal{Z}$ to $H_2^1(C_6,\R)$ is equivalent to the study of the Kontsevich--Zorich cocycle over the locus determined by the family of quadratic differentials $(C_3,q_{(4)})\in\mathcal{Q}(7,1^5)$ of genus $4$. In fact, in the forthcoming paper \cite{FMZ-III}, we will adopt the latter point of view (i.e., we will study directly this family of quadratic differentials of genus $4$).
\end{remark}

\begin{lemma}\label{lemma:rk:B:Z} The restriction of the form $B$ to $W_2^{1,0}$ has rank  $2$. In particular, $\textrm{Ann}(B^{\R})\cap H_2^1(C_6,\R)$ is a $4$-dimensional (real) subspace of the $8$-dimensional space $H_2^1(C_6,\R)$.
\end{lemma}

\begin{proof}Let $\alpha\in V_2^{1,0}$ and $\beta_0,\beta_1,\beta_2\in V_4^{1,0}$ be a basis of $W_2^{1,0}$. By the arguments in the proof of Theorem~\ref{thm:degenerateB}, we have that $B_\omega(\alpha,\alpha)=0$ and $B_\omega(\beta_j,\beta_l)=0$ for all $0\leq j,l\leq 3$. Hence, the matrix of $B|_{W_2^{1,0}}$ in the basis $\{\alpha,\beta_0,\beta_1,\beta_2\}$ is
$$\left(
\begin{array}{cccc}
0 & b_0 & b_1 & b_2 \\
b_0 & 0 & 0 & 0 \\
b_1 & 0 & 0 & 0 \\
b_2 & 0 & 0 & 0
\end{array}
\right)$$
where $b_j:=B_{\omega}(\alpha,\beta_j)$, $j=0,1,2$. Hence, it suffices to prove that one of the entries $b_j$ is non-zero to conclude that the rank of $B|_{W_2^{1,0}}$ is $2$. To do so, we make the following choice of basis $\alpha=dz/w^2$, $\beta_j=(z-z_1)^j dz/w^4$, and we compute
\begin{eqnarray*}
B_\omega(\alpha,\beta_2)=\int\frac{\alpha\beta_2}{\omega}\overline{\omega}=\int \frac{|z-z_1|^2}{|w|^6}dz\overline{dz}\neq 0,
\end{eqnarray*}
that is $b_2\neq 0$.
\end{proof}

%\begin{remark}The same argument above shows also that $\beta_0=dz/w^4$ and $\beta_1=(z-z_1)dz/w^4$ is a basis of $\textrm{Ann}(B|_{W_2^{1,0}})$. However, this ``explicit'' basis of $\textrm{Ann}(B)$ doesn't seem useful to study its invariance under Teichm\"uller flow.
%\end{remark}

\begin{lemma}\label{lemma:rk:Ec:Z}The non-negative part of the Lyapunov spectrum of the restriction of the Kontsevich-Zorich cocycle to $H_2^1(C_6,\R)$ with respect to any $\SL$-invariant ergodic probability $\mu$ has the form
$$\{\lambda^{\mu}=\lambda^{\mu}>0=0\}.$$
In particular, the neutral Oseledets bundle $E_0^\mu$ intersects $H_2^1(C_6,\R)$ in a subspace of real dimension $4$.
\end{lemma}

We begin the proof of this lemma by following~\cite{Eskin:Kontsevich:Zorich:cyclic} to see that the Lyapunov exponents of $G_t^{KZ}|_{H^1_2(C_6,\R)}$ have multiplicity $2$ (at least). Indeed, given a vector $v\in H_2^1(C_6,\R)$ corresponding to a Lyapunov exponent $\lambda$, we have that $T^*v$ corresponds to the same Lyapunov exponent (as $T^*$ commutes with the monodromy). On the hand, since the eigenvalues of $T^*|_{H_2^1(C_6,\R)}$ are not real (i.e., $\varepsilon^2,\varepsilon^4\notin\R$), it follows that $T^*v$ is not collinear to $v$, that is, $\lambda$ has multiplicity $2$ at least.

Consequently, the non-negative part of the Lyapunov spectrum of $G_t^{KZ}$ restricted to $H_2^1(C_6,\R)$ (with respect to $\mu$) has the form
$$\{\lambda^{\mu}=\lambda^{\mu}\geq \theta^{\mu}=\theta^{\mu}\}.$$

On the other hand, we know that $B^\R$ restricted to $H_2^1(C_6,\R)$ is not degenerate (see Lemma~\ref{lemma:rk:B:Z}). Hence, from Theorem~\ref{th:Ann}, we conclude that $$\lambda^{\mu}>0.$$
In other words, the proof of the previous lemma is reduced to show that $\theta^{\mu}=0$, that is,
\begin{lemma}\label{lemma:zero:exp}
The non-negative part of the Lyapunov spectrum of $G_t^{KZ}$ restricted to $H_2^1(C_6,\R)$ has two vanishing exponents.
\end{lemma}

In the forthcoming paper \cite{FMZ-III}, we will deduce this lemma along the following lines.  Since $V_4=V_4^{1,0}\oplus \overline{V_2^{1,0}}$, $\textrm{dim}_{\C}(V_4^{1,0})=3$, and $\textrm{dim}_{\C}(V_2^{1,0})=1$, the intersection form~\eqref{eq:Intform} has signature $(3,1)$ and hence the action of $\SL$ through monodromy of the Gauss-Manin connection on the subspace $V_4$ of the complex Hodge bundle $H_{\C}^1$ is represented by $U(3,1)$ matrices. In \cite{FMZ-III}, we will see that, \textit{in general}, a cocycle preserving a pseudo-Hermitian form of signature $(p,q)$ (i.e., with values in the matrix group $U(p,q)$) has $|p-q|$ zero Lyapunov exponents \textit{at least}. By applying this general principle in the context of the previous lemma, we have that the Kontsevich--Zorich cocycle has (at least) $3-1=2$ zero Lyapunov exponents, so that the lemma follows.
%By the precedent corollary, this forces the presence of $3-1=2$ vanishing exponents in the non-negative part of the Lyapunov spectrum of $G_t^{KZ}|_{H_2^1(C_6,\R)}$.
%\end{proof}

\smallskip
In any event, by Lemmas~\ref{lemma:rk:B:Z} and~\ref{lemma:zero:exp}, we can discuss the main result of this appendix. In the following theorem, we denote by $\mu$ the natural $\SL$-invariant ergodic probability fully supported on $\mathcal{Z}$ (obtained from the so-called \textit{Masur-Veech} probability on $\mathcal{H}(2)$ via the isomorphism $\mathcal{Z}\simeq\mathcal{H}(2)$ previously constructed).
\begin{theorem}\label{thm:AnnB:Ec} We have $\textrm{Ann}(B^{\R})\cap H_2^1(C_6,\R)\neq E_{(0)}^{\mu}\cap H_2^1(C_6,\R)$. Consequently, $\textrm{Ann}(B^{\R})\neq E_{(0)}^{\mu}$, hence
$E_{(0)}^{\mu}$ is not $\SL$-invariant .
\end{theorem}

Actually, this theorem is a consequence of the following fact:

\begin{theorem}\label{thm:matrix:comp} There is no $2$-dimensional Teichm\"uller-flow and $\SO$-invariant (i.e., $\SL$-invariant ) continuous subbundle $V\subset W_2^{1,0}$.
\end{theorem}

%IT REMAINS TO CHECK IF THE PREVIOUS STATEMENT CAN BE SLIGHTLY IMPROVED TO A ZARISKI-LIKE RESULT

Assuming momentarily this theorem, one can conclude Theorem~\ref{thm:AnnB:Ec} as follows. We have that $\textrm{Ann}(B)\cap W_2^{1,0}$ is a continuous (actually, real-analytic) and $\SO$-invariant $2$-dimensional subbundle of $W_2^{1,0}$. From Theorem~\ref{thm:matrix:comp} it follows that
 $\textrm{Ann}(B)\cap W_2^{1,0}$ is not  flow-invariant, and hence $\textrm{Ann}(B^{\R})\cap H_2^1(C_6,\R)$ is not flow-invariant. Since $E_{(0)}^{\mu}\cap H_2^1(C_6,\R)$ is flow-invariant, it follows that $E_{(0)}^{\mu}\cap H_2^1(C_6,\R)$ cannot coincide with
 $\textrm{Ann}(B^{\R})\cap H_2^1(C_6,\R)$. By Theorem~\ref{th:Ann}, if $E_{(0)}^{\mu}$ were
 $\SL$-invariant , it would be a subbundle of the bundle $\textrm{Ann}(B^{\R})$, hence it would coincide with it by dimensional reasons, thereby contradicting the first part of the statement.

\smallskip
Concerning the proof of Theorem \ref{thm:matrix:comp}, let us just say a few words (for more details see \cite{FMZ-III}): the basic idea is that $\SL$-invariance of a continuous subbundle $V$ can be tested along \textit{pseudo-Anosovs} (i.e., periodic orbits of Teichm\"uller flow); indeed, the existence of $\SL$-invariant  continuous subbundles $V$ implies that the monodromy matrices along pseudo-Anosovs passing by the same Riemann surface should share a common subspace, and (the non-validity of) this last property can be tested by direct calculation.

Closing this appendix, let us mention that the \textit{precise} value of the positive Lyapunov exponent coming from $H_2^1(C_6,\R)$ can be computed from the main formula of \cite{Eskin:Kontsevich:Zorich} and a computation with Siegel-Veech constants related to $\mathcal{Q}(1,-1^5)$ (see \cite{FMZ-III}):
\begin{proposition}\label{pr:exp:value:Z}Let $\nu$ be any $\SL$-invariant ergodic probability supported on $\mathcal{Z}$. Then, the non-negative part of the Lyapunov spectrum of the Kontsevich--Zorich  cocycle with respect to $\nu$ has the form
$$\{1>4/9=4/9>1/3>0=0=0=0=0=0\}.$$
\end{proposition}

\subsection*{Acknowledgments}

We would like to heartily thank A.~Avila, A.~Eskin, P.~Hubert, M.~Kontsevich, M.~M\"oller and
J.-C.~Yoccoz for several discussions related to the material of this paper. Also, we are thankful to the anonymous referees for their comments and suggestions leading to the current version of this article.

%J.-C. Yoccoz, who has motivated our work
%by asking whether the example in~\cite{ForniSurvey} could be
%generalized. We are also very grateful to M. M\"oller for several interesting conversations about cyclic covers. The third author also thanks A.~Eskin
%and M.~Kontsevich for the pleasure to work with them on the
%project on Lyapunov
%exponents~\cite{Eskin:Kontsevich:Zorich}; in particular,
%on its part related to evaluation of individual Lyapunov
%exponents of square-tiled cyclic
%covers~\cite{Eskin:Kontsevich:Zorich:cyclic}. Furthermore, we are grateful to Y. Guivarc'h for asking us about the possible isometric behavior of Kontsevich-Zorich cocycle on neutral Oseledets bundles, and
%A. Avila, P. Hubert and J.-C. Yoccoz for their interest in this work (specially for several discussions related to the last section of this article).

The authors thank Coll\`ege de France, Institut des Hautes \'Etudes Scientifiques (IHES), Hausdorff Research Institute for Mathematics (HIM) and Max Planck Institute for Mathematics (MPIM) for their
hospitality during the preparation of this paper.

The first author was supported by the National Science Foundation grant DMS~0800673.

The second author was partially supported by the Balzan Research Project of J. Palis and the French ANR grant ``GeoDyM'' (ANR-11-BS01-0004).

\end{document}